\documentclass[11pt]{article}

\usepackage{jmlr2e}
\usepackage{color}

\usepackage{amsmath, amsthm, bm}
\usepackage{enumitem}
\usepackage{comment}
\usepackage{subcaption}

\newcommand{\tr}{\text{tr}}
\renewcommand{\vec}{\text{vec}}

\newcommand{\asymptotic}{{M+N\rightarrow \infty}}

\DeclareMathOperator*{\argmin}{arg\,min}

\newtheorem{assump}{Assumption}
\newtheorem{thm}{Theorem}
\newtheorem{corr}{Corollary}
\newtheorem{lem}{Lemma}
\newtheorem{eg}{Example}
\newcommand{\rc}[1]{{\color{blue}{#1}}}
\newcommand{\cc}[1]{{\color{red}{#1}}}
\newcounter{remark}

\ShortHeadings{KoPA: Automated Kronecker Product Approximation}{Cai, Chen and Xiao}
\firstpageno{1}

\begin{document}
\title{KoPA: Automated Kronecker Product Approximation}

\author{\name Chencheng Cai \email chencheng.cai@rutgers.edu \\
\addr Department of Statistics \\
Rutgers University\\
Piscataway, NJ 08854, USA
\AND
\name Rong Chen \email rongchen@stat.rutgers.edu\\
\addr Department of Statistics \\
Rutgers University\\
Piscataway, NJ 08854, USA
\AND
\name Han Xiao \email hxiao@stat.rutgers.edu\\
\addr Department of Statistics \\
Rutgers University\\
Piscataway, NJ 08854, USA
}
\editor{}

\maketitle

\begin{abstract}%
  We consider the problem of matrix approximation and denoising induced by the Kronecker
  product decomposition. Specifically, we propose to approximate a given matrix by the sum of a few
  Kronecker products of matrices, which we refer to as the Kronecker product
  approximation (KoPA). Because the Kronecker product is an extensions of the outer product from vectors to matrices, KoPA extends the low rank matrix approximation, and includes it as a special case. Comparing with the latter, KoPA also offers a greater flexibility, since it allows
  the user to choose the configuration, which are the dimensions of the two smaller
  matrices forming the Kronecker product. On
  the other hand, the configuration to be used is usually
  unknown, and needs to be determined from the data in order to achieve the optimal balance between accuracy and parsimony. We propose to use
  extended information criteria to select the
  configuration. Under the paradigm of high dimensional analysis, we show
  that the proposed procedure is able to select the true configuration
  with probability tending to one, under suitable conditions on the
  signal-to-noise ratio. We demonstrate the
  superiority of KoPA over the low rank approximations through
  numerical studies, and several benchmark image examples.
\end{abstract}

\begin{keywords}  Information Criterion, Kronecker Product, Low Rank Approximation, Matrix Decomposition, Random Matrix
\end{keywords}

\section{Introduction}\label{sec:intro}

Observations that are matrix/tensor valued have been commonly seen in
various scientific fields and social studies. 
In recent years, advances in technology have made high dimensional matrix/tensor type
data possible and more and more
prevalent. Examples include high resolution images in face recognition and
motion detection \citep{turk1991face, bruce1986understanding, parkhi2015deep}, brain images through fMRI \citep{belliveau1991functional, maldjian2003automated}, adjacent matrices of
social networks of millions of nodes \citep{goldenberg2010survey}, the covariance matrix of
thousands of stock returns \citep{ng1992multi, fan2011high}, the import/export network among hundreds
of countries \citep{chen2019factor}, etc. 
Due to the high dimensionality
of the data, it is often useful and preferred to store, compress,
represent, or summarize the matrices/tensors through low dimensional
structures. In particular, low rank approximations of matrices have
been ubiquitous. Finding a low rank approximation
of a given matrix is closely related to the singular value
decomposition (SVD), and the connection was revealed as early as \cite{Eckart1936the}. SVD has proven extremely useful in
matrix completion \citep{candes2009exact, candes2010matrix,
  cai2010singular}, community detection \citep{le2016optimization},
image denoising \citep{guo2015efficient}, among many others.

In this paper, we investigate matrix approximations induced by the
Kronecker product. 
Since the Kronecker product is an extension of the outer product, we call the proposed method KoPA (Kronecker outer Product Approximation). Kronecker product is an operation on two matrices
which generalizes the outer product from vectors to
matrices. Specifically, the Kronecker product of a $p\times q$
matrix $\bm A=(a_{ij})$ and a $p'\times q'$ matrix $\bm B=(b_{ij})$,
denoted by $\bm A\otimes \bm B$, is defined as a
$(pp')\times (qq')$ matrix which takes the form of a block
matrix.  In $\bm A\otimes\bm B$, there are $pq$ blocks of size
$p'\times q'$, where the $(i,j)$-th block is the scalar product
$a_{ij}\bm B$. We refer the readers to \cite{horn1991topics} and
\cite{van1993approximation} for overviews of the properties and computations
of the Kronecker product. Kronecker product has also found wide
applications in signal processing, image restoration and quantum
computing, etc. For example, in the statistical modeling of a
multi-input multi-output (MIMO) channel communication system,
\cite{Werner2008On} modeled the covariance matrix of channel signals
as the Kronecker product of the transmit covariance matrix and the
receive covariance matrix. In compressed sensing,
\cite{Duarte2012Kronecker} utilized Kronecker products to provide a
sparse basis for high-dimensional signals. In image restoration,
\cite{Kamm1998Kronecker} considered the blurring operator as a
Kronecker product of two smaller matrices. In quantum computing,
\cite{Kaye2007introduction} represented the joint state of quantum
bits as a Kronecker product of their individual states.

In SVD, a matrix is represented as the sum of rank one matrices, where
each of them is represented as the outer product of the left
singular vector and the corresponding right singular vector (after
the transpose). Similarly, the Kronecker Product
Decomposition (KPD) of a $(pp')\times(qq')$ matrix $\bm C$ is defined as
\begin{equation*}
  \bm C= \sum_{k=1}^{d} \bm A_k\otimes \bm B_k.
\end{equation*}
where $d=\min\{pq,p'q'\}$, and $\bm A_k$ and $\bm B_k$ are
$p\times q$ and $p'\times q'$ respectively. In the definition of
the KPD, the dimensions of $\bm A_k$ and $\bm B_k$ have to be specified, which (in this case, $p\times q$ and $p'\times q'$) we refer to as the {\it configuration} of the KPD. 
Further constraints on $\bm A_k$ and $\bm B_k$ are necessary to make the
decomposition well defined and unique, but we will defer the exact
definition of KPD to Section~\ref{sec:kpm}. Since the Kronecker
product is an extension of the vector outer product, so is
KPD of SVD. In particular, if $p=1$, and $q'=1$, then $\bm A_k$ and $\bm B_k$ are column
and row vectors respectively, and the KPD, under this particular configuration, becomes the SVD.

Similar to  rank-one approximation, the best matrix approximation
given by a Kronecker product is formulated as finding the closest
Kronecker product under the Frobenius norm. This was introduced in the
matrix computation literature as the nearest Kronecker product (NKP)
problem in \cite{van1993approximation}, who also demonstrated its
equivalence to the best rank one approximation and therefore also to
the SVD, after a proper rearrangement of the matrix entries. 
Such an
equivalence is also maintained if one seeks the best approximation of
a given matrix by the sum of $K$ Kronecker products of the same configuration,
 $\sum_{k=1}^K \bm A_k\otimes \bm B_k$.
Despite of its
connection to SVD, finding a best Kronecker approximation also
involves a pre-step: determining the configurations of the Kronecker
products, i.e., determining the dimensions of $\bm A_k$ and $\bm B_k$. One of
our major contributions in this paper is on the selection of the
configuration based on an information criterion.

Although the configuration selection poses new challenges, KPD also
provides a framework that is more flexible than SVD. Here we use the
cameraman's image, a benchmark in image analysis, to illustrate the
potential advantage of KPD over SVD. The left panel in Figure
\ref{fig:cameraman_intro} is the 512$\times$512 pixel image of a cameraman in
gray scale. The middle panel shows the best rank-1 approximation of
the original image given by the leading term of SVD. The rank-1
approximation explains $45.63\%$ of the total variation of the
original image with 1023 parameters. The right panel in Figure
\ref{fig:cameraman_intro} displays the image obtained by the nearest
Kronecker product of configuration
$(16\times 32)\otimes(32\times 16)$. With the same number of parameters as
the rank-1 approximation, this nearest Kronecker product approximation
explains $77.55\%$ of the variance of the original image.

\begin{figure}[!htpb]
    \centering
    \includegraphics[width=0.3\textwidth]{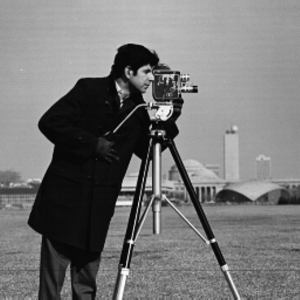}
    \includegraphics[width=0.3\textwidth]{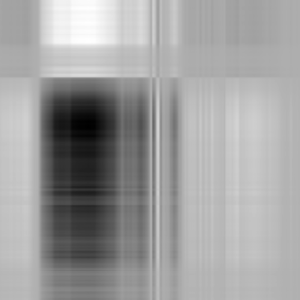}
    \includegraphics[width=0.3\textwidth]{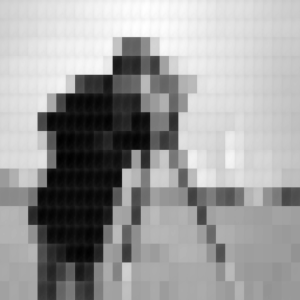}
    \caption{(Left) Original cameraman's image; (Middle) SVD approximation; (Right) KPD approximation}
    \label{fig:cameraman_intro}
\end{figure}

We will revisit the cameraman's image in Section~\ref{sec:numerical} with a more detailed analysis. We notice here that the superiority of KoPA over low rank approximation in representing images is partially due to the similarity of local blocks in the image. In this regard it is related to the patch based de-noising methods  \citep{dabov2007image, chatterjee2011patch} in the field of image processing, which explore the recurrence of similar local pattern throughout the image.
However, we have a substantially distinct focus in this paper. One of our main objectives is to devise a formal procedure to determine the configuration, or the ``patch size”, from the data, which is usually chosen in an {\it ad hoc} manner in patch based methods. We introduce a statistical model to characterize the image generating mechanism, and propose to use information criteria to select the configuration. Practically it implies an emphasis on the balance between the complexity (number of model parameters) and accuracy (closeness to the original image). Furthermore, the KoPA framework and the model selection also has potential applications in high dimensional panel time series, large network analysis, recommending systems, and other matrix-type data analysis.
For example, in modeling dense networks \citep{leskovec2010kronecker}, the adjacency matrix can be represented by a Kronecker product $\bm A\otimes\bm B$, where $\bm A$ and $\bm B$ correspond to the inter- and inner-community structures respectively. As a second example, the KoPA may as well replace the low rank approximation in the synchronization problem \citep{chen2008extended,singer2011angular} to identify the groups/clusters of the individuals, at the same time of denoising the distance matrix. It is worth mentioning that KoPA can also be used to speed up the computation. If the transition matrix of a Markov Chain can be represented as one or a sum of a few Kronecker products, then the state update can be calculated more efficiently \citep{dayar2012analyzing}. KoPA plays its role in guiding the choice of the Kronecker product approximation of the transition matrix.

In this paper, we focus on the model
\begin{equation*}
\bm Y = \lambda\bm A\otimes \bm B + \sigma \bm E,
\end{equation*}
where $\bm E$ is a standard Gaussian ensemble consisting of IID
standard normal entries, $\lambda>0$ and $\sigma>0$ indicates the
strength of signal and noise respectively. We consider the matrix de-noising problem which aims to recover the
Kronecker product $\lambda\bm A\otimes \bm B$ from the noisy
observation $\bm Y$. Here the configuration of the Kronecker product,
i.e. the dimensions of $\bm A$ and $\bm B$, is to be determined from
the data. We propose to use information criteria (which include
AIC and BIC as special cases) to select the configuration, and prove
its consistency under some conditions on the signal-to-noise
ratio. The consistency of the configuration selection is established
for both deterministic and random $\bm A$ and $\bm B$, under the paradigm of high dimensional analysis, where the dimension of $\bm Y$ diverges to infinity.

The rest of the paper is organized as follows. In Section
\ref{sec:kpm}, we give the precise definition of the KPD, and introduce the model, with a review of some of their basic
properties. In Section \ref{sec:estimation-and-config}, we propose
the information criteria for selecting the configuration of the
Kronecker product. We investigate and establish the consistency of
the proposed selection procedure in
Section~\ref{sec:theory}. Extension to the multi-term Kronecker product
models is discussed in Section \ref{sec:extension}.  In Section
\ref{sec:numerical}, we carry out extensive simulations to
assess the performance of our method, and demonstrate its superiority
over the SVD approach. We also present a detailed analysis of the cameraman's
image. 

\textbf{Notations:} Throughout this paper, for a vector $v$, $\|v\|$
denotes its Euclidean norm. And for a matrix $\bm M$,
$\|\bm M\|_F=\sqrt{\tr(\bm M'\bm M)}$ and
$\|\bm M\|_S=\max_{\|u\|=1}\|\bm Mu\|$ denote its Frobenius norm and
spectral norm respectively. For any two real numbers $a$ and $b$,
$a\wedge b$ and $a\vee b$ stand for $\min\{a, b\}$ and $\max\{a, b\}$
respectively. For any number $x$, $x_+$ denotes the positive part
$x\vee 0 = \max\{x, 0\}$.

\section{Kronecker Product Model}\label{sec:kpm}

\subsection{Kronecker Product Decomposition}
\label{sec:kpd}
We first repeat the definition of the Kronecker product of a $p\times q$ matrix $\bm A$ and a $p'\times q'$ matrix $\bm B$, which is given by
$$
    \bm A\otimes \bm B = \begin{bmatrix}
    a_{1,1}\bm B &  a_{1, 2}\bm B & \cdots& a_{1, q}\bm B\\
    a_{2,1}\bm B & a_{2, 2}\bm B & \cdots& a_{2, q}\bm B\\
    \vdots & \vdots &&\vdots\\
    a_{p, 1}\bm B & a_{p, 2}\bm B&\cdots & a_{p,q}\bm B
    \end{bmatrix}.
$$

Let $\bm C$ be a $(pp')\times(qq')$ real matrix, its Kronecker
Product Decomposition (KPD) of configuration $(p,q,p',q')$ is
defined as
\begin{equation}
  \label{eq:kpd}
  \bm C= \sum_{k=1}^{d}  \lambda_k \bm A_k\otimes \bm B_k.
\end{equation}
where $d=\min\{pq,p'q'\}$, each $\bm A_k$ is a $p\times q$ matrix
with Frobenius norm $\|\bm A_k\|_F=1$, each $\bm B_k$ is a
$p'\times q'$ matrix with $\|\bm B_k\|_F=1$, and
$\lambda_1\geqslant\lambda_2\geqslant\cdots\geqslant\lambda_d\geqslant 0$. The matrices
$\bm A_k$ are mutually orthogonal in the sense that
$\hbox{tr}(\bm A_k\bm A_l')=0$ for $1\leqslant k<l\leqslant d$, and so are the
matrices $\bm B_k$. 

The best way to see that the KPD is a valid definition is through its
connection with the SVD, after a proper rearrangement of the
elements of $\bm C$, as demonstrated in \cite{van1993approximation}.
Denote by $\text{vec}(\cdot)$ the vectorization of a matrix by stacking its rows. If $\bm A=(a_{ij})$ is a $p\times q$ matrix, then
$$\text{vec}(\bm A) := [a_{1,1}, \dots, a_{1, q}, \dots, a_{p, 1},\dots, a_{p, q}]'.$$
If $\bm B=(b_{ij})$ is a $p'\times q'$ matrix, then
$\text{vec}(\bm A)[\text{vec}(\bm B)]'$ is a $(pq)\times(p'q')$
matrix containing the same set of elements as the Kronecker product
$\bm A\otimes \bm B$, but in different positions. We
define the rearrangement operator $\mathcal R$ to represent this relationship. Write the matrix
$\bm C$ as a $p\times q$ array of blocks of the same block size
$p'\times q'$, and denote by $C_{i,j}^{p',q'}$ the $(i,j)$-th
block, where $1\leqslant i\leqslant p,\,1\leqslant j\leqslant q$. The operator
$\mathcal R$ maps the matrix $\bm C$ to
\begin{equation}
  \label{eq:rearrange1}
    \mathcal R_{p,q}[\bm C] = \begin{bmatrix}
    \text{vec}(\bm C_{1, 1}^{p',q'}), \dots, \text{vec}(\bm C_{1, q}^{p',q'}),\dots, \text{vec}(\bm C_{p, 1}^{p',q'}), \ldots,\text{vec}(\bm C_{p,q}^{p',q'})
    \end{bmatrix}',
\end{equation}
When applied to a Kronecker product $\bm A\otimes \bm B$, it holds that
\begin{equation}
  \label{eq:rearrange2}
    \mathcal R_{p,q}[\bm A\otimes \bm B] = \vec(\bm A)[\vec (\bm B)]'.
\end{equation}
In view of \eqref{eq:rearrange1} and \eqref{eq:rearrange2}, we see
that the KPD in \eqref{eq:kpd} corresponds to the SVD of the
rearranged matrix $\mathcal R_{p,q}[\bm C]$, and the conditions imposed on
$\bm A_k$ and $\bm B_k$ are derived from the properties of the
singular vectors.

Here, we note that the rearrangement operator $\mathcal R$ is configuration dependent, which we emphasize by explicitly specifying the dimension of $\bm A_k$ (in this case, $p$ and $q$) in the subscript of $\mathcal R$, see \eqref{eq:rearrange1} and \eqref{eq:rearrange2}. When there is no ambiguity, the subscript of $\mathcal R$ may be omitted for notational simplicity. According to the definition, the mapping $\mathcal R_{p, q}: \mathbb R^{pp'\times qq'} \rightarrow \mathbb R^{pq\times p'q'}$ is an isomorphism since it is linear and bijective. In addition, since the order of elements does not change the Frobenius norm, the mapping $\mathcal R$ is also isometric under Frobenius norm. 

\subsection{Kronecker Product Model}
\label{sec:kpms}

We consider the model where the observed $P\times Q$ matrix $\bm Y$ is a noisy
version of an unknown Kronecker product
\begin{equation}
  \bm Y = \lambda\bm A\otimes \bm B + \dfrac{\sigma}{\sqrt{PQ}} \bm E\label{eq:model-normalized}.
\end{equation}
To resolve the obvious unidentifiability regarding $\bm A$ and
$\bm B$, we require
\begin{equation}
  \label{eq:identi}
  \|\bm A\|_F=\|\bm B\|_F=1, 
\end{equation}
so that $\lambda>0$ indicates the strength of the signal part. Note
that under \eqref{eq:identi}, ${\bm A}$ and ${\bm B}$ are identified
up to a sign change. We assume that the noise matrix $\bm E$ has IID
stand normal entries, and consequently the strength of the noise is
controlled by $\sigma>0$. The dimensions of $\bm A$ and $\bm B$
correspond to the integer factorization of the dimension of $\bm
Y$. For convenience, we assume throughout this article that the
dimension of the observed matrix $\bm Y$ in
(\ref{eq:model-normalized}) is $2^M\times 2^N$ with
$M, N\in \mathbb N$. As a result, the dimension of $\bm A$ must be of
the form $2^{m_0}\times 2^{n_0}$, where $0\leqslant m_0\leqslant M$ and $0\leqslant n_0\leqslant N$,
and the corresponding dimension of $\bm B$ is $2^{m^\dagger_0}\times
2^{n_0^\dagger}$, where $m_0^\dagger = M-m_0$ and $n_0^\dagger = N-n_0$. Therefore, we can simply use the pair $(m_0,n_0)$ to denote the
configuration of the Kronecker product in \eqref{eq:model-normalized}.
An implicit advantage of this assumption lies in the fact that if two
configurations $(m,n)$ and $(m',n')$ are different, then the number of
rows of $\bm A$ under one configurations divides the one under the
other, and similarly for the numbers of columns, and for $\bm B$. For
example, if $m\leqslant m'$, then the number of rows of $\bm A$ under the
former configuration, which is $2^m$, divides the number of rows
$2^{m'}$ under the latter. This fact leads to a more elegant treatment of the theoretical analysis in Section~\ref{sec:theory}.

For image analysis, assuming the dimension
to be powers of 2 seems rather reasonable. On the other hand, for
other applications where the dimension of the observed matrix are not
powers of 2, one can transform the matrix to fulfill the
assumption. For example, one can super-sample the matrix to increase
the dimension to the closest powers of 2, or augment the matrix by
padding zeros. The methodology proposed in this paper can be applied to any integer numbers $P$ and $Q$ with more than two factors.

We will consider two mechanisms for the signal part $\lambda\bm A\otimes \bm B$.

\vspace{0.1in}

\noindent {\bf Deterministic Scheme.} We assume that $\lambda$, $\bm A$ and
$\bm B$ are deterministic, satisfying \eqref{eq:identi}. We define the following signal-to-noise ratio to measure the signal strength
\begin{equation*}
  \dfrac{\|\lambda \bm A\otimes \bm B\|_F^2}{\mathbb E\|\sigma\bm E/2^{(M+N)/2}\|_F^2} = \frac{\lambda^2}{\sigma^2}. 
\end{equation*}

\noindent {\bf Random Scheme.} Assume that $\lambda$, $\bm A$ and
$\bm B$ are random and independent with $\bm E$. Although $\bm A$ and $\bm B$ are stochastic, we assume that they have been rescaled so that \eqref{eq:identi} is fulfilled. In this case the signal-to-noise ratio is defined as
\begin{equation*}
  \dfrac{\mathbb E\|\lambda \bm A\otimes \bm B\|_F^2}{\mathbb E\|\sigma\bm E/2^{(M+N)/2}\|_F^2} =\dfrac{\mathbb E\,\lambda^2}{\sigma^2}.
\end{equation*}

\stepcounter{remark}
\noindent {\bf Remark \arabic{remark}.}  
We distinguish between these two schemes to account for the different assumptions on data generating mechanism. In the random scheme, the observed matrix data is assumed to be randomly chosen from a (super-)population of matrices with an ad-hoc prior, which for example can be chosen as the Kronecker product of two independent Gaussian random matrices. Under the random scheme assumption, ill-behaved matrices arise with negligible probabilities under the prior. Similar assumptions have been used in factor analysis and random effects models. The deterministic scheme incorporates arbitrary matrices. Additional assumptions need to be imposed to exclude extreme cases for which the proposed model selection would fail.

\subsection{Estimation with a Known Configuration} 
\label{sec:estimation}

Suppose we want to estimate $\bm A$ and $\bm B$ based on a given configuration $(m,n)$, that is, the dimensions of $\bm A $ and $\bm B$ are 
$2^{m}\times 2^{n}$ and $2^{m^\dagger}\times 2^{n^\dagger}$ respectively. Again we use $m^\dagger = M-m$ and $n^\dagger = N-n$ to ease the notation when $M$ and $N$ are known.  To estimate
$\bm A$ and $\bm B$ in \eqref{eq:model-normalized} from the observed
matrix $\bm Y$, we solve the minimization problem
\begin{equation}
  \label{eq:nkp}
  \min_{\lambda,\bm A,\bm B}\|\bm Y-\lambda\bm A\otimes\bm B\|_F^2, \quad\hbox{subject to }\|\bm A\|_F=\|\bm B\|_F=1.
\end{equation}
Since we have assumed that the noise matrix contains IID standard
normal entries, \eqref{eq:nkp} is also equivalent to the MLE. This
optimization problem has been formulated as the nearest Kronecker
product (NKP) problem in the matrix computation literature
\citep{van1993approximation}, and solved through the SVD after
rearrangement. According to Section~\ref{sec:kpd}, after applying the
rearrangement operator, the cost function in \eqref{eq:nkp} is
equivalent to
\begin{equation*}
  \|\bm Y - \lambda \bm A\otimes \bm B\|_F^2 = \|\mathcal R[\bm Y] - \lambda\vec(\bm A)[\vec(\bm B)]'\|_F^2.
\end{equation*}
We note that the rearrangement operator $\mathcal{R}$ defined in
\eqref{eq:rearrange1} depends on the configuration of the block
matrix, and in the current case, on the configuration $(m,n)$. 
Let $\mathcal R[\bm Y] = \sum_{k=1}^d \lambda_ku_kv_k'$ be the SVD of the rearranged matrix $\mathcal R_{m, n}[\bm Y]$, where
$\lambda_1\geqslant \dots \geqslant \lambda_d \geqslant 0$ are the singular
values in decreasing order, $u_k$ and $v_k$ are the corresponding left
and right singular vectors and $d = 2^{m + n}\wedge 2^{m^\dagger + n^\dagger}$. The
estimators for model (\ref{eq:model-normalized}) are given by
\begin{equation}
   \hat\lambda = \lambda_1 = \|\mathcal R[\bm Y]\|_S,\quad \hat {\bm A} = \vec^{-1}(u_1), \quad \hat {\bm B} = \vec^{-1}(v_1),\quad \hat\sigma^2 = \|\bm Y\|_F^2 - \hat\lambda^2,\label{eq:nkp-est}
\end{equation}
where $\vec^{-1}$ is the inverse operation of $\vec(\cdot)$ that
restores a vector back into a matrix of proper dimensions.

We exam a few special cases of the configuration $(m,n)$. When
$(m, n) = (0, 0)$ or $(m, n) = (M, N)$, the nearest Kronecker
product approximation of $\bm Y$ is always itself. For instance, if
$m=n=0$, the estimators are
$$\hat\lambda = \|\bm Y\|_F,\quad \hat{\bm A}=1,\quad \hat{\bm B} = \hat\lambda^{-1}\bm Y,\quad \hat\sigma^2=0.$$
These two configurations are obviously over-fitting, and we shall
exclude them in the subsequent analysis.

When $(m, n) = (0,N)$ or $(m,n)=(M,0)$, the nearest Kronecker product
approximation of $\bm Y$ is the same as the rank-1 approximation of
$\bm Y$ without rearrangement. When the true configuration used to
generate $\bm Y$ is chosen, that is $(m, n)=(m_0, n_0)$, the problem is equivalent to denoising a perturbed
rank-1 matrix, since
\begin{equation}
    \mathcal R_{m_0,n_0}[\bm Y] = \lambda \vec(\bm A)\vec(\bm B)' + \dfrac{\sigma}{2^{(M+N)/2}} \mathcal R_{m_0,n_0}[\bm E], \label{eq:rearrange-true}
\end{equation}
where the rearranged noise matrix $\mathcal R_{m_0,n_0}[\bm E]$ is still a standard Gaussian ensemble. Therefore $\lambda$, $\bm A$ and $\bm B$ can
be recovered consistently when
$\sigma \|\mathcal R_{m_0,n_0}[\bm E]\|_S =o_p( \lambda\, 2^{(M+N)/2})$. Details will be
discussed in Section~\ref{sec:theory}.

\section{Configuration Determination through an Information Criterion}\label{sec:estimation-and-config}

Our primary goal is to recover the Kronecker product
$\lambda\bm A\otimes \bm B$ from $\bm Y$, based on model
\eqref{eq:model-normalized}.
It depends on the configuration of the Kronecker
product, which is typically unknown. We propose to use the information
criterion based procedure to select the configuration.

Recall that the dimension of $\bm Y$ is
$2^M\times 2^N$. If the dimension of $\bm A$ is $2^m\times 2^n$, then
the dimension of $\bm B$ must be $2^{m^\dagger}\times 2^{n^\dagger}$, where $m^\dagger=M-m$ and $n^\dagger=N-n$. Therefore,
the configuration can be indexed by the pair $(m,n)$, which takes
value from the Cartesian product set
$\{0,\dots, M\}\times\{0, \dots, N\}$.

For any given configuration $(m,n)$, the estimation procedure in
Section~\ref{sec:estimation} leads to the corresponding estimators
$\hat\lambda$, $\hat{\bm A}$ and $\hat{\bm B}$. Denote the estimated
Kronecker product by
$\hat {\bm Y}^{(m, n)} = \hat\lambda \hat{\bm A}\otimes \hat{\bm B}$.
Note that all of $\hat\lambda$, $\hat{\bm A}$ and $\hat{\bm B}$ depend
implicitly on the configuration $(m, n)$ used in estimation, and
should be written as $\hat\lambda=\hat\lambda^{(m,n)}$ etc. However, we
will suppress the configuration index from the notation for
simplicity, whenever its meaning is clear in the context. Under the
assumption that the noise matrix $\bm E$ is a standard Gaussian
ensemble, we define the information criterion as
\begin{equation}
    \mathrm{IC}_\kappa(m, n) = 2^{M+N}\ln {\|\bm Y - \hat {\bm Y}^{(m, n)}\|_F^2} + \kappa\eta, \label{eq:ic}
\end{equation}
where $\eta = 2^{m+n}+2^{m^\dagger + n^\dagger}$
is the number of parameters involved in the Kronecker product of the
configuration $(m, n)$, and $\kappa \geqslant 0$ controls the penalty on the
model complexity. The information criterion \eqref{eq:ic} can be
viewed as an extended version of the BIC. Similar proposals have been introduced by
\cite{chen2008extended} and \cite{foygel2010extended} in the linear
regression and graphical models setting, respectively.  The
information criterion \eqref{eq:ic} reduces to the log mean square
error when $\kappa=0$, and corresponds to the Akaike information criterion
(AIC) \citep{Akaike1998information} when $\kappa=2$, and the Bayesian
information criterion (BIC) \citep{schwarz1978estimating} when
$\kappa =\ln 2^{M+N} =(M+N)\ln 2$.

\smallskip
\stepcounter{remark}
\noindent {\bf Remark \arabic{remark}.} Strictly speaking, the number of parameters
involved in the Kronecker product $\lambda\bm A\otimes \bm B$ should
be $2^{m+n}+2^{m^\dagger+n^\dagger}-1$ because of the constraints
\eqref{eq:identi}. Since it does not affect the selection procedure to
be introduced in \eqref{eq:ic-selection}, we will use
$\eta = 2^{m+n}+2^{m^\dagger + n^\dagger}$ for simplicity.
\smallskip

The information criterion (\ref{eq:ic}) can be calculated for all
configurations, and the one corresponding to the smallest value of
\eqref{eq:ic} will be selected, based on which the estimation
procedure in Section~\ref{sec:estimation} proceeds. In other words, the selected configuration
$(\hat m,\hat n)$ is obtained through
\begin{equation}
    (\hat m, \hat n) = \argmin_{(m, n)\in\mathcal C}\ \mathrm{IC}_\kappa(m, n),\label{eq:ic-selection}
\end{equation}
where $\mathcal C$ is the set of all candidate configurations.

As discussed in Section~\ref{sec:estimation}, when $m=n = 0$ or
$(m, n)=(M, N)$, it holds that $\hat{\bm Y}=\bm Y$, and the information
criterion \eqref{eq:ic} will be $-\infty$, no matter what value $\kappa$
takes. Therefore, these two configurations should be excluded in model
selection and we use
$$\mathcal C := \{0, \dots, M\}\times \{0, \dots, N\}\setminus \{(0, 0), (M, N)\},$$
as the set of candidate configurations in
(\ref{eq:ic-selection}). Note that the set
$\{0, \dots, M\}\times \{0, \dots, N\}$ forms a rectangle lattice in
$\mathbb Z^2$, and $(m, n) = (0, 0)$ and $(m, n)=(M,N)$ are the bottom left
and top right corner of the lattice. Therefore, we sometimes intuitively refer to
these two configurations as the ``corner cases'' in the sequel. Furthermore, define $\mathcal W$ as the set of all
wrong configurations
$$\mathcal W := \mathcal C \setminus \{(m_0, n_0)\}.$$

We now provide a heuristic argument to show how the selection
procedure (\ref{eq:ic-selection}) is able to select the true
configuration $(m_0, n_0)$. We will leave some technical results aside, and only highlight the essential idea. Precise statements and their rigorous analysis will be presented in Section~\ref{sec:theory}.
For simplicity, assume that $\lambda$, $\sigma$ and $\kappa$ are fixed constants.
Also assume that both $(m_0+n_0)$ and $(m_0^\dagger+n_0^\dagger)$ diverge, so that the number of parameters $\eta_0=2^{m_0+n_0}+2^{m_0^\dagger+n_0^\dagger}$ is of a smaller magnitude than $2^{M+N}$.

According to \eqref{eq:nkp-est}, for a given configuration $(m,n)$,
$\mathcal R_{m, n}[\hat{\bm Y}]$ equals the first SVD component of
$\mathcal R_{m, n}[\bm Y]$, and it follows that
$\|\bm Y - \hat{\bm Y}\|_F^2 = \|\bm Y\|_F^2 - \|\hat{\bm Y}\|_F^2 =
\|\bm Y\|_F^2 - \hat\lambda^2$, and the information criterion
\eqref{eq:ic} can be rewritten as
\begin{equation}
    \mathrm{IC}_\kappa(m, n) = 2^{M+N}\ln (\|\bm Y\|_F^2 - \hat\lambda^2) + \kappa\eta. \label{eq: ic2}
\end{equation}

For the true configuration $(m, n) = (m_0, n_0)$, the rearranged matrix
$\mathcal R_{m_0, n_0}[\bm Y]$ takes the form (\ref{eq:rearrange-true}), where
the first term is a rank-1 matrix of spectral norm
$\lambda$, and the noise term has a spectral norm of the
order $O(2^{-(m_0+n_0)/2}+2^{-(m_0^\dagger+n_0^\dagger)/2})$ (details given in Section \ref{sec:theory}), which is negligible relative to $\lambda$, under the assumption
$m_0+n_0\gg 1, m_0^\dagger + n_0^\dagger \gg 1$. 
So under the true configuration, $\hat\lambda\approx\lambda$.
On the other hand, the number of parameters $\eta_0 = o(2^{M+N})$, making the penalty term much smaller than the log likelihood in (\ref{eq:ic}). To summarize,
\begin{equation*}
  \mathrm{IC}_\kappa(m_0, n_0)\approx 2^{M+N}\ln\left[\|\lambda\bm A\otimes \bm B + \sigma\, 2^{-(M+N)/2}\bm E\|_F^2-\lambda^2\right]
  \approx 2^{M+N}\ln\sigma^2.
\end{equation*}

For a wrong configuration $(m, n)\in\mathcal W$ that is close to the true one, the spectrum norm
$\|\mathcal R_{m, n}[\bm E]\|_S$ and the number of parameters $\eta$ are still negligible. However, the estimated coefficient $\hat\lambda$ is smaller than $\lambda$ since
$$\hat\lambda = \|\mathcal R_{m, n}[\bm Y]\|_S\approx \|\mathcal R_{m, n}[\lambda \bm A\otimes \bm B]\|_S<\lambda. 
$$
Let us assume that $\|\mathcal R_{m, n}[\lambda \bm A\otimes \bm B]\|_S\leq\phi\lambda$ for some $0<\phi<1$, which implies that for the wrong configuration $(m,n)$, 
\begin{align*}
  \mathrm{IC}_\kappa(m,n) & \approx 2^{M+N}\ln\left[{\|\lambda\bm A\otimes \bm B + \sigma2^{-(M+N)/2}\bm E\|_F^2-\hat\lambda^2}\right]\\
  &\approx 2^{M+N}\ln\left[{\|\sigma 2^{-(M+N)/2}\bm E\|_F^2+\lambda^2-\phi^2\lambda^2}\right]\\
  & \approx 2^{M+N}\ln\left[\sigma^2\left(1+\frac{(1-\phi^2)\lambda^2}{\sigma^2}\right)\right].
\end{align*}
Therefore, the information criterion \eqref{eq:ic} is in favor of the true configuration over a wrong but close-to-truth one, and the two quantities are separated by $$\mathrm{IC}_\kappa(m,n)-\mathrm{IC}_\kappa(m_0,n_0)\approx 2^{M+N}\ln[1+(1-\phi^2)\lambda^2/\sigma^2].$$

On the other hand, for a wrong configuration $(m, n)\in\mathcal W$
that is close to the corner configuration $(0, 0)$ or $(M,N)$, the singular value $\|\mathcal R_{m, n}[\bm E]\|_S$ can be as large as $1/2$, making the separation between $\mathrm{IC}_\kappa(m,n)$ and $\mathrm{IC}_\kappa(m_0,n_0)$ by the log likelihood not guaranteed, i.e. it can happen that $\hat\lambda>\lambda$ under the wrong configuration. But at the same time the number of parameters $\eta$ is also approximately $2^{M + N}$ 
, so $\mathrm{IC}_\kappa(m,n)$ receives a heavy penalty, which once again makes it greater than $\mathrm{IC}_\kappa(m_0,n_0)$. 

In summary, the trade-off between log likelihood and model
complexity plays its role here, as expected. Wrong but close-to-truth
configurations involve similar numbers of parameters as the true one,
but lead to much smaller likelihoods. On the other hand, a close-to-corner
configuration may yield a $\hat{\bm Y}$ closer to the original $\bm Y$, but requires much more
parameters to do so. The true configuration can thus be selected because it reaches the optimal balance between the the likelihood and model complexity.

In the preceding discussion we have assumed many convenient conditions to simplify the arguments and to signify the essential idea. In particular, by assuming that $\lambda$ is a positive constant, the signal strength in model \eqref{eq:model-normalized} is quite strong. In Section~\ref{sec:theory} we will make effort to establish the model selection consistency under minimal conditions.

\section{Theoretical Results}\label{sec:theory}
In this section we provide a theoretical guarantee of
the configuration selection procedure proposed in
Section~\ref{sec:estimation-and-config}, by establishing its asymptotic consistency. Throughout this
section all our discussion will be based on model
\eqref{eq:model-normalized}.

\subsection{Assumptions and Estimation Consistency under Known Configuration}
 
We first introduce the assumptions of the theoretical analysis. Recall that for model \eqref{eq:model-normalized}, $(m_0,n_0)$ denotes the true
configuration, i.e. the matrices $\bm A$ and $\bm B$ are of dimensions
$2^{m_0}\times 2^{n_0}$ and $2^{m_0^\dagger}\times 2^{n_0^\dagger}$ respectively.
For the asymptotic analysis, we make the following
assumption on the sizes of $\bm A$ and $\bm B$,
which follows the paradigm of high dimensional analysis. 

\begin{assump}[Assumption on Dimension]
  \label{assump:hd}
  Consider model \eqref{eq:model-normalized}. As $M+N\rightarrow \infty$, assume that the true
  configuration $(m_0,n_0)$ satisfies
$$\dfrac{m_0 + n_0}{\ln\ln(MN)}\rightarrow\infty,\quad \dfrac{m_0^\dagger + n_0^\dagger}{\ln\ln(MN)} \rightarrow\infty, $$
where $m_0^\dagger = M-m_0$ and $n_0^\dagger = N-n_0$. 
\end{assump}
The condition entails that the numbers of entries in $\bm A$ and
$\bm B$ will need to diverge to infinity, and so is that of $\bm
Y$. It is also ensured that the true configuration cannot stay too close to the corners.
We remark that this will be the only condition
on the sizes of the involved matrices. In particular, we do not require all of
$m_0,n_0,m_0^\dagger,n_0^\dagger$ to go to infinity. Consequently, the low rank approximation (when $(m_0,n_0)=(M,0)$ or $(m_0,n_0)=(0,N)$) is also covered by the KoPA framework and our analysis as a special case.

The number of parameters involved in the Kronecker product
$\lambda\bm A\otimes \bm B$ is $\eta_0 = 2^{m_0+n_0} +
2^{m_0^\dagger + n_0^\dagger}$. It is a much smaller number than $2^M\times 2^N$, the number of elements in $\bm Y$. Hence Assumption~\ref{assump:hd} implies a significant dimension reduction.

We also make the following assumption on the error matrix $\bm E$.
\begin{assump}[Assumption on Noise]
  \label{assump:Emat}
  Consider model \eqref{eq:model-normalized}. Assume that $\bm E$
  is a standard Gaussian ensemble, i.e. with IID standard normal
  entries. 
\end{assump}

We conclude this subsection with the convergence rates of the estimators $\hat\lambda$
,$\hat{\bm A}$ and $\hat{\bm B}$, given by the estimation procedure in
Section \ref{sec:estimation} under the true configuration. Since the
error matrix $\bm E$ has IID standard normal entries, according to \cite{Vershynin2010Introduction}, the expectation
of the largest singular value of the rearranged error matrix
$\mathcal R_{m_0, n_0}[\bm E]$ is bounded by
$$s_0 = 2^{(m_0+n_0)/2} + 2^{(m_0^\dagger + n_0^\dagger)/2}.$$

\begin{thm}\label{thm:truth}
  Let $\hat\lambda$, $\hat{\bm A}$ and $\hat{\bm B}$ be the estimators
  obtained under the true configuration, as given in
  (\ref{eq:nkp-est}). Suppose Assumptions \ref{assump:hd} and
  \ref{assump:Emat} hold, then for the deterministic scheme of
  model \eqref{eq:model-normalized}, we have
$$
    \dfrac{\hat\lambda - \lambda}{\lambda} = O_p\left(\frac{r_0}{\lambda/\sigma}\right),\quad
    \|\hat{\bm A}-\bm A\|_F^2 = O_p\left(\frac{r_0}{\lambda/\sigma}\right),\quad
    \|\hat{\bm B}-\bm B\|_F^2 = O_p\left(\frac{r_0}{\lambda/\sigma}\right),
$$
where 
$$r_0=\dfrac{s_0}{2^{(M+N)/2}}=2^{-(m_0+n_0)/2} + 2^{-(m_0^\dagger + n_0^\dagger)/2}.$$
\end{thm}

\subsection{Consistency of Configuration Selection}
\label{sec:deterministic}
To study the consistency of the configuration selection proposed in Section \ref{sec:estimation-and-config}, we need assumptions on the signal-to-noise ratio. We choose to present model \eqref{eq:model-normalized} with both $\lambda$ and $\sigma$ so that it is able to account for any actual data generating mechanism. On the other hand, the mathematical properties would only depend on the ratio $\lambda/\sigma$.
The strength of the signal also depends on the contrast between true and wrong configurations. If a configuration $(m, n)\in\mathcal W$ is used for the estimation, $\bm Y$ is rearranged as
\begin{equation}
    \mathcal R_{m, n}[\bm Y] = \lambda \mathcal R_{m,n}[\bm A\otimes \bm B] + \sigma 2^{-(M+N)/2}\mathcal R_{m, n}[\bm E]. \label{eq:wrong-rearrange}
\end{equation} 
Ignoring the noise term, only the first singular value component of $\mathcal R_{m,n}[\bm A\otimes \bm B]$ (multiplied by $\lambda$) is expected to enter $\hat{\bm Y}$.
When the true configuration is used, $\mathcal R_{m,n}[\bm A\otimes \bm B]$ is a rank-1 matrix, and its leading singular value equals 1 (recall that we have assumed that $\|\bm A\|_F=\|\bm B\|_F=1$). On the other hand, if a wrong configuration is used, then $\mathcal R_{m,n}[\bm A\otimes \bm B]$ is no longer rank-1, and its leading singular value should be smaller than 1. Define
\begin{equation*}
\phi := \max_{(m, n)\in\mathcal W} \|\mathcal R_{m,n}[\bm A \otimes {\bm B}]\|_S.   
\end{equation*}
The quantity $\phi$ characterize how much of the signal $\bm A\otimes\bm B$ can be captured by a wrong configuration, and it always holds that $0<\phi\leq 1$, so we also introduce $$\psi^2:=1-\phi^2,$$ and call it the {\it representation gap}. Note that $0\leq\psi^2<1$, and the larger $\psi^2$ is, the easier it is to separate true and wrong configurations. The following assumption shows the interplay between the representation gap $\psi^2$ and the signal-to-noise ratio $\lambda/\sigma$.
\begin{assump}[Representation Gap]
\label{assump:gap}
For model \eqref{eq:model-normalized}, assume that $\bm A$ and $\bm B$ are deterministic matrices, and
\begin{equation}
\label{eq:gap1}
\lim_{\asymptotic} \frac{2^{(M+N)/2}}{2^{(m_0+n_0)/2}+2^{(m_0^\dagger+n_0^\dagger)/2}}\cdot(\lambda/\sigma)\cdot\psi = \infty,
\end{equation}
and
\begin{equation}
\label{eq:gap2}
\lim_{\asymptotic} 2^{(M+N)/4}\cdot(\lambda/\sigma)\cdot\psi^2 =\infty.
\end{equation}
\end{assump}
In both \eqref{eq:gap1} and \eqref{eq:gap2}, the signal-to-noise ratio and the representation gap $\psi^2$ can diminish to zero, as long as they do not converge to zero too fast. In this sense, Assumption~\ref{assump:gap} is very flexible by requiring only very week signal strength.

\stepcounter{remark}
\noindent {\bf Remark \arabic{remark}.} We have defined $\phi$ as the maximum over $\mathcal W$, the set of all wrong
configurations. In fact, if we let $\phi_{m,n} := \|\mathcal R_{m,n}[\bm A \otimes {\bm B}]\|_S$, and $\psi^2_{m,n}=1-\phi_{m,n}^2$, then Assumption~\ref{assump:gap} can also be given through $\psi^2_{m,n}$ instead of an uniform lower bound $\psi^2$, leading to a weaker version of the assumption. On the other hand, as will be shown in Section~\ref{sec:random}, if $\bm A$ and $\bm B$ are randomly generated according to the Random Scheme, then indeed all $\psi^2_{m,n}$ are larger than or around $1/2$ with an overwhelming probability. This is suggesting that using the lower bound $\psi^2$ in Assumption~\ref{assump:gap} for the deterministic scheme is still reasonable. Therefore, we do not spell out the detailed version of Assumption~\ref{assump:gap} using $\psi^2_{m,n}$, but present it in the current simple form.

\stepcounter{remark}
\noindent {\bf Remark \arabic{remark}.} Notions similar to the representation gap appear as key parameters in many
other problems. For example, in variable selection of linear regression problems,
the representation gap would be the smallest absolute non-zero coefficient in the model. In matrix rank determination problems or factor models, the representation gap would be the eigen-gap, or the smallest nonzero singular value.

The following theorem quantifies the separation of the information criterion \eqref{eq:ic} between the true and wrong configurations.
\begin{thm}\label{thm:ic-gap-nonrandom}
  Consider model \eqref{eq:model-normalized}, and assume Assumptions \ref{assump:hd}, \ref{assump:Emat}, \ref{assump:gap}. If
\begin{equation}
    \kappa\geq2\ln 2,\quad \text{and}\quad \kappa = o\left(\dfrac{2^{M+N}\ln(1+(\lambda/\sigma)^2\psi^2)}{2^{m_0+n_0} + 2^{m_0^\dagger+n_0^\dagger}}\right),\label{eq:ic-gap-nonrandom-cond}
\end{equation}
then 
\begin{equation*}
    \min_{(m,n)\in\mathcal W}\mathbb E [\mathrm{IC}_\kappa(m,n)] - \mathbb E[\mathrm{IC}_\kappa(m_0,n_0)] \geq 2^{M+N}\cdot\ln[1+(\lambda/\sigma)^2\psi^2]\cdot(1+o(1)).
\end{equation*}
\end{thm}
To be precise, we note that for a sequence of numbers $\{a_k\}$, the statement $a_k\geq o(1)$ is understood as $\max\{-a_k,0\}=o(1)$. According to Assumptions~\ref{assump:gap}, $(\lambda/\sigma)^2\psi^2\gg 2^{-(M+N)/2}$, so Theorem~\ref{thm:ic-gap-nonrandom} shows that the separation of the information criterion is at least of the order $2^{(M+N)/2}$.

\stepcounter{remark}
\noindent {\bf Remark \arabic{remark}.} The first condition in \eqref{eq:ic-gap-nonrandom-cond} ensures that the penalty on the number of parameters is large enough to exclude configurations close to $(0, 0)$ and $(M, N)$. The second condition in \eqref{eq:ic-gap-nonrandom-cond} is imposed so that the contribution from the penalty term under the true configuration is dominated by the representation  gap. The exact formula of the difference in expected information criterion is given by \eqref{eq:proof-deic} in Appendix.

Next theorem establishes the consistency of \eqref{eq:ic}. We need to define the symbol $\gtrsim$: for two sequences of positive numbers $\{a_k\}$ and $\{b_k\}$, $a_k\gtrsim b_k$ is defined as $\lim\inf_{k\rightarrow\infty}a_k/b_k>0$.
\begin{thm}\label{thm:consistency_nonrandom}
  Assume the same conditions of Theorem \ref{thm:ic-gap-nonrandom}, then
  $$P\left[\mathrm{IC}_\kappa(m_0, n_0) < \min_{(m,n)\in\mathcal W}\mathrm{IC}_\kappa(m, n)\right]\geqslant 1-\exp\left\{- C2^{M+N}+\ln(MN)\right\},$$ 
  where the constant $C$, depending on $\lambda/\sigma$ and $\psi$, is of order
  $$C(\lambda/\sigma, \psi)\gtrsim (\alpha^{1/3}-1)\wedge \left(\dfrac{\alpha - \alpha^{2/3}}{1+\lambda/\sigma}\right)^2,$$
  with $\alpha = 1+(\lambda/\sigma)^2\psi^2$. In particular, the preceding convergence rate implies the consistency of the configuration selection, i.e.
  \begin{equation}
      \label{eq:consistency}
    \lim_{\asymptotic}P\left[\mathrm{IC}_\kappa(m_0, n_0) < \min_{(m,n)\in\mathcal W}\mathrm{IC}_\kappa(m, n)\right]=1.
  \end{equation}
\end{thm}

\stepcounter{remark}
\noindent {\bf Remark \arabic{remark}.} In Assumption~\ref{assump:gap}, we focus on the minimal signal-to-noise ratio and representation gap. On the other hand, if they are large such that $\liminf (\lambda/\sigma)^2\psi^2\geq 1/2$, then the condition $\kappa\geq 2\ln 2$ can be dropped from Theorem~\ref{thm:ic-gap-nonrandom} and Theorem~\ref{thm:consistency_nonrandom}, which will continue to hold if we set $\kappa=0$ in \eqref{eq:ic}. In other words, if the signal strength and the representation gap are sufficiently large, one can simply use mean squared error to select the configuration. Specifically, it requires $\lambda^2\psi^2/\sigma^2 > 1/2$ to enable the use of $\kappa = 0$ in the information criterion.

\subsection{Model Selection under Random Scheme}
\label{sec:random}

In this section we consider the consistency of the model selection under the random scheme \eqref{eq:random}. 
First of all, similar convergence rates as Theorem~\ref{thm:truth} can be obtained under the random scheme.
\begin{corr}\label{corr:truth-random}
  Assume Assumptions \ref{assump:hd} and \ref{assump:Emat}. If $\bm A$ and $\bm B$ are generated according to the random scheme \eqref{eq:random}, then the conclusion of Theorem~\ref{thm:truth} continue to hold.
\end{corr}

If a configuration $(m, n)\in\mathcal W$ is used, then the estimation procedure given in Section~\ref{sec:estimation} rearranges $\bm Y$ as \eqref{eq:wrong-rearrange}.
In Section~\ref{sec:deterministic} for the deterministic scheme, we introduce $\phi$ as the upper bound of $\|{\cal R}_{m,n}[{\bm A}\otimes{\bm B}]\|_S$ over all wrong configurations. For the random scheme, it turns out this upper bound and hence the representation gap $\psi$, depending on $\bm A$ and $\bm B$, is also random. We introduce the following “random” version of Assumption~\ref{assump:gap}.
\begin{assump}[Representation Gap]
\label{assump:gap_random}
Assuem in model \eqref{eq:model-normalized}, $\lambda$, $\bm A$ and $\bm B$ are random and independent with $\bm E$. Assume there exist two sequences of positive numbers $\{\lambda_0\}$ and $\{\psi_0\}$ satisfying \eqref{eq:gap1} and \eqref{eq:gap2} (by replacing $\lambda$ and $\psi$ therein), such that
  $$\limsup_{\asymptotic}\mathbb E[\lambda^2/\lambda_0^2] < \infty,\quad \limsup_{\asymptotic}\mathbb E[\psi^2/\psi_0^2] < \infty,$$
  and for any constant $c>0$
  \begin{equation}
      \lim_{\asymptotic} MN\cdot P\left[{\lambda^2}/{\lambda_0^2} <1-c\right]= \lim_{\asymptotic} MN\cdot P\left[{\psi^2}/{\psi_0^2} <1-c\right]=0.\label{eq:gap-random-tail-cond}
  \end{equation}
\end{assump}
With Assumption~\ref{assump:gap_random}, Theorem~\ref{thm:ic-gap-nonrandom} and \ref{thm:consistency_nonrandom} continue to hold under the random scheme, as asserted by the next theorem.
\begin{thm}\label{thm:ic-gap}
  Consider model \eqref{eq:model-normalized} with random $\lambda$, $\bm A$ and $\bm B$. Under Assumptions~\ref{assump:hd}, \ref{assump:Emat} and \ref{assump:gap_random}, it holds that
  \begin{equation*}
    \min_{(m,n)\in\mathcal W}\mathbb E [\mathrm{IC}_\kappa(m,n)] - \mathbb E[\mathrm{IC}_\kappa(m_0,n_0)] \geq 2^{M+N}\cdot\ln[1+(\lambda_0/\sigma)^2\psi_0^2]\cdot(1+o(1)).
\end{equation*}
Furthermore, the consistency \eqref{eq:consistency} holds.
\end{thm}

Assumption~\ref{assump:gap_random} is formulated to single out the minimal condition required for the consistency under the random scheme. There is no specific distributional assumptions imposed on $\bm A$ and $\bm B$.
In the rest of this section, we demonstrate that how it can be satisfied under normality. 
\begin{eg}
\label{example:normal}
Consider model \eqref{eq:model-normalized}. Suppose that
\begin{equation}
  \label{eq:random}
  \lambda\bm A\otimes\bm B = \dfrac{\lambda_0\tilde{\bm A}\otimes\tilde{\bm B}}{2^{(M+N)/2}},
\end{equation}
where $\lambda_0$ is deterministic, and $\tilde{\bm A}$ and $\tilde{\bm B}$ are independent, and both
consisting of IID standard normal entries. In order to fulfill the
identifiability condition \eqref{eq:identi}, we let
$\bm A=\tilde{\bm A}/\|\tilde{\bm A}\|_F$,
$\bm B=\tilde{\bm B}/\|\tilde{\bm B}\|_F$, and
$\lambda=\lambda_0\cdot\|\tilde{\bm A}\|_F\cdot\|\tilde{\bm B}\|_F/2^{(M+N)/2}$. Also assume that $\bm A$ and $\bm B$ are both independent with $\bm E$. For this example, the signal-to-noise ratio becomes
\begin{equation*}
  \dfrac{\mathbb E\|\lambda \bm A\otimes \bm B\|_F^2}{\mathbb E\|\sigma\bm E/2^{(M+N)/2}\|_F^2} 
  =\dfrac{\lambda_0^2}{\sigma^2}. 
\end{equation*}
\end{eg}

Recall that $\phi$ is defined as the upper bound of $\|{\cal R}_{m,n}[{\bm A}\otimes{\bm B}]\|_S$ over all wrong configurations. Only when the true configurations $(m_0,n_0)$ is used, the rearrangement ${\cal R}_{m_0,n_0}[{\bm A}\otimes{\bm B}]$ has the simple structure of a rank-1 matrix. Under a wrong configuration ${\cal R}_{m,n}[{\bm A}\otimes{\bm B}]$ no longer takes any special form. Nevertheless, the following lemma characterizes how the spectral norm of ${\cal R}_{m,n}[{\bm A}\otimes{\bm B}]$ depends on further rearrangements of both $\bm A$ and $\bm B$. It is a property of the Kronecker products and the KPD \eqref{eq:kpd}, so we present it in the general form, without referring to any ``true'' configuration.
\begin{lem}\label{lem:further-decomposition}
  Let $\bm A$ be a $2^m\times 2^n$ matrix and $\bm B$ be a
  $2^{m^\dagger}\times 2^{n^\dagger}$ matrix. Then for any
  $m', n'\in\mathbb Z,\ 0\leqslant m'\leqslant M,\ 0\leqslant
  n'\leqslant N$,
  \begin{equation*}
    \|\mathcal R_{m',n'}[\bm A\otimes \bm B]\|_S = \|\mathcal R_{m\wedge m', n\wedge n'}[\bm A]\|_S\cdot \|\mathcal R_{(m'-m)_+, (n'-n)_+}[\bm B]\|_S
\end{equation*}
\end{lem}

Applying Lemma~\ref{lem:further-decomposition} to Example~\ref{example:normal} leads to the following corollary.
\begin{corr}\label{corr:further-decomp-random}
  For Example~\ref{example:normal}, under Assumption~\ref{assump:hd}, it holds that 
$$\max_{(m,n)\in\mathcal W} \|\mathcal R_{m,n}[\bm A\otimes \bm B]\|_S = \dfrac{1}{\sqrt{2}}+o_p(1).$$
And as a consequence, Assumption~\ref{assump:gap_random} holds with the $\lambda_0$ in \eqref{eq:random} and $\psi_0^2=1/2$.
\end{corr}

\section{Multi-term Kronecker Product Models}
\label{sec:extension}

In this section, we extend the one-term Kronecker product model in
(\ref{eq:model-normalized}) to the following $K$-term Kronecker
product model.
\begin{equation}
    \bm Y = \sum_{k=1}^K\lambda_k\bm A_k\otimes \bm B_k +\dfrac{\sigma}{2^{(M+N)/2}} \bm E,\label{eq:model-multi-term}
\end{equation}
where $\lambda_1\geqslant \lambda_2\geqslant \cdots \geqslant \lambda_K > 0$ and
$\bm A_k\in\mathbb R^{2^{m_0}\times 2^{n_0}}$,
$\bm B_k\in\mathbb R^{2^{m_0^\dagger}\times 2^{n_0^\dagger}}$, $k=1,\cdots, K$ satisfy the
following orthonormal condition:
$$\tr(\bm A_k\bm A_l') = \tr(\bm B_k\bm B_l')=\delta_{kl}:=\begin{cases}
1&\text{if }k=l,\\
0&\text{if }k\neq l.
\end{cases}$$ 
The orthonormal condition implies the identifiability:
if $\lambda_1>\lambda_2>\cdots >\lambda_k >0$, then $\bm A_k$ and $\bm B_k$ are
identified up to a sign change, see Section~\ref{sec:kpd}. Note that
the $K$ terms in model (\ref{eq:model-multi-term}) have the same configuration $(m_0, n_0)$. Therefore, although multiple terms are present, there is only one configuration to be determined. 

Once the configuration is given, the rearranged $\bm Y$ becomes the sum of a rank $K$ matrix and a noise matrix. The determination of $K$ turns into the rank selection problem in low rank approximation, and existing methods \citep{bai2003inferential,ahn2013eigenvalue} can be applied. Therefore, we focus on the choice of the configuration for model \eqref{eq:model-multi-term}. We propose to use the same procedure based on the one-term model, although there are actually $K$ terms in model \eqref{eq:model-multi-term}. We show that, if the leading term in \eqref{eq:model-multi-term} is strong enough, i.e. if $\lambda_1$ is large enough, compared with other $\lambda_k$ as well as $\sigma$, the information criterion introduced in Section~\ref{sec:estimation-and-config} will continue to select the true configuration consistently.

For ease of presentation, we only provide the result and analysis of the two-term model
\begin{equation}
    \bm Y = \lambda_1\bm A_1\otimes \bm B_1 + \lambda_2\bm A_2\otimes \bm B_2 +\dfrac{\sigma}{2^{(M+N)/2}} \bm E.\label{eq:model-two-term}
\end{equation}
Similar results can be directly extended to the multi-term model.

We propose to use the same configuration selection procedure in
Section \ref{sec:estimation-and-config}, that is, for any candidate configuration
$(m, n)\in\mathcal C$, although $\bm Y$ is generated from the two-term
model (\ref{eq:model-two-term}), we nonetheless still calculate the
information criterion \eqref{eq:ic} by fitting the one-term Kronecker
product model (\ref{eq:model-normalized}) to $\bm Y$. This approach avoids 
the need of the determination of the number of Kronecker product terms when 
seeking the correct configuration. It allows the separation of the two. In this case,
the estimated $\hat\lambda$ used in the information criterion
(\ref{eq: ic2}) is
\begin{equation}
\hat\lambda = \|\mathcal R_{m, n}[\bm Y]\|_S = \|\lambda_1\mathcal R_{m, n}[\bm A_1\otimes \bm B_2]+\lambda_2\mathcal R_{m, n}[\bm A_2\otimes \bm B_2] +\sigma2^{-(M+N)/2}\mathcal R_{m, n}[\bm E]\|_S.\label{eq:lambda-hat-two-term}
\end{equation}
Note that under the true configuration, we have $\hat\lambda \approx \lambda_1$. To bound $\hat\lambda$ under wrong configurations, we define
\begin{align*}
    \phi_1 = \max_{(m, n)\in\mathcal W} \|\mathcal R_{m,n}[{\bm A}_1\otimes{\bm B}_1]\|_S,\qquad
    \phi_2 = \max_{(m, n)\in\mathcal W} \|\mathcal R_{m,n}[{\bm A}_2\otimes{\bm B}_2]\|_S,
\end{align*}
and the representation gaps
$$\psi_1^2 := 1-\phi_1^2,\qquad \psi_2^2:=1-\phi_2^2.$$

Even though $\vec(\bm A_1)$ and $\vec(\bm A_2)$ are orthogonal
according to the model assumption, the column spaces of
$\mathcal R_{m, n}[\bm A_1\otimes \bm B_1]$ and 
$\mathcal R_{m, n}[\bm A_2\otimes \bm B_2]$ are not necessarily
orthogonal. In the worst case when
$\mathcal R_{m, n}[\bm A_1\otimes \bm B_1]$ and
$\mathcal R_{m, n}[\bm A_2\otimes \bm B_2]$ have the same column
space and the same row space, then $\hat\lambda$ in
(\ref{eq:lambda-hat-two-term}) can be close to
$\lambda_1\phi_1 + \lambda_2\phi_2$, which may exceed
$\lambda_1$. Therefore, we need to bound the distance between the
column (and row) spaces of
$\mathcal R_{m, n}[\bm A_1\otimes \bm B_1]$ and
$\mathcal R_{m, n}[\bm A_2\otimes \bm B_2]$. For this purpose, we
make use of the principal angles between linear subspaces. 
Specifically, if $\bm M_1$ and $\bm M_2$ are two
matrices of the same number of rows, the smallest principal angle between their
column spaces, denote by
$\Theta(\bm M_1, \bm M_2)$, is defined as
$$\cos \Theta(\bm M_1, \bm M_2) = \sup_{u_1\neq 0, u_2\neq 0} \dfrac{u_1'\bm M_1'\bm M_2u_2}{\|\bm M_1u_1\|\|\bm M_2u_2\|}.$$

We first discuss the deterministic scheme, where $\bm A_k$ and
$\bm B_k$ are non-random. In Assumption~\ref{assump:Theta}, $\theta_c$
and $\theta_r$ are lower bounds of the smallest possible principal
angles between the column spaces and the row spaces of the two
rearranged components, respectively.

\begin{assump}\label{assump:Theta}
There exist $0<\xi<1$ such that
$$\max_{(m, n)\in\mathcal W_A}\cos \Theta(\mathcal R_{m, n}[\bm A_1\otimes \bm B_1], \mathcal R_{m, n}[\bm A_2\otimes \bm B_2])\leqslant \xi,$$
and
$$\max_{(m, n)\in\mathcal W_B}\cos \Theta([\mathcal R_{m, n}[\bm A_1\otimes \bm B_1]]', [\mathcal R_{m, n}[\bm A_2\otimes \bm B_2]]')\leqslant \xi,$$
where 
$$\mathcal W_A = \{(m, n)\in\mathcal W: m+n \geqslant m^\dagger + n^\dagger\}, \mathcal W_B = \{(m, n)\in\mathcal W: m+n < m^\dagger + n^\dagger\}.$$
\end{assump}

\stepcounter{remark}
\noindent {\bf Remark \arabic{remark}.} The assumption may look unintuitive at first sight, since it might be thought that the matrices $\bm A_i\otimes \bm B_i$, after the rearrangement under wrong configurations, are in general full rank. This is, however, not true, in view of Lemma~\ref{lem:further-decomposition}, most easily seen when the wrong configuratoin $(m,n)$ is nested with the true one $(m_0,n_0)$ in the sense $m\leq n_0$ and $n\leq n_0$. On the other hand, the conditions in Assumption~\ref{assump:Theta} are given separately over $\mathcal W_A$ and $\mathcal W_B$. In each of them, the matrices involved have more rows than columns, and the condition is on the corresponding column spaces.

The following lemma provides an upper bound of the spectral norm of a sum of two matrices. It utilizes the principal angles between the column and row spaces to make the bound sharper than the one given by the triangular inequality. Assumption~\ref{assump:Theta} enables us to apply Lemma~\ref{lem:norm-of-sum} to bound $\hat\lambda$ in \eqref{eq:lambda-hat-two-term}.
\begin{lem}\label{lem:norm-of-sum}
  Suppose $\bm M_1$ and $\bm M_2$ are two matrices of the same
  dimension. Let $\|\bm M_1\|_S=\mu$, $\|\bm M_2\|_S=\nu$. Denote the
  principle angles between the column spaces and the row spaces as
  $\theta = \Theta(\bm M_1, \bm M_2)$,
  $\eta = \Theta(\bm M_1', \bm M_2')$, respectively. Then
$$\|\bm M_1+\bm M_2\|_S^2\leqslant \Lambda^2(\mu, \nu, \theta, \eta),$$
where
\begin{align*}
    \Lambda^2(\mu, \nu, \theta, \eta)=\left.\dfrac{1}{2}\right[&
    \sqrt{\left(\mu^2+\nu^2 + 2\mu\nu\cos\theta\cos\eta\right)^2-4\mu^2\nu^2\sin^2\theta\sin^2\eta} \\
    &\left.\phantom{\frac{1}{2}}\!\!\!\!\!+\mu^2+\nu^2+2\mu\nu\cos\theta\cos\eta \right].
\end{align*}
\end{lem}
Similar to Assumption~\ref{assump:gap}, we assume the signal strengths $\lambda_1$, $\lambda_2$ and the noise level $\sigma$ satisfy the following assumption.
\begin{assump}\label{assump:snr-two-term}
For model \eqref{eq:model-two-term}, we assume that $\lambda_k$ and the matrices ${\bm A}_k$, ${\bm B}_k$, $k=1,2$ are deterministic and 
\begin{equation}
    \lim_{\asymptotic} \dfrac{2^{M+N}}{2^{m+n}+2^{m^\dagger+n^\dagger}}\dfrac{\lambda_1^2\psi_1^2 - \lambda_2^2\phi_2^2 - 2\lambda_1\lambda_2\phi_1\phi_2\xi}{\sigma^2+\lambda_2^2}=\infty \label{eq:two-term-gap1}
\end{equation}
and
\begin{equation}
    \lim_{\asymptotic} 2^{(M+N)/4}\dfrac{\lambda_1^2\psi_1^2 - \lambda_2^2\phi_2^2 - 2\lambda_1\lambda_2\phi_1\phi_2\xi}{(\lambda_1+\lambda_2)\sigma}=\infty.\label{eq:two-term-gap2}
\end{equation}
\end{assump}
The conditions \eqref{eq:two-term-gap1} and \eqref{eq:two-term-gap2} correspond to \eqref{eq:gap1} and \eqref{eq:gap2} in the one-term model. Specifically, when $\lambda_2 = 0$, the two-term model reduces to one-term case, and Assumption~\ref{assump:snr-two-term} reduces to Assumption~\ref{assump:gap} as well. The main result for the two-term model is stated in Theorem~\ref{thm:gap-two-term}.

\begin{thm}\label{thm:gap-two-term}
  Consider the two-term model \eqref{eq:model-two-term}, where $\lambda_k$ and the matrices $\bm A_k$ and $\bm B_k$ are deterministic. Suppose Assumptions~\ref{assump:hd}, \ref{assump:Emat}, \ref{assump:Theta} and \ref{assump:snr-two-term} hold.
If $\kappa$ satisfies
\begin{equation*}
    \kappa \geqslant 2\ln 2\quad \text{and}\quad \kappa=o\left(\dfrac{2^{M+N}\alpha}{2^{m_0+n_0}+2^{M+N-m_0-n_0}}\right),
\end{equation*}
then 
$$\min_{(m,n)\in\mathcal W} \mathbb E[\mathrm{IC}_\kappa(m, n)] - \mathbb E[\mathrm{IC}_\kappa(m_0, n_0)]\geqslant 2^{M+N}\alpha(1+ o_p(1)),$$
where 
\begin{equation}
\alpha=\ln\left(1+\dfrac{\lambda_1^2\psi_1^2 - \lambda_2^2\phi_2^2 - 2\lambda_1\lambda_2\phi_1\phi_2\xi}{\sigma^2+\lambda_2^2}\right).\label{eq:two-term-alpha}
\end{equation}
Furthermore, the consistency \eqref{eq:consistency} continues to hold.
\end{thm}

Similar to Theorem~\ref{thm:ic-gap-nonrandom}, we have shown that for the two-term model, the information criterion obtained by fitting a one-term model can still separate the true and wrong configurations with a gap of the order $O(2^{M+N}\alpha)$. 
On the other hand, comparing with Assumption~\ref{assump:gap}, Theorem~\ref{thm:gap-two-term} depends on Assumption~\ref{assump:snr-two-term}, which requires not only the signal-to-noise ratio ($\lambda_1/\sigma$), but also the relative strength of the two terms ($\lambda_1/\lambda_2$) to be large enough. Comparing the two term model \eqref{eq:model-two-term} with the one term model (i.e. $\lambda_2=0$), we note that the information criterion gap $\alpha$ in Theorem~\ref{thm:gap-two-term} is smaller than the one given by Theorem~\ref{thm:ic-gap-nonrandom}.
This phenomenon can be intuitively explained through \eqref{eq:two-term-alpha}. On one hand, $\lambda_2^2$ contributes to the noise term when extracting the first KPD component, since $\lambda_2^2+\sigma^2$ appears in the denominator in \eqref{eq:two-term-alpha}. On the other hand, over-fitting due to the second Kronecker product reduces $\|\bm Y - \hat{\bm Y}\|_F^2$ under the wrong configuration, which is quantified by $\lambda_2^2\phi_2^2+ 2\lambda_1\lambda_2\phi_1\phi_2\xi$ in the numerator of \eqref{eq:two-term-alpha}.

Similar to Example~\ref{example:normal}, we consider the following example of the two term model under normality.
\begin{eg}
\label{example:two-term-normal}
Consider the two term model \eqref{eq:model-two-term}. Suppose that
\begin{equation*}
  \lambda_k\bm A_k\otimes\bm B_k = \lambda_{k0}\tilde{\bm A}_k\otimes\tilde{\bm B}_k / 2^{(M+N)/2}, \quad k=1,2,
\end{equation*}
where all of the five matrices $\tilde{\bm A_k}$ and $\tilde{\bm B_k}$
and $\bm E$ are independent, and each consisting of IID standard
normal entries. To translate it back into the form of
\eqref{eq:model-two-term}, we let
$\bm A_k=\tilde{\bm A}_k/\|\tilde{\bm A}_k\|_F$,
$\bm B_k=\tilde{\bm B}_k/\|\tilde{\bm B}_k\|_F$, and
$\lambda_k=\lambda_{k0}\cdot\|\tilde{\bm A}_k\|_F\cdot\|\tilde{\bm
  B}_k\|_F/2^{(M+N)/2}$.
\end{eg}

For Example~\ref{example:two-term-normal}, it turns out that with probabilities tending to one, $\xi$ is close to 0 and the representation gaps $\psi_1^2$ and $\psi_2^2$ are close to $1/2$ (due to Corollary~\ref{corr:further-decomp-random}). As an immediate consequence, Theorem~\ref{thm:gap-two-term} yields a information criterion gap of the size
$$\alpha = \ln\left(1+\dfrac{\lambda_{10}^2-\lambda_{20}^2}{2(\sigma^2+\lambda_{20}^2)}\right).$$
However, by a refined analysis of Assumption~5 under the normality of Example~\ref{example:two-term-normal}, we are able to prove the following improved result.
\begin{corr}\label{corr:two-term-random-scheme}
  Consider Example~\ref{example:two-term-normal}. Under Assumptions~\ref{assump:hd} and \ref{assump:Emat}, Theorem~\ref{thm:gap-two-term} holds with the information criterion gap
  $$\alpha = \ln\left(1+\dfrac{\lambda_{10}^2}{2(\sigma^2+\lambda_{20}^2)}\right).$$
\end{corr}

\section{Examples}\label{sec:numerical}

We illustrate the performance of the estimation and configuration
selection procedure through simulation studies in
Section~\ref{sec:simu}, and image examples in
Section~\ref{sec:example}.

\subsection{Simulations}
\label{sec:simu}

We design two simulation studies: the first one on the performance of the
estimation procedure introduced in Section~\ref{sec:estimation}, and the second
one on the configuration selection proposed in
Section~\ref{sec:estimation-and-config}. Many implications of the theoretical results in Section~\ref{sec:theory}
surface from the outcomes of the numerical studies.

\subsubsection{Estimation with known configuration}
\label{sec:simu_estimation}

We first consider the performance of the estimators of $\lambda$,
$\bm A$ and $\bm B$ given in \eqref{eq:nkp-est}, when the true
configuration $(m_0,n_0)$ is known. Throughout this subsection the
simulations are based on model \eqref{eq:model-normalized} with
$m_0=5,\,n_0=5,\,M=10,\,N=10$ and $\sigma=1$. 

The model \eqref{eq:model-normalized} after the rearrangement under
the true configuration becomes
$$\mathcal R_{m_0, n_0}[\bm Y] = \lambda \vec(\bm A)\vec(\bm B)' + \sigma 2^{-(M+N)/2}\mathcal R_{m_0, n_0}[\bm E],$$
where $\vec(\bm A)\in\mathbb R^{2^{m_0+n_0}}$, $\vec(\bm B)\in\mathbb R^{2^{m_0^\dagger+n_0^\dagger}}$ are unit vectors.
Without loss of generality, set $\vec(\bm A) = (1, 0, \dots, 0)'$,
$\vec(\bm B) = (1, 0, \dots, 0)'$.
In this experiment, the noise level is fixed at $\sigma=1$, so the
signal-to-noise ratio is controlled by $\lambda$, which takes values
from the set $\{e^1, e^2, \dots, e^{16}\}$. 
For each value of $\lambda$, we
calculate the errors of the corresponding estimators $\hat\lambda$,
$\hat{\bm A}$ and $\hat{\bm B}$ by
$$\ln \left(\dfrac{\hat\lambda}{\lambda} - 1\right)^2\quad\text{and}\quad
\ln\|\hat{\bm A} - \bm A\|_F^2+\ln\|\hat{\bm B} - \bm B\|_F^2
.$$ The errors based on 20 repetitions are reported in Figure
\ref{fig:error_true_config}.
\begin{figure}[!hptb]
    \centering
    \includegraphics[width=0.45\textwidth]{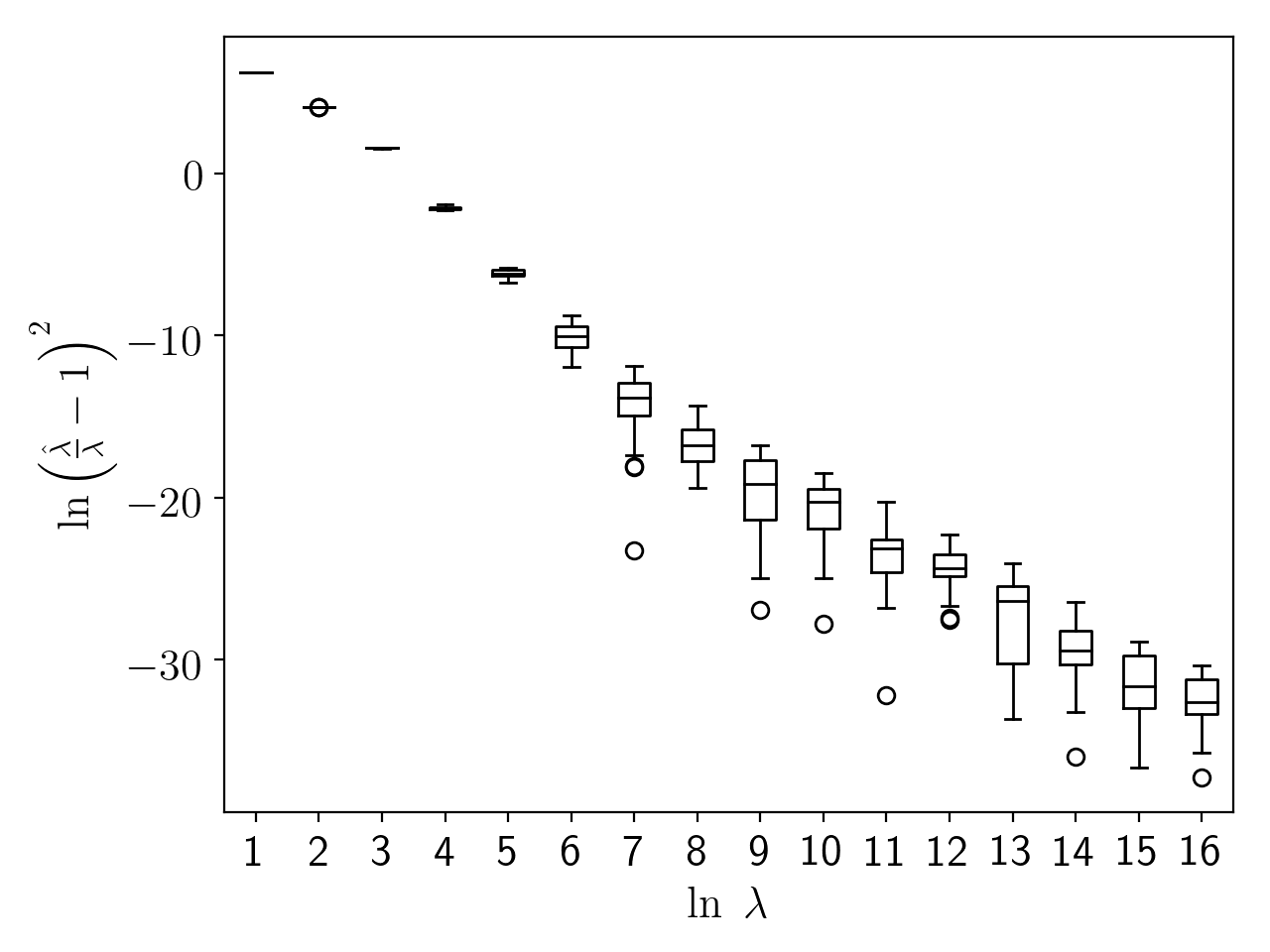}
    \includegraphics[width=0.45\textwidth]{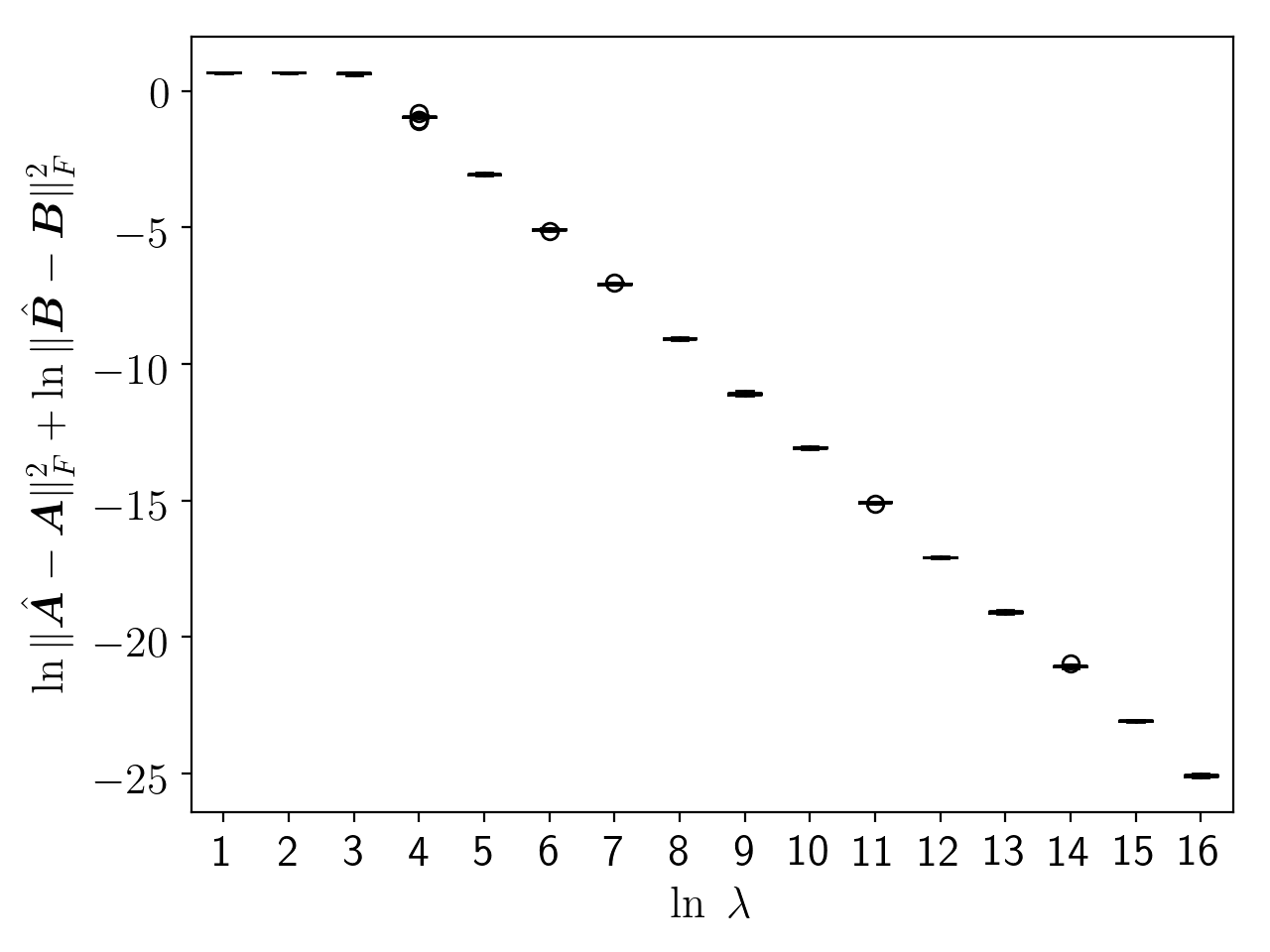}
    \caption{Boxplots for errors in $\hat\lambda$, $\hat{\bm A}$ and
      $\hat{\bm B}$ against the signal-to-noise ratio.}
    \label{fig:error_true_config}
\end{figure}

Figure \ref{fig:error_true_config} displays an interesting linear
pattern, that is, as the signal-to-noise ratio increases,
$\ln \left(\frac{\hat\lambda}{\lambda}-1\right)^2$ is approximately
linear against $\ln \lambda$ with a slope around $-2$, and so is the
error $\ln(\|\hat{\bm A} - \bm A\|_F^2\|\hat{\bm B} - \bm B\|_F^2)$
for the matrix estimators. We note that this pattern is consistent
with Theorem~\ref{thm:truth}, which asserts that
$$\dfrac{\hat\lambda}{\lambda}-1=O_p\left( \dfrac{1}{\lambda}\right)\quad \text{and}\quad \|\hat{\bm A} - \bm A\|_F\|\hat{\bm B} - \bm B\|_F=O_p\left( \dfrac{1}{\lambda}\right),$$
since $r_0$ defined in Theorem~\ref{thm:truth} remains a constant here as we vary the signal strength
$\lambda$ in the simulation.

\subsubsection{Configuration Selection}
\label{sec:simu_config}

We now demonstrate the performance of the information criterion based
procedure for selecting the configuration. Two criteria will be
considered: MSE (when $\kappa=0$) and AIC (when $\kappa=2$). Corresponding to the one- and multi-term models considered in Sections~\ref{sec:theory} and \ref{sec:extension}, we carry
out two experiments respectively.\\[1em]
\noindent\textbf{Experiment 1: One-term KoPA model}\\[1em]
The simulation is based on model \eqref{eq:model-normalized}. Two configurations
are considered: (i) $M=N=9$,
$m_0=4,\ n_0=4$, and (ii) $M=N=10$, $m_0 = 5,\ n_0=4$. Similar to Section~\ref{sec:simu_estimation}, the noise
level is fixed at $\sigma=1$, so the signal-to-noise ratio is
controlled by $\lambda$. To control the representation gap $\psi^2$, 
we construct the matrices $\bm A$ and $\bm B$ as follows:
\begin{align*}
    \bm A & = \sqrt{\varphi^2}\begin{bmatrix}
1\\ 0
\end{bmatrix}\otimes \bm D_1+ \sqrt{1-\varphi^2}\begin{bmatrix}
0\\ 1
\end{bmatrix}\otimes \bm D_2, \\
\bm B & = \sqrt{\varphi^2}\begin{bmatrix}
1\\ 0
\end{bmatrix}\otimes \bm D_3+ \sqrt{1-\varphi^2}\begin{bmatrix}
0\\ 1
\end{bmatrix}\otimes \bm D_4,
\end{align*}
where 
$\vec(\bm D_i), i=1,2,3,4$ are independent random unit vectors such that $\vec(\bm D_1)$ and $\vec(\bm D_2)$ are orthogonal, and so are $\vec(\bm D_3)$ and $\vec(\bm D_4)$.
In the experiment, five values of $\varphi^2$ are considered: $\varphi^2\in\{0.1, 0.2, 0.3, 0.4, 0.5\}$. We remark that the construction above controls the representation gaps for configurations $(1, 0)$ and $(m_0+1, n_0)$ at $\varphi^2$ exactly, and the representation gaps for configurations with $m+n\in\{1, M+N-1\}$ (close to trivial configurations) or $|m-m_0|+|n-n_0|=1$ (close to the true configuration) at roughly 0.5. Consequently, when $\varphi^2=0.1, 0.2, 0.3, 0.4$, the overall representation gap $\psi^2$ is at the desired level $\varphi^2$ with high probabilities. But when $\varphi^2=0.5$, the representation gap $\psi^2$ can be slightly smaller than $0.5$.

In Figure~\ref{fig:occurrence_one-term}, we plot the empirical frequencies of the correct configuration selection, out of 100 repetitions, against the signal-to-noise ratio $\lambda/\sigma$. Note that the x-axis scale in Sub-figures~\ref{fig:occurrence_aic_9} and \ref{fig:occurrence_aic_10}
is different from that in \ref{fig:occurrence_mse_9} and \ref{fig:occurrence_mse_10}.  The performances of both MSE ($\kappa = 0$) and AIC ($\kappa = 2$) are illustrated. BIC ($\kappa=(M+N)\ln 2$) has a very similar performance to AIC, and is not reported here. 



\begin{figure}[htb]
    \centering
    \begin{subfigure}{.45\textwidth}
    \centering
    \includegraphics[width=.8\linewidth]{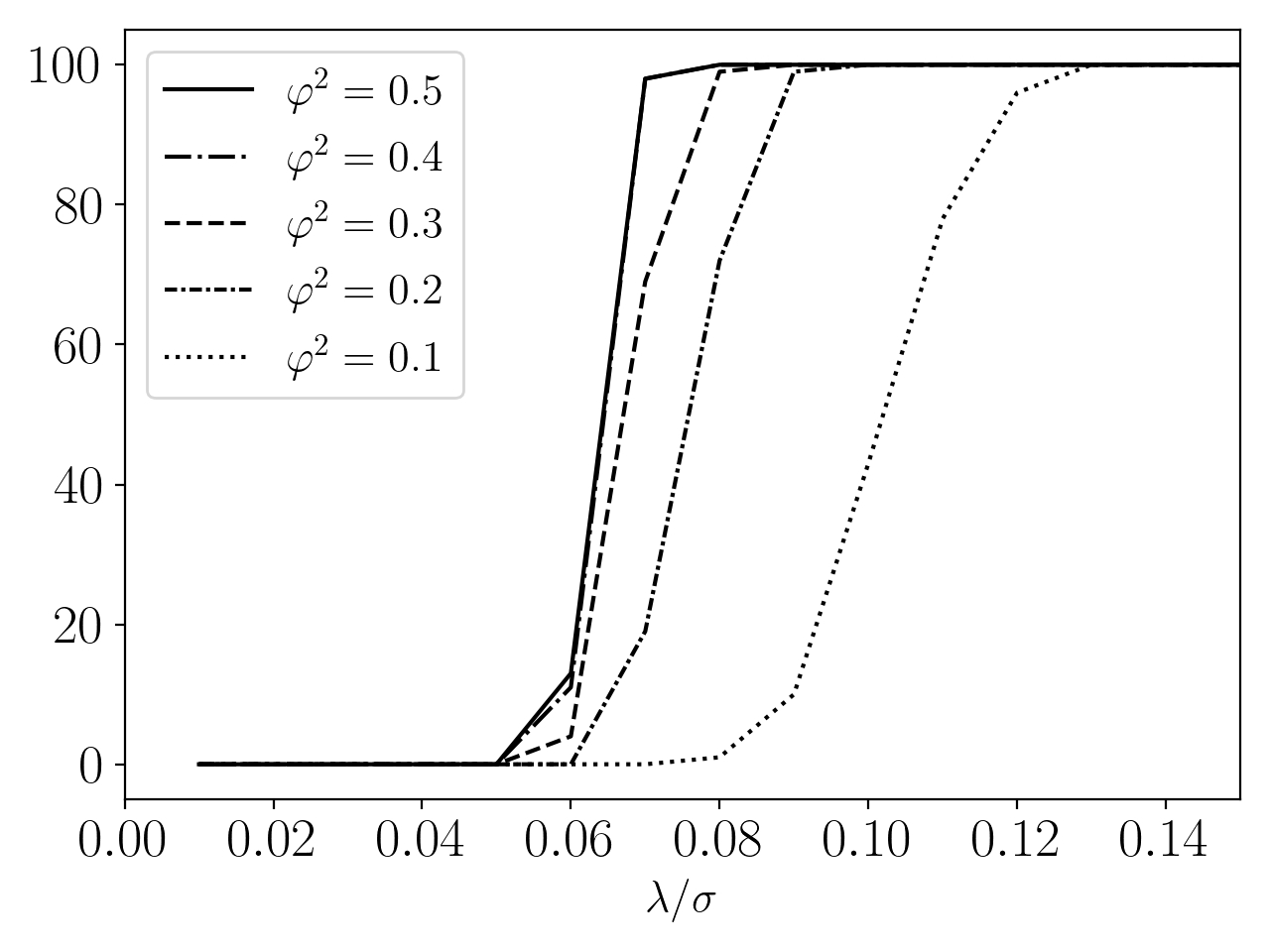}\\[-2mm]  
    \caption{$M=N=9$, AIC}
    \label{fig:occurrence_aic_9}
    \end{subfigure}
    \begin{subfigure}{.45\textwidth}
    \centering
    \includegraphics[width=.8\linewidth]{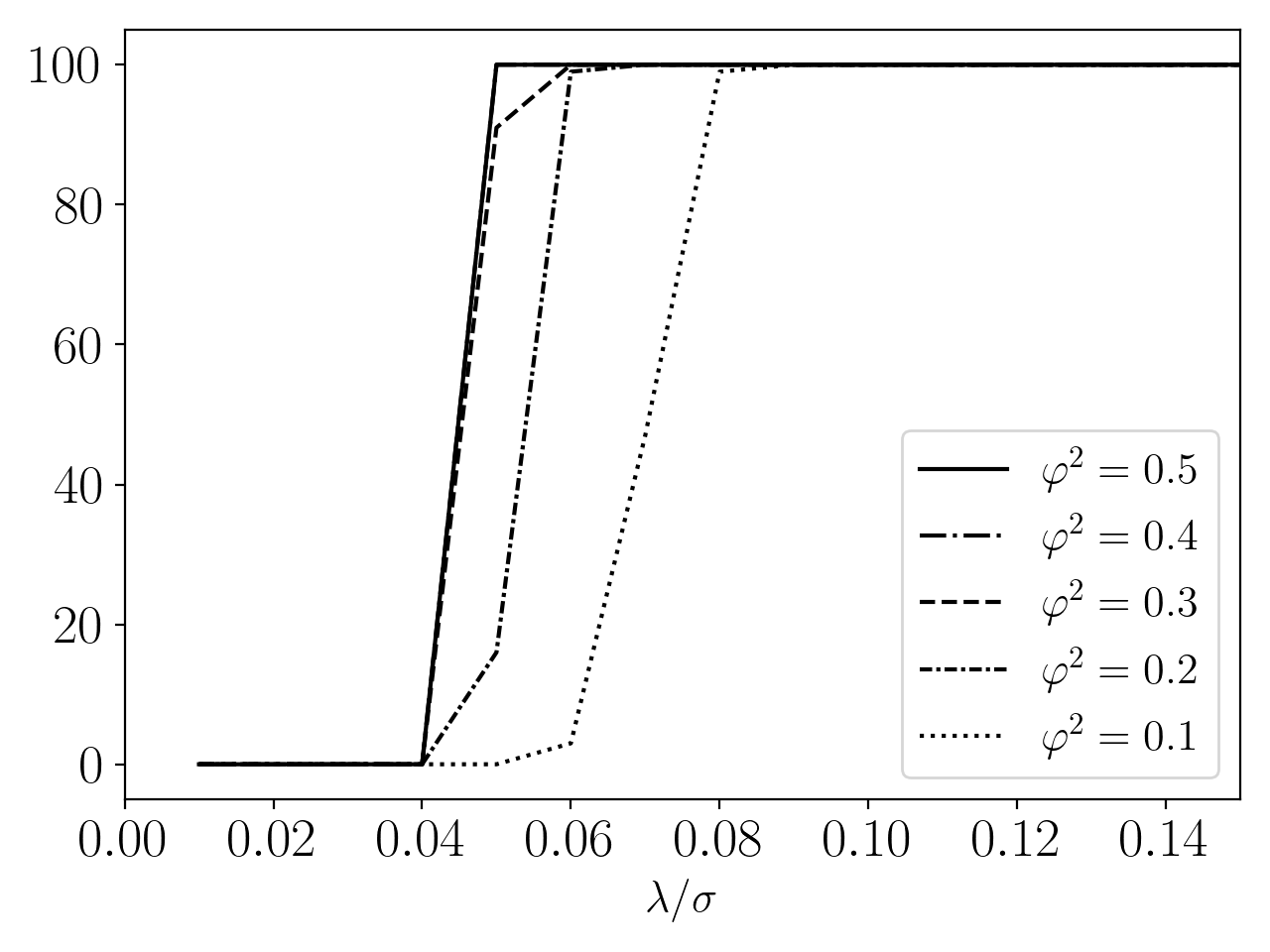} \\[-2mm] 
    \caption{$M=N=10$, AIC}
    \label{fig:occurrence_aic_10}
    \end{subfigure}\\[3mm]
    \begin{subfigure}{.45\textwidth}
    \centering
    \includegraphics[width=.8\linewidth]{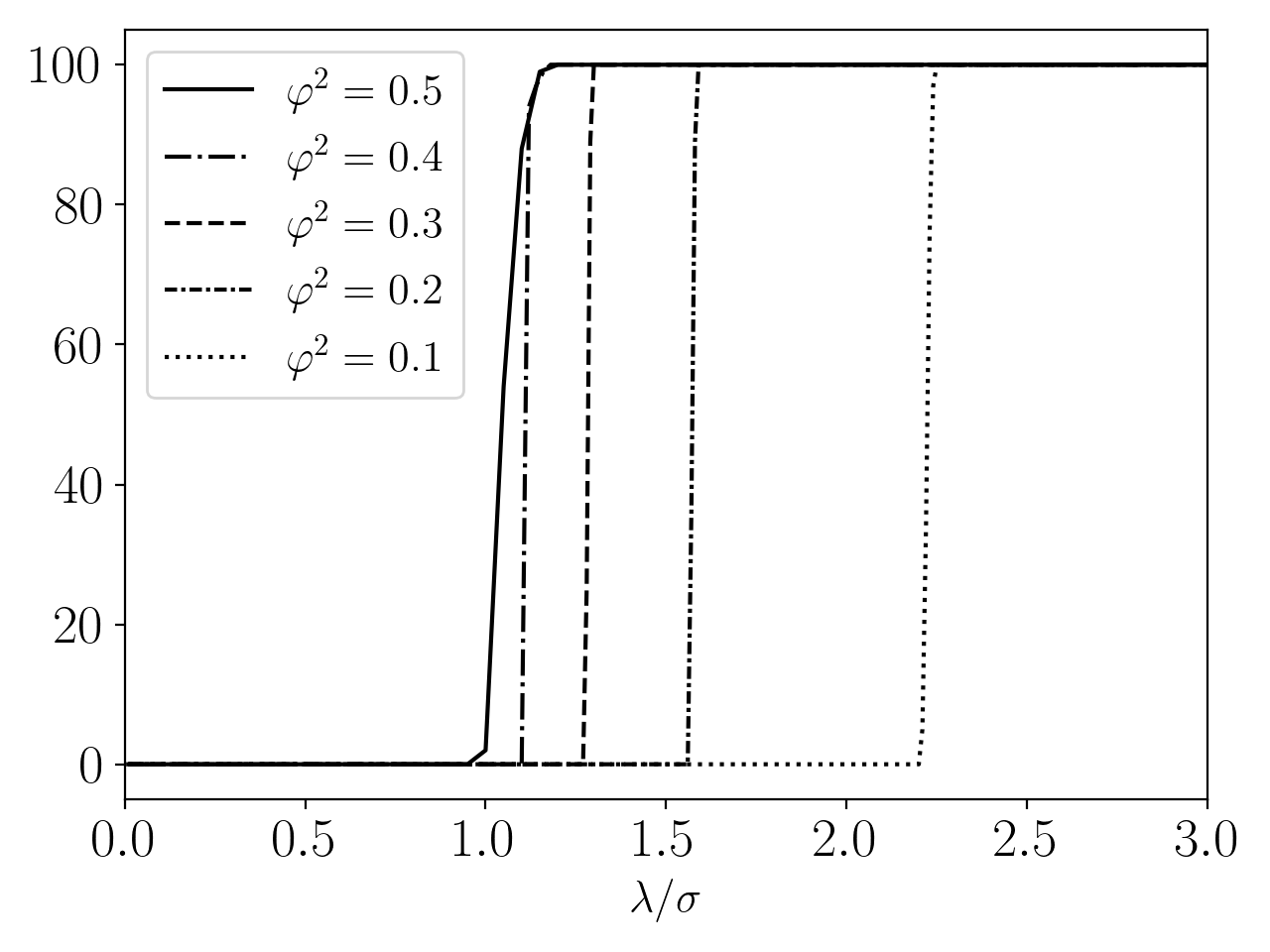} \\[-2mm] 
    \caption{$M=N=9$, MSE}
    \label{fig:occurrence_mse_9}
    \end{subfigure}
    \begin{subfigure}{.45\textwidth}
    \centering
    \includegraphics[width=.8\linewidth]{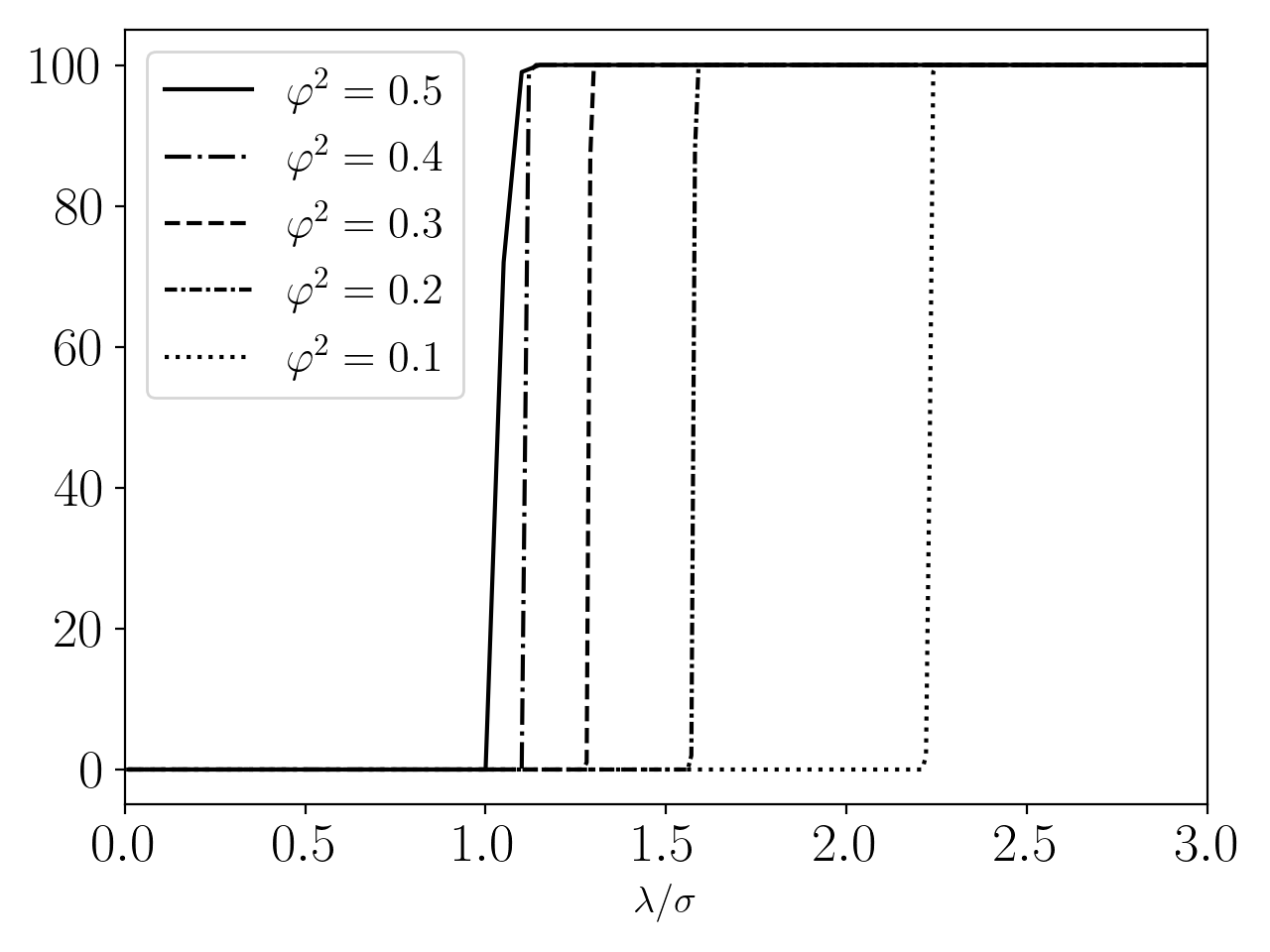}  \\[-2mm]
    \caption{$M=N=10$, MSE}
    \label{fig:occurrence_mse_10}
    \end{subfigure}
    \caption{The empirical frequencies of the correct configuration selection out of 100 repetitions.
    }
    \label{fig:occurrence_one-term}
\end{figure}

For extremely weak signal-to-noise ratio $\lambda \leqslant 0.03$, neither
of MSE and AIC is able to select the true configuration with a
high probability, for both configurations.
This does not
contradict with Theorem~\ref{thm:consistency_nonrandom}. When the signal is very
weak, larger dimensions of the observed matrix $\bm Y$ are required
for the consistency. As the signal-to-noise ratio increases from $0.01$
to $0.13$, the probability that the true configuration is selected
increases gradually and eventually gets very close to one for AIC as shown in Figures~\ref{fig:occurrence_aic_9} and \ref{fig:occurrence_aic_10}. We also note that the performance gets better as the representation gap $\psi^2$ increases. These observations are echoing
Theorem~\ref{thm:ic-gap-nonrandom}, which shows that AIC (with $\kappa = 2 > 2\ln 2$) only
requires a minimal condition $(\lambda/\sigma)^2\psi^2 > 0$ to achieve the
consistency, and the separation gap of AIC is a monotone function of $(\lambda/\sigma)^2\psi^2$. On the other hand, the performance of MSE exhibits a phase transition: it only starts to select the true
configuration with a decent probability when the signal-to-noise ratio
$\lambda/\sigma$ exceed a certain threshold. The theoretical asymptotic threshold for MSE is $\lambda/\sigma \geqslant \sqrt{1/(2\psi^2)}$ as discussed in Remark 5. For $\psi^2\in\{0.5, 0.4, 0.3, 0.2, 0.1\}$ used in this simulation, the corresponding thresholds for $\lambda/\sigma$ are $\{1, 1.12, 1.29, 1.58, 2.24\}$, which can be clearly visualized in Figures~\ref{fig:occurrence_mse_9} and \ref{fig:occurrence_mse_10}. 

Comparing Figures~\ref{fig:occurrence_aic_9} with Figures~\ref{fig:occurrence_aic_10}, we see that the empirical frequency curve increases from 0 to 100 much faster when the matrices are larger. This is consistent with Theorem~\ref{thm:ic-gap-nonrandom}, which shows that the probability of correct configuration selection approaches 1 exponentially fast. \\[2em]

\noindent\textbf{Experiment 2: Two-term KoPA model}\\[1em]
In the second experiment, we consider the two-term KoPA model in \eqref{eq:model-two-term} where $\bm A_k$ and $\bm B_k$ are generated under the random scheme in Example~\ref{example:two-term-normal} such that $\psi_1^2\approx 1/2$, $\psi_2^2\approx 1/2$ and $\xi \approx 0$. According to Theorem~\ref{thm:gap-two-term}, besides the signal-to-noise ratio $\lambda_1/\sigma$, the relative strength of the second term $\lambda_2/\lambda_1$ (for the random scheme adopted in this experiment, see Corollary~\ref{corr:two-term-random-scheme}) affects the configuration selection as well. 

In this simulation, we fix the configurations to $M=N=9, (m_0, n_0)=(4, 4)$ and consider four different relative strengths of the second term $\lambda_2^2/\lambda_1^2\in\{0.3, 0.4, 0.5, 0.6\}$. Similar to Experiment 1, we report the empirical frequencies of correct configurations selection of MSE and AIC, out of 100 repetitions, as a function of the signal-to-noise ratio $\lambda_1/\sigma$ in Figure~\ref{fig:occurrence_two-term}.

\begin{figure}[htb]
    \centering
    \begin{subfigure}{.45\textwidth}
    \centering
    \includegraphics[width=.8\linewidth]{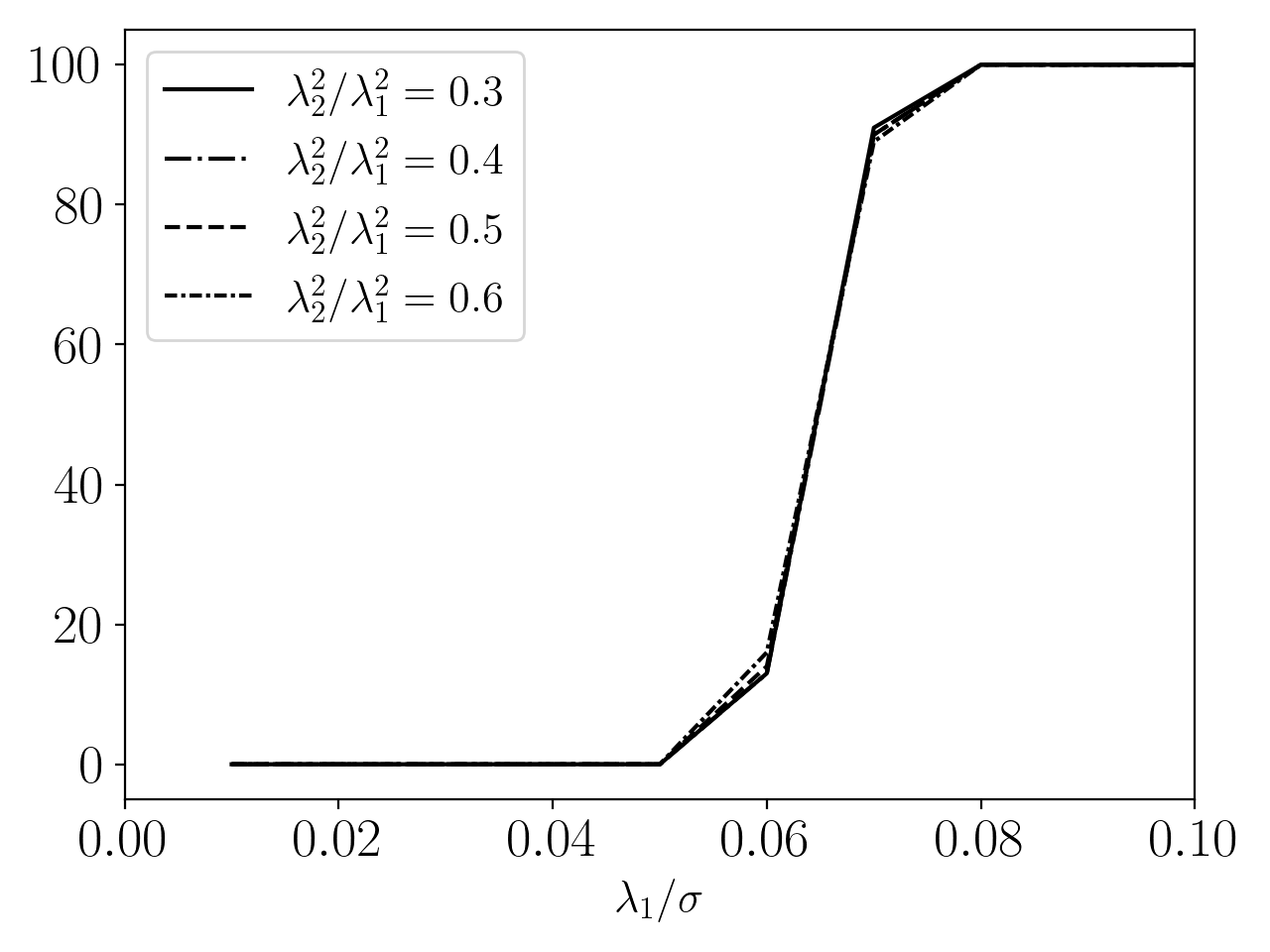}  
    \caption{$M=N=9$, AIC}
    \label{fig:occurrence_two_term_aic_9}
    \end{subfigure}
    \begin{subfigure}{.45\textwidth}
    \centering
    \includegraphics[width=.8\linewidth]{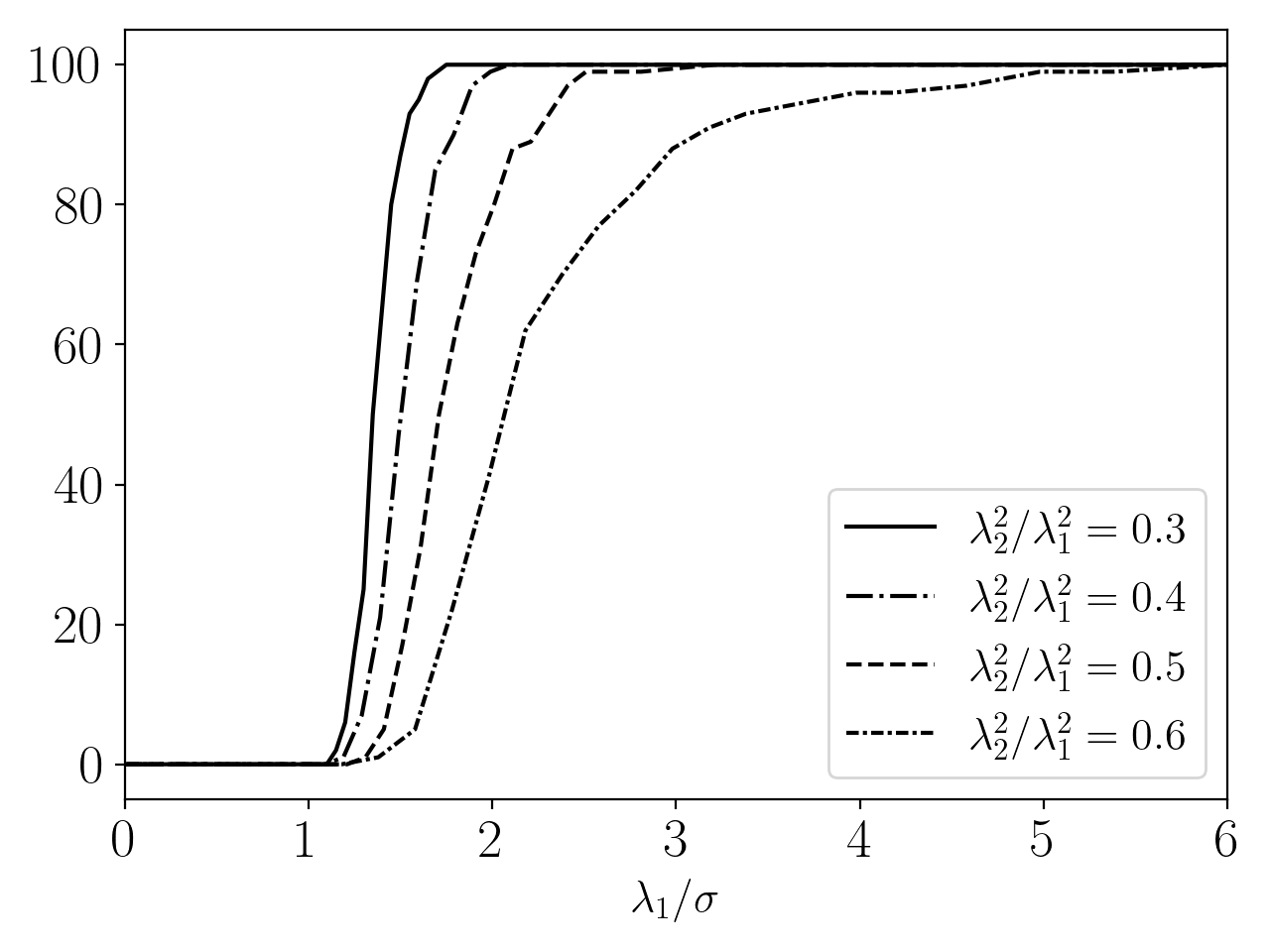}  
    \caption{$M=N=9$, MSE}
    \label{fig:occurrence_two_term_mse_9}
    \end{subfigure}
    \caption{The empirical frequencies of the correct configuration selection out of 100 repetitions in a two-term model.}
    \label{fig:occurrence_two-term}
\end{figure}

Figure~\ref{fig:occurrence_two_term_aic_9} shows that the performance of AIC is in-sensitive to the ratio $\lambda_2^2/\lambda_1^2$ over the experimented range. To the contrary, it is seen from Figure~\ref{fig:occurrence_two_term_mse_9} that MSE performs better when the ratio $\lambda_2^2/\lambda_1^2$ gets smaller, which is consistent with Corollary~\ref{corr:two-term-random-scheme}.


\subsection{Analysis of Image Examples}
\label{sec:example}
\subsubsection{The cameraman's image}\label{sec:cameraman}
In this section we revisit and analyze the cameraman image introduced in
Section~\ref{sec:intro}. The original image, denoted by $\bm Y_0$, has $512\times 512$ pixels.
Each entry of $\bm Y_0$ is a real number between
0 and 1, where 0 codes black and 1 indicates white. The grayscale
cameraman image $\bm Y_0$ is displayed in Figure
\ref{fig:cameraman_intro}.  


Our analysis will be based on the de-meaned version ${\bm Y}$ of the
original image $\bm Y_0$. We demonstrate how well the image $\bm Y$
can be approximated by a Kronecker product or the sum of a few
Kronecker products, and make comparisons with the low rank
approximations given by SVD.

We first consider the configuration selection by MSE, AIC and BIC on the original image $\bm Y$.
Figure~\ref{fig:ic_cameraman} plots the heat maps for the information
criterion $\mathrm{IC}_\kappa(m,n)$ for all candidate configurations in the set
\begin{equation*}
  \mathcal{C}=\{(m,n):\;0\leqslant m,n\leqslant 9\}\setminus\{(0,0),(9,9)\},
\end{equation*}
where the top-left and bottom-right corners are always excluded from
the consideration. Since darker cells correspond to smaller values of
the information criteria, we see that MSE and AIC select the
configuration $(8,9)$, and BIC selects $(6,7)$.

We also observe an overall pattern in Figure~\ref{fig:ic_cameraman}:
configurations with larger $(m, n)$ values are more preferable than
those with smaller $(m, n)$. Note that the Kronecker product does not
commute, and with configuration $(m, n)$ the product 
is a $2^m\times 2^n$ block
matrix, each block of the size $2^{9-m}\times 2^{9-n}$. Real images usually show the locality of pixels in the sense that nearby
pixels tend to have similar colors. Therefore, it can be understood that larger values of $m$ and $n$ are
preferred, since they are better suited to capture the locality. Actually, for the cameraman's image, the configuration $(8, 9)$
accounts for 99.50\% of the total variation of ${\bm Y}$. The penalty
on the number of parameters in AIC is not strong enough to offset the
closer approximation given by the configuration $(8, 9)$. With a
stronger penalty term, BIC selects a configuration that is closer to
the center of the configuration space, involving a much smaller number
of parameters.
\begin{figure}[!tb]
  \centering
  \includegraphics[width=0.3\textwidth]{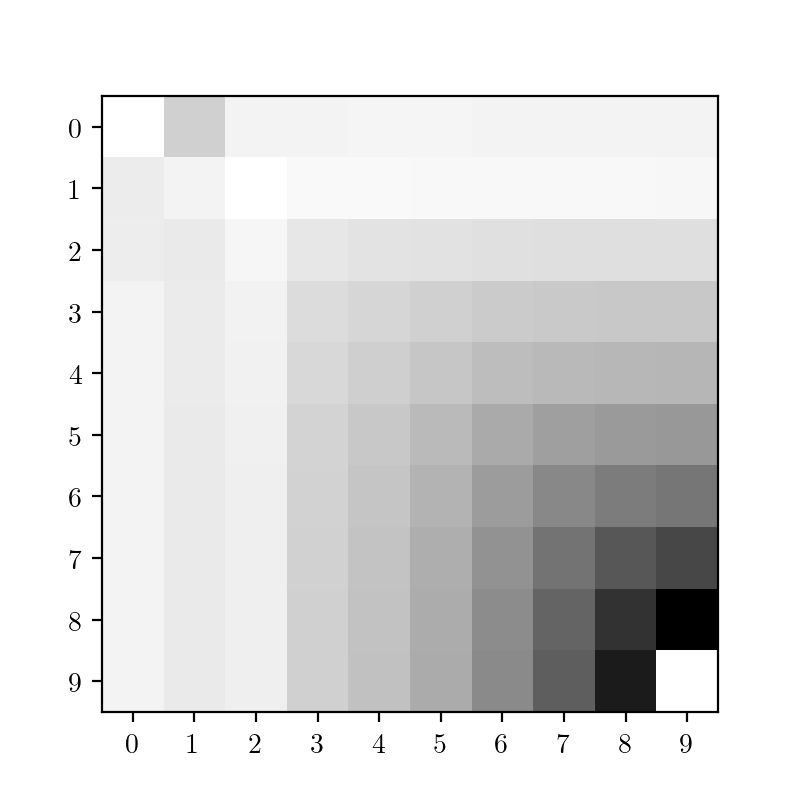}
  \includegraphics[width=0.3\textwidth]{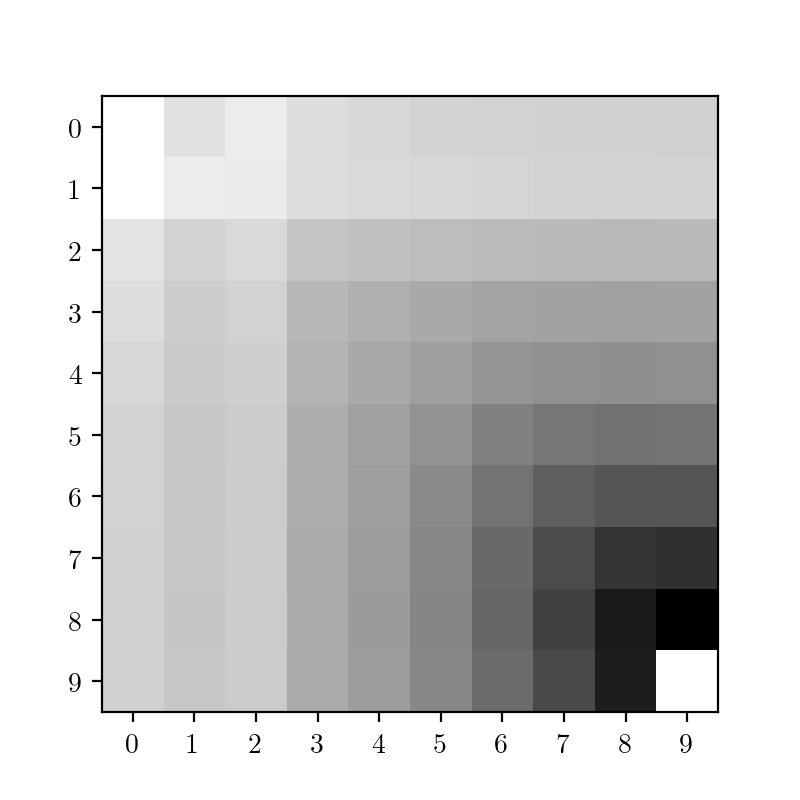}
  \includegraphics[width=0.3\textwidth]{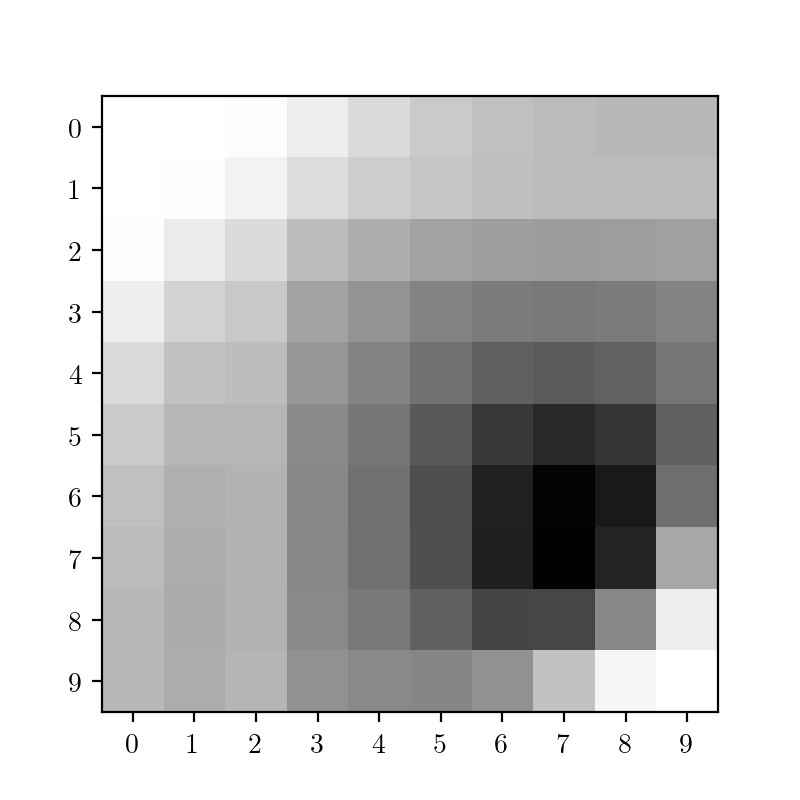}
  \caption{Information Criteria for the cameraman's image. (Left) MSE (Mid) AIC (Right) BIC. Darker color corresponds to lower IC value.}
  \label{fig:ic_cameraman}
\end{figure}

From the perspective of image compressing, KoPA is more flexible than the low
rank approximation, by allowing a choice of the configuration,
and hence a choice of the compression rate. To compare their performances, we use the ratio $\|\hat{\bm Y}\|_F^2/\|\bm Y\|_F^2$ to measure how close the approximation $\hat{\bm Y}$ is to the original image $\bm Y$. In Figure
\ref{fig:compression_cameraman}, these ratios are plotted against the numbers of parameters for the KPD, marked by “+” on the solid line. Since the number of parameters involved in the Kronecker product with configuration $(m,n)$ is $\eta = 2^{m+n} + 2^{M+N-m-n}$, 
the configurations $\{(m,n):\, m+n=c\}$ for any given $0<c<M+N$ have the same number of parameters. Among these configurations, we only plot the one with the largest $\|\hat{\bm Y}\|_F^2/\|\bm Y\|_F^2$.
On the other hand, each cross stands for a rank-$k$ approximation of $\bm Y$, where its value on the horizontal axis is the number of parameters
$$\eta = 1+\sum_{j=1}^k (2^M+2^N - 2j + 1)\quad \text{for }k=1,\dots, 2^{M\wedge N}.$$ 
According to Figure \ref{fig:compression_cameraman}, there always
exists a one-term Kronecker product which provides
a better approximation of the original cameraman's image than the best low rank
approximation involving the same number of parameters.
 
\begin{figure}[!hptb]
    \centering
    \includegraphics[width=0.5\textwidth]{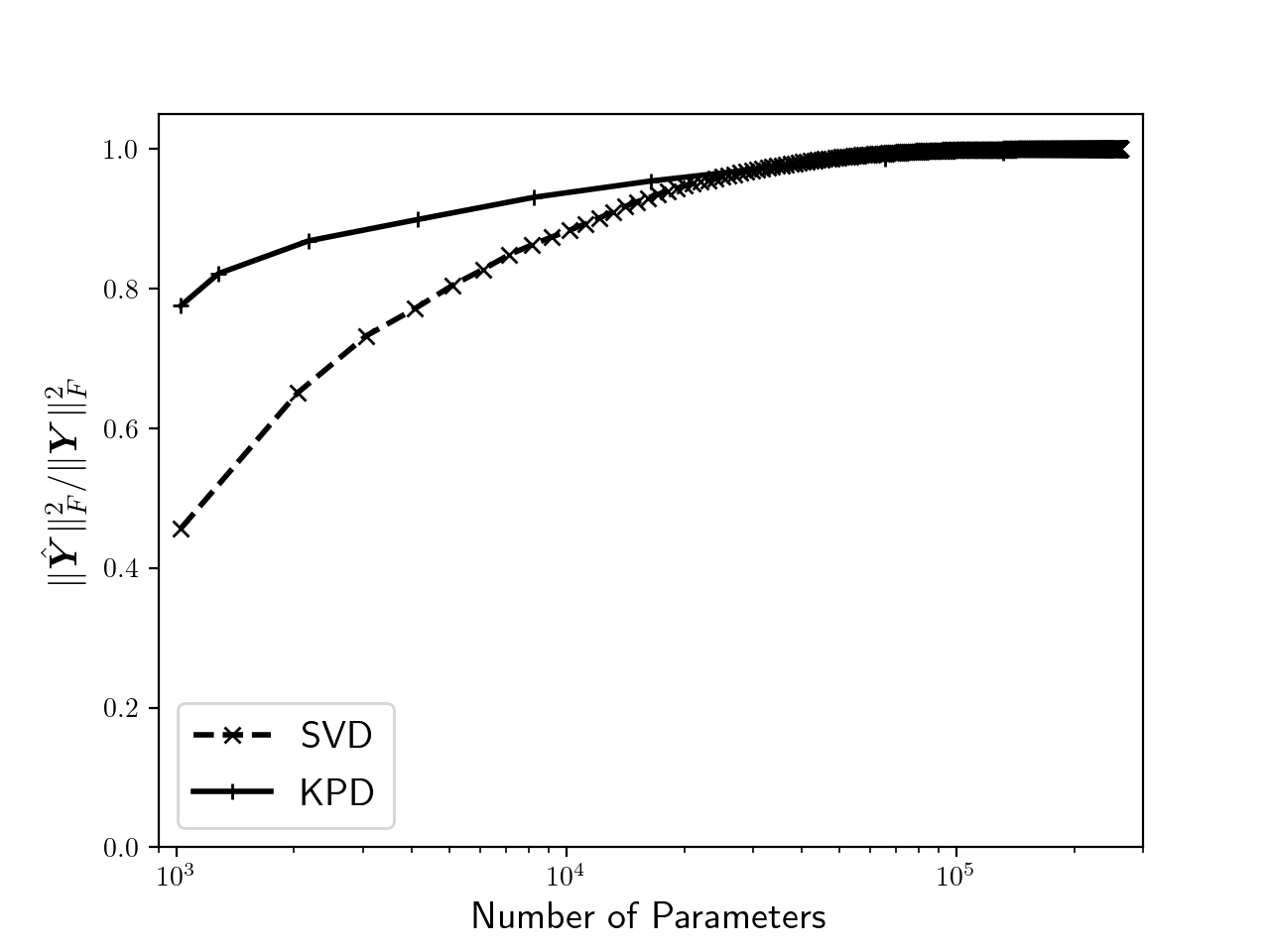}
    \caption{Percentage of variance explained against number of parameters, for KoPA with all configurations, and for low rank approximations of all ranks.}
    \label{fig:compression_cameraman}
\end{figure}

We also consider de-noising the images corrupted by additive Gaussian white noise
$$\bm Y_\sigma = {\bm Y} + \sigma \bm E,$$
where $\bm E$ is a matrix with IID standard normal entries. We
experiment with three levels of corruption:
$\sigma = 0.1, 0.2, 0.3$. Examples of the corrupted images with different $\sigma$
are shown in Figure \ref{fig:blurred_cameraman} with the values rescaled to $[0, 1]$ for plotting.

\begin{figure}[!tb]
  \centering
  \includegraphics[width=0.3\textwidth]{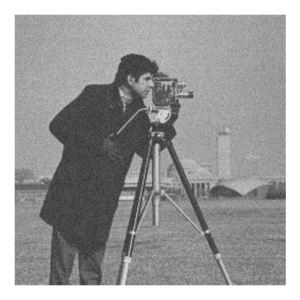}
  \includegraphics[width=0.3\textwidth]{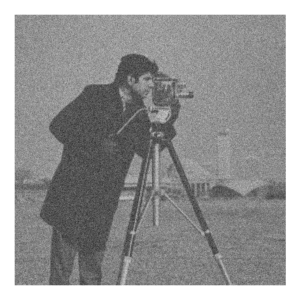}
  \includegraphics[width=0.3\textwidth]{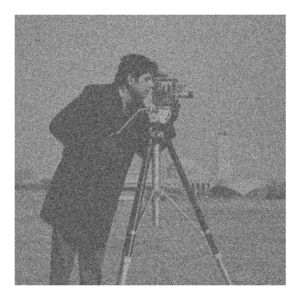}
  \caption{Noisy cameraman's images when (Left) $\sigma=0.1$ (Mid) $\sigma=0.2$ (Right) $\sigma=0.3$}
  \label{fig:blurred_cameraman}
\end{figure}
For the corrupted images, the information criteria $\mathrm{IC}_\kappa(m,n)$ are
calculated, and the corresponding heat maps are plotted in
Figure~\ref{fig:ic_blurred_cameraman}. With added noise, AIC and BIC tend
to select configurations in the middle of the configuration space.
 
 \begin{figure}[!hptb]
     \centering
     \begin{tabular}{cccc}
          &MSE & AIC & BIC\\
         \rotatebox{90}{$\sigma = 0.1$} & 
         \raisebox{-.5\height}{\includegraphics[width=0.3\textwidth]{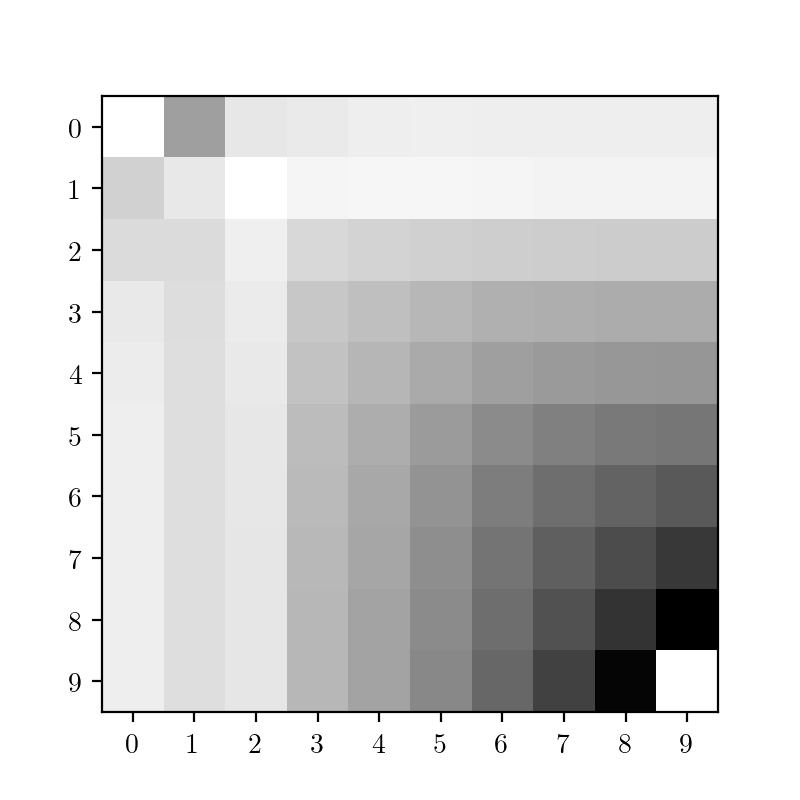}}&
         \raisebox{-.5\height}{\includegraphics[width=0.3\textwidth]{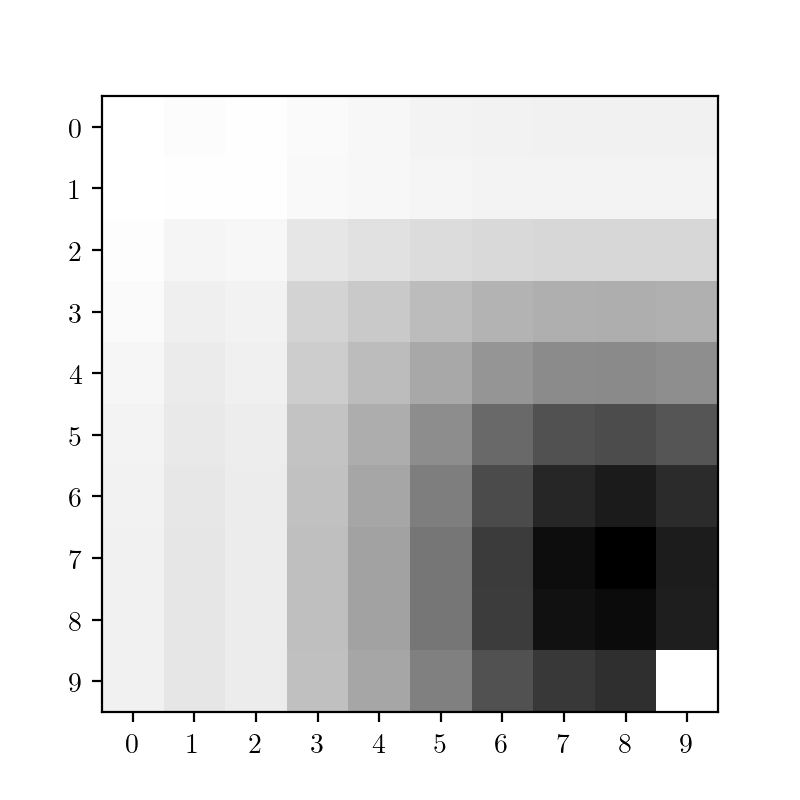}}&
         \raisebox{-.5\height}{\includegraphics[width=0.3\textwidth]{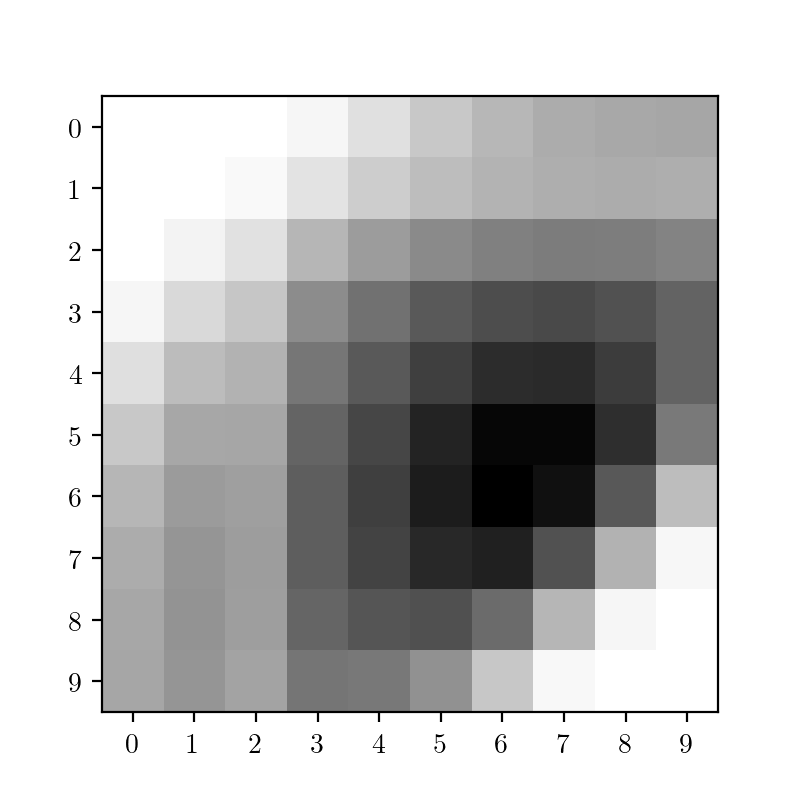}}\\
         \rotatebox{90}{$\sigma = 0.2$} & 
         \raisebox{-.5\height}{\includegraphics[width=0.3\textwidth]{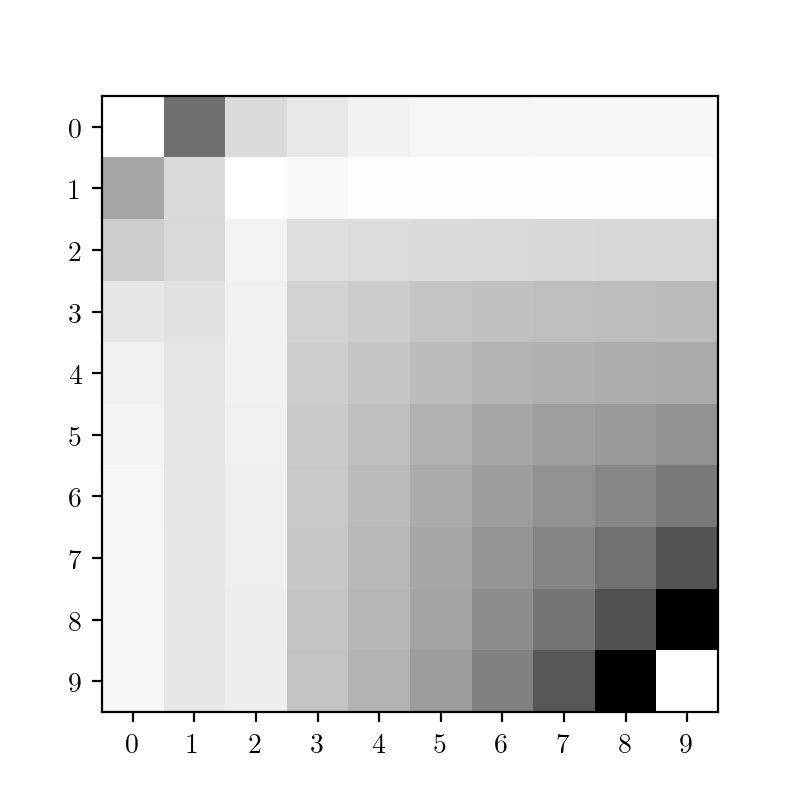}}&
         \raisebox{-.5\height}{\includegraphics[width=0.3\textwidth]{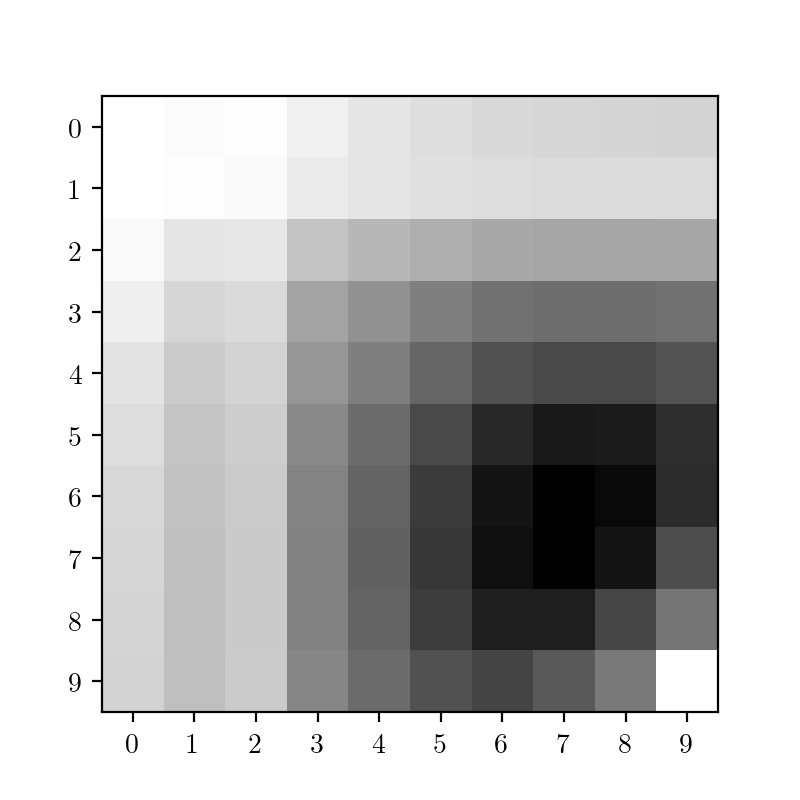}}&
         \raisebox{-.5\height}{\includegraphics[width=0.3\textwidth]{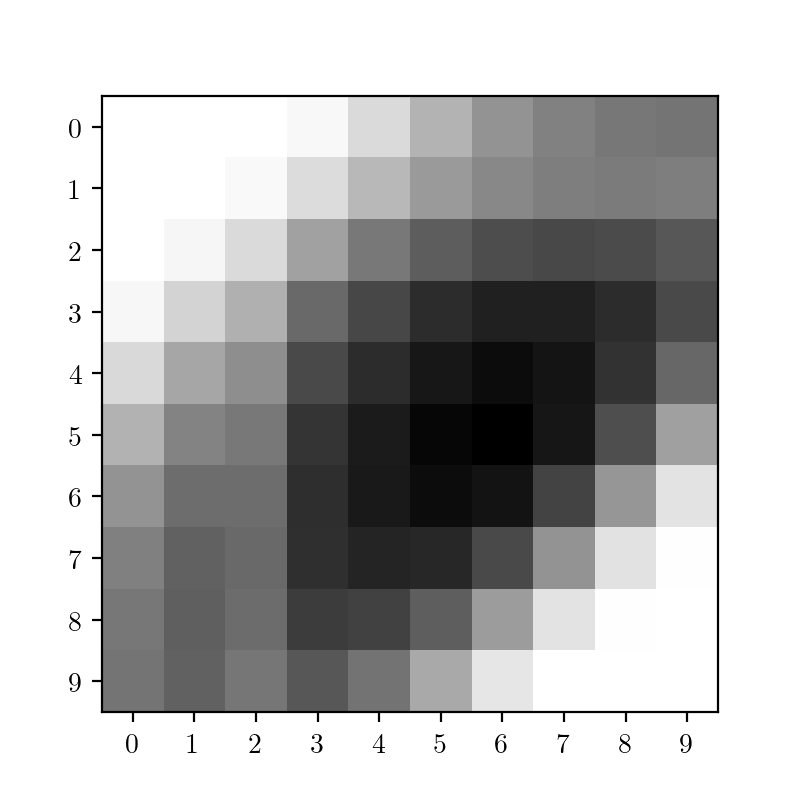}}\\
         \rotatebox{90}{$\sigma = 0.3$} & 
         \raisebox{-.5\height}{\includegraphics[width=0.3\textwidth]{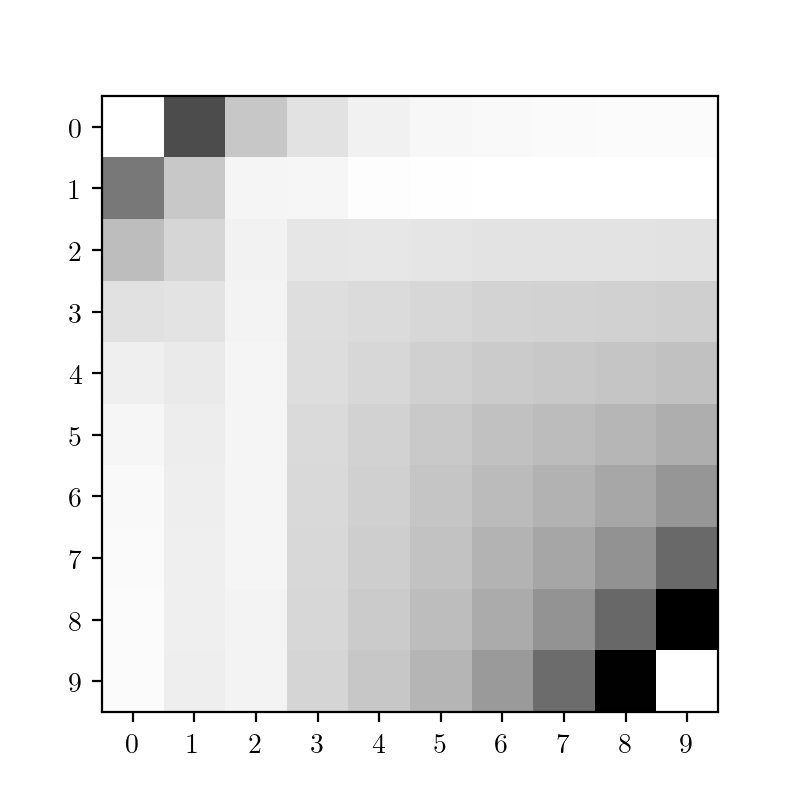}}&
         \raisebox{-.5\height}{\includegraphics[width=0.3\textwidth]{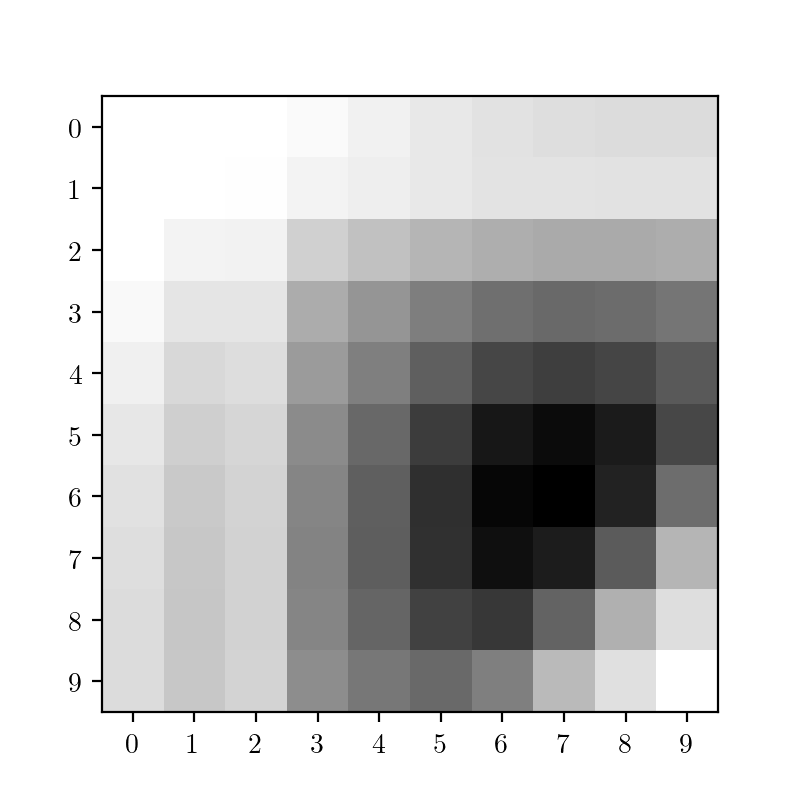}}&
         \raisebox{-.5\height}{\includegraphics[width=0.3\textwidth]{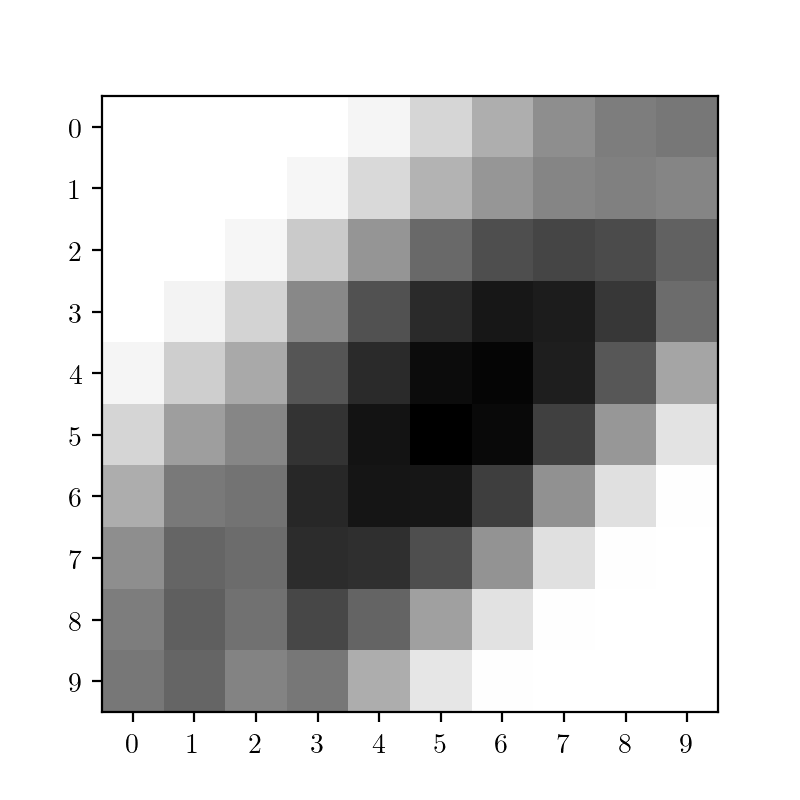}}
     \end{tabular}
     \caption{Heat maps for three different information criteria for the camera's images with different noise levels. Darker color means lower IC value.}
     \label{fig:ic_blurred_cameraman}
 \end{figure}
 
 Now we consider multi-term Kronecker approximation. Following the
 discussion in Section \ref{sec:extension}, for each of three corrupted
 images $\bm Y_\sigma$, we use the configuration selected by BIC in
 Figure \ref{fig:ic_blurred_cameraman}.
 \begin{figure}[!tb]
     \centering
     \begin{tabular}{cccc}
     & $\sigma = 0.1$ & $\sigma = 0.2$ & $\sigma = 0.3$\\
     \rotatebox{90}{KoPA} & \raisebox{-.5\height}{\includegraphics[width=0.3\textwidth]{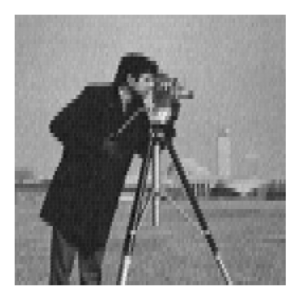}}&
     \raisebox{-.5\height}{\includegraphics[width=0.3\textwidth]{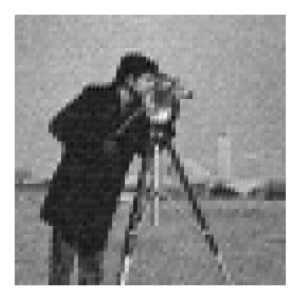}}&
     \raisebox{-.5\height}{\includegraphics[width=0.3\textwidth]{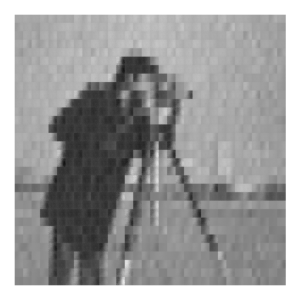}}\\
     \rotatebox{90}{SVD}& \raisebox{-.5\height}{\includegraphics[width=0.3\textwidth]{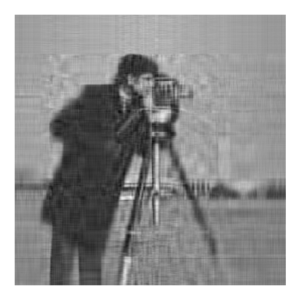}}&
     \raisebox{-.5\height}{\includegraphics[width=0.3\textwidth]{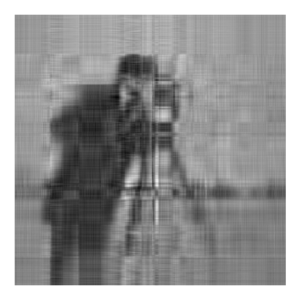}}&
      \raisebox{-.5\height}{\includegraphics[width=0.3\textwidth]{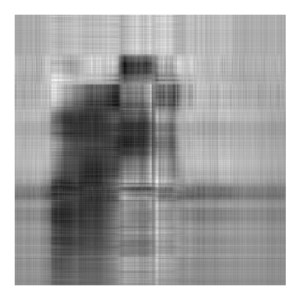}}
     \end{tabular}
     \caption{The fitted image given by multi-term KoPA, and the
       SVD approximation with similar number of parameters. 
       }
     \label{fig:two-term}
 \end{figure}
 Specifically, configurations $(6, 6)$, $(5, 6)$ and $(5, 5)$ are
 selected when $\sigma = 0.1$, 0.2 and 0.3, respectively. A two-term
 Kronecker product model \eqref{eq:model-two-term} is then fitted
 under the selected configuration. The fitted images are plotted in
 the upper panel of Figure \ref{fig:two-term}. Each of them is compared with the image obtained by the low rank approximation involving a similar number of parameters as the two-term KoPA.
 From Figure~\ref{fig:two-term}, it
 is quite evident that the details can easily be recognized from the images reconstructed by the two-term KoPA, but can hardly be perceived in those given by the low rank approximation.
 
 Finally, we examine the reconstruction error defined by
 $$\dfrac{\|{\bm Y} - \hat{\bm Y}\|_F^2}{\|{\bm Y}\|_F^2},$$
 where $\bm Y$ is the original image and $\hat{\bm Y}$ is the one reconstructed from $\bm Y_\sigma$.
 For each of the three noisy images, we continue to use the
 configuration selected by BIC. With fixed configurations, we keep increasing the
 number of terms in the KoPA until $\bm Y_\sigma$ is fully fitted, and plot the
 corresponding reconstruction error against the number of parameters in
 Figure~\ref{fig:error-kpd-svd}. It has the familiar ``U'' shape, showing 
 the trade-off between estimation bias and variation. A similar curve is given for the low rank
 approximations exhausting all possible ranks. From
 Figure~\ref{fig:error-kpd-svd}, it is seen that the multi-term KoPA constantly
 outperforms the low rank approximation at any given number of
 parameters. Furthermore, the minimum reconstruction error that KoPA
 can reach is always smaller than that given by the low rank
 approximation.
 
 \begin{figure}[!tb]
     \centering
     \includegraphics[width=0.3\textwidth]{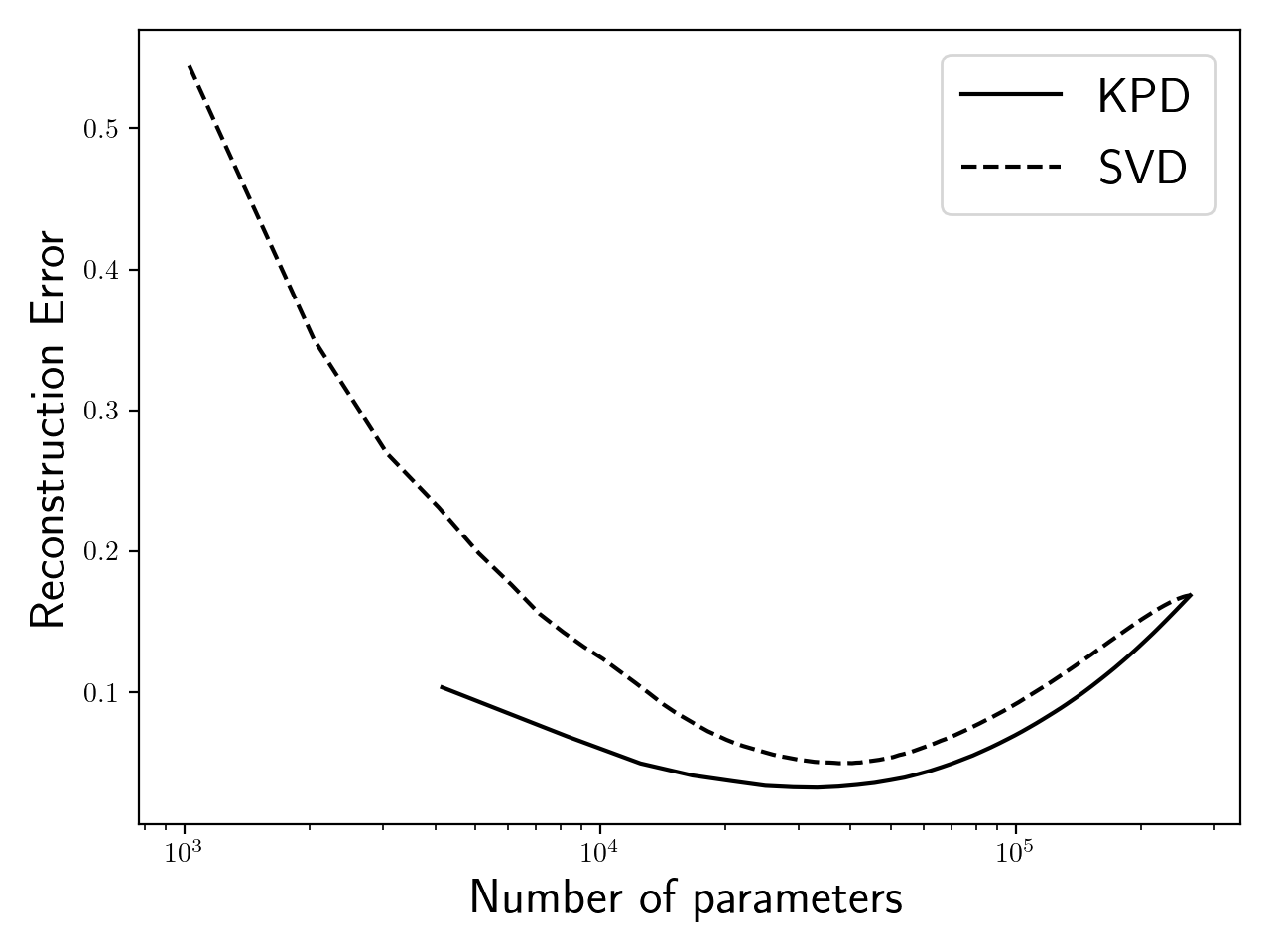}
     \includegraphics[width=0.3\textwidth]{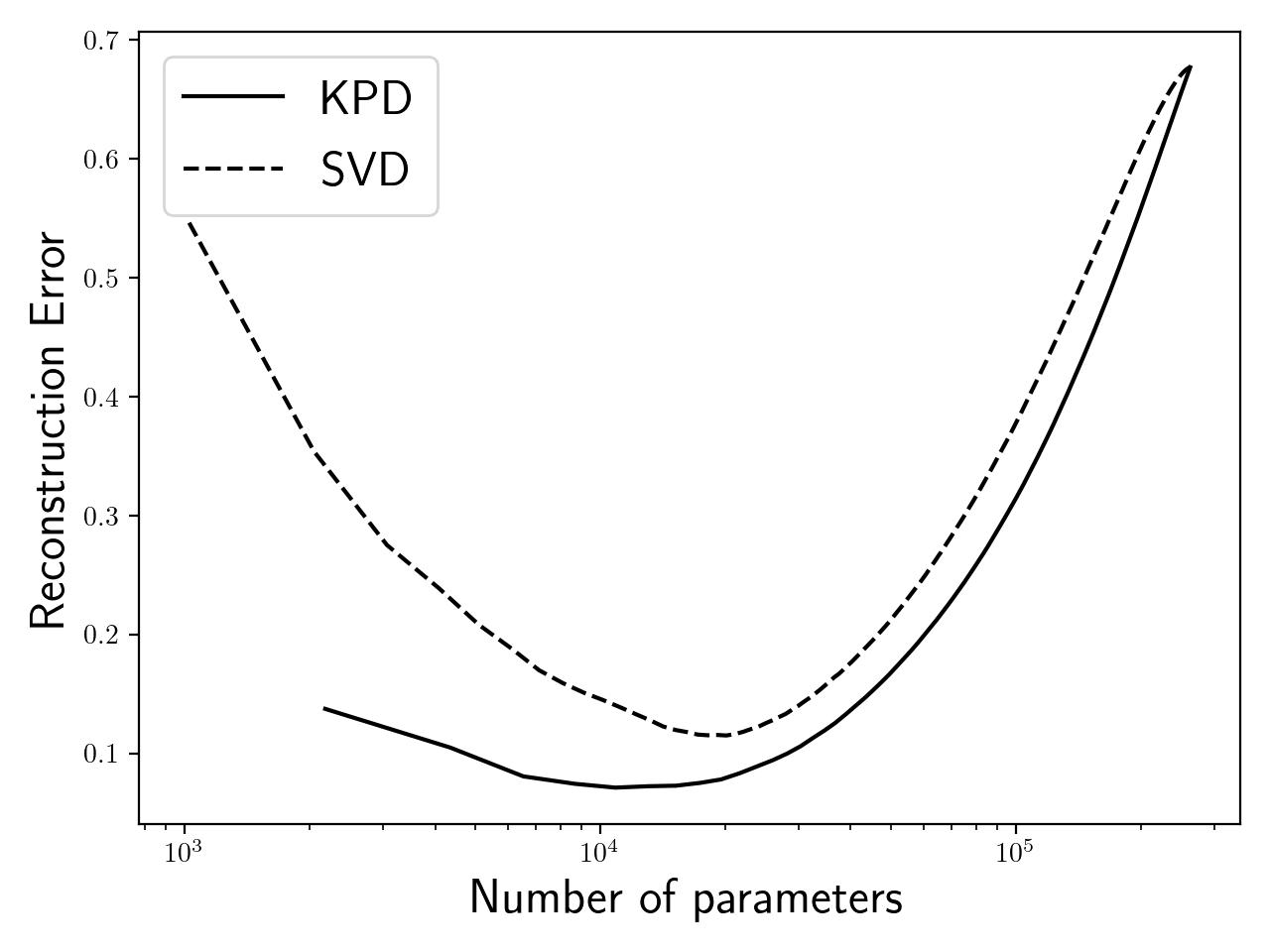}
     \includegraphics[width=0.3\textwidth]{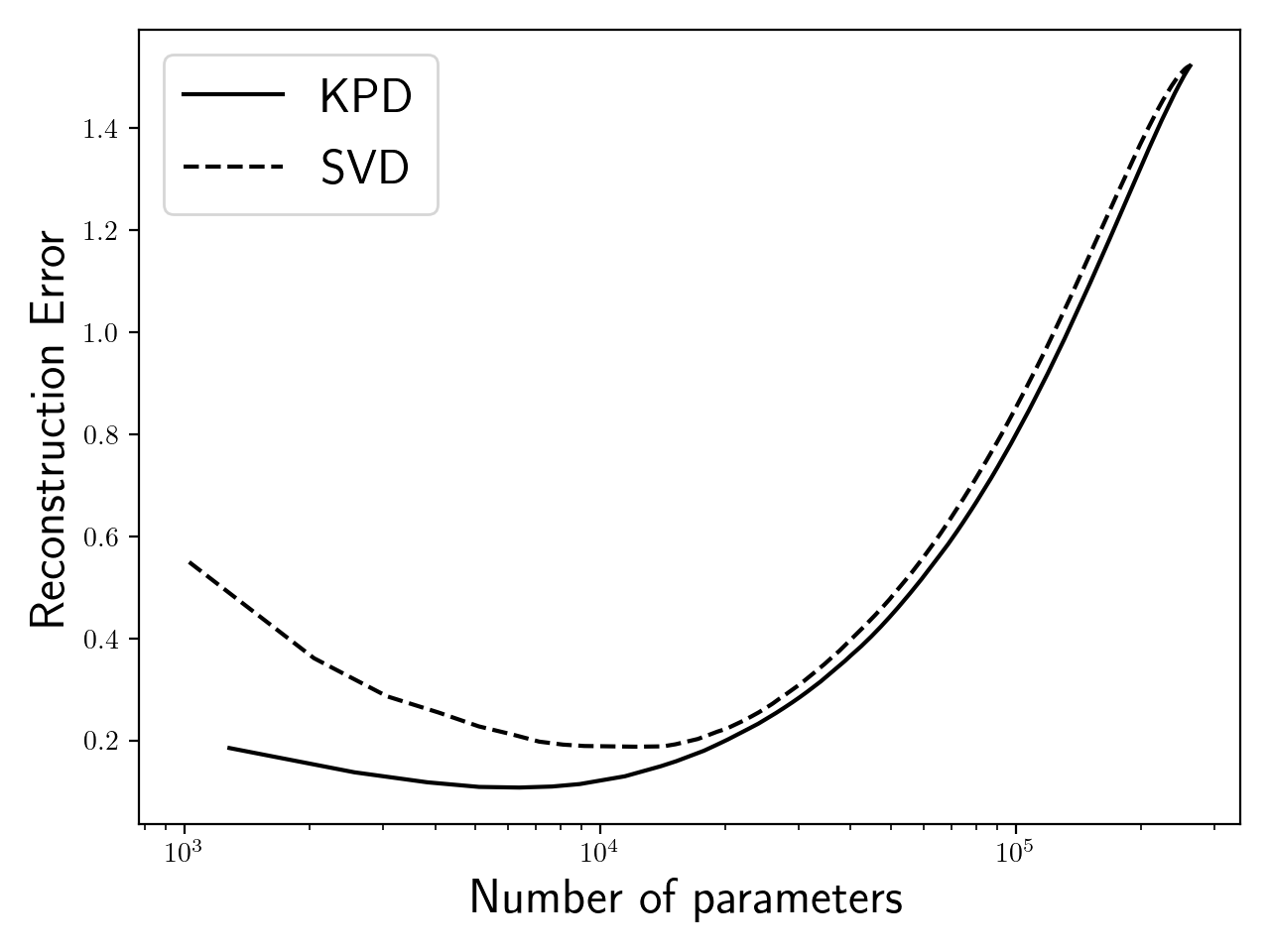}
     \caption{Reconstruction error against the number of parameters
       for KoPA and low rank approximations. The three panels from
       left to right correspond to $\sigma = 0.1$, $\sigma=0.2$ and
       $\sigma=0.3$ respectively.}
     \label{fig:error-kpd-svd}
 \end{figure}
 
 \subsubsection{More images}
 To assess the performance of KoPA model in image denoising, we repeat the experiment in Section~\ref{sec:cameraman} to a larger set of test images. The 10 test images printed in Figure~\ref{fig:test-images} are collected from Image Processing Place\footnote{http://www.imageprocessingplace.com/root\_files\_V3/image\_databases.htm} and The Waterloo image Repository\footnote{http://links.uwaterloo.ca/Repository.html}. Each of the 10 test images is a $512\times 512$ gray-scaled matrix, same as the cameraman's image. We corrupt the test image with additive Gaussian noise, whose amplitude is 0.5 times the standard deviation of all its pixel values:
 $$\bm Y_\sigma = \bm Y + 0.5\cdot \mathrm{std}(\bm Y)\cdot\bm E.$$
 
 \begin{figure}[!tb]
 \centering
 \includegraphics[width=0.19\textwidth]{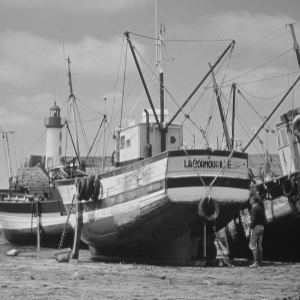}
 \includegraphics[width=0.19\textwidth]{cameraman_low_res.png}
 \includegraphics[width=0.19\textwidth]{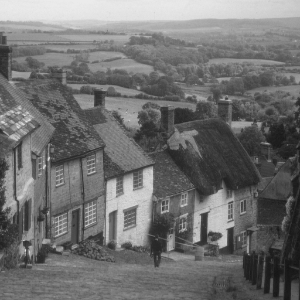}
 \includegraphics[width=0.19\textwidth]{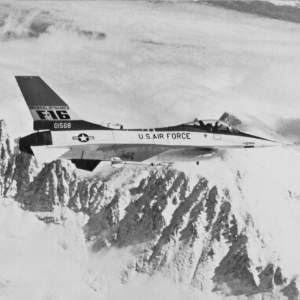}
 \includegraphics[width=0.19\textwidth]{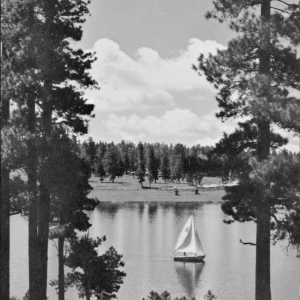}\\
 \includegraphics[width=0.19\textwidth]{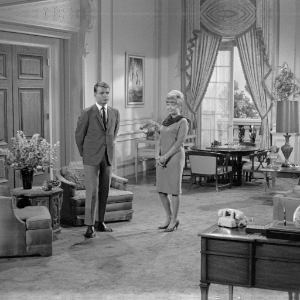}
 \includegraphics[width=0.19\textwidth]{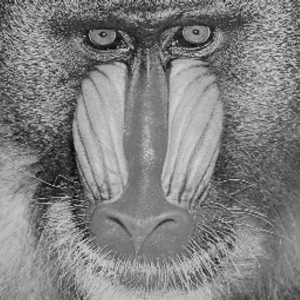}
 \includegraphics[width=0.19\textwidth]{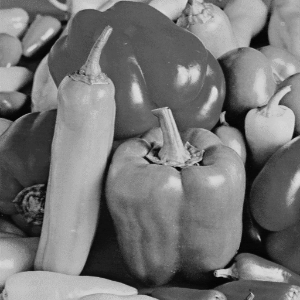}
 \includegraphics[width=0.19\textwidth]{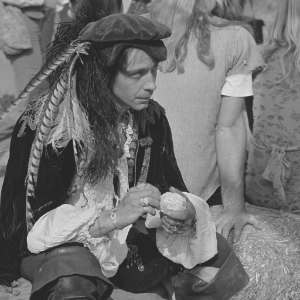}
 \includegraphics[width=0.19\textwidth]{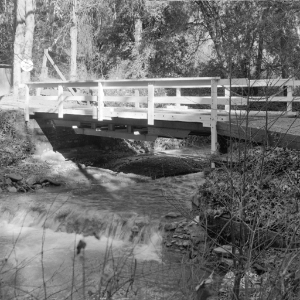}
 \caption{List of test images.}
 \label{fig:test-images}
 \end{figure}
 
 We compare five methods of denoising these images: one-term SVD and KoPA models, multi-term SVD and KoPA models, image denoising algorithm through total variation regularization \citep{chambolle2004algorithm}. Since determining the number of terms in multi-term models is beyond the scope of this article, the number of terms in the multi-term models are chosen to minimize the reconstruction error. The performance of the five approaches on the ten images are reported in Table~\ref{tab:reconstruction-errors}. 
 
 \begin{table}[!tb]
     \centering
     \begin{tabular}{c|ccccc}
     \hline
        image & SVD & KoPA & mSVD & mKoPA & TVR\\
        \hline
        boat & 0.4709 & 0.1757 & 0.0853 & 0.0613 & 0.0356\\
        cameraman & 0.5446 & 0.1337 & 0.0644 & 0.0399 & 0.0294\\
        goldhill & 0.4632& 0.1391 & 0.0759 & 0.0568 & 0.0363\\
        jetplane & 0.7347 & 0.1853 & 0.0866 & 0.0596 & 0.0302\\
        lake & 0.5425 & 0.1287 & 0.0825 & 0.0539 & 0.0308\\
        livingroom & 0.6747 & 0.2055 & 0.0995 & 0.0811 & 0.0589\\
        mandril & 0.6949 & 0.3557 & 0.1471 & 0.0889 & 0.0739\\
        peppers & 0.7394 & 0.1075 & 0.0734 & 0.0445 & 0.0224\\
        pirate & 0.7746 & 0.1533 & 0.1018 & 0.0686 & 0.0413\\
        walkbridge & 0.6617 & 0.2085 & 0.1263 & 0.0925 & 0.0593\\
        \hline
     \end{tabular}
     \caption{Reconstruction errors of one-term SVD, one-term KoPA, multi-term SVD(mSVD), multi-term KoPA(mKoPA) and total variation regularization (TVR) on the ten test images.}
     \label{tab:reconstruction-errors}
 \end{table}
 
 For each image, the configuration of the KoPA is selected by BIC ($\kappa=18\ln 2$).
 From Table~\ref{tab:reconstruction-errors}, the KoPA-based methods outperform SVD-based approaches, which is not surprising as SVD corresponds to a special configuration in KoPA models. On the other hand, the image denoising based on KoPA (and multi-term KoPA) is close to the TVR (total variation regularization) method but the latter does have a superior performance. 
 
 We note that KoPA and TVR are not directly comparable. Image is a special type of matrix data, whose entries usually possess certain continuity in values. TVR fully utilizes this continuity by imposing regularization on the total variation while SVD and KoPA do not. 
 The difference can be seen from Figure~\ref{fig:mandril_fitted} as well. The TVR can recover the smooth region (the mandrill's nose) well, while the multi-term KoPA model has more details in non-smooth regions (the mandrill's fur and beard). Finally we remark that the performance of KoPA approach on image analysis can possibly be improved by adding a similar penalty term on the smoothness of $\bm B$. 

\begin{figure}[htb]
    \centering
    \includegraphics[width=0.25\textwidth]{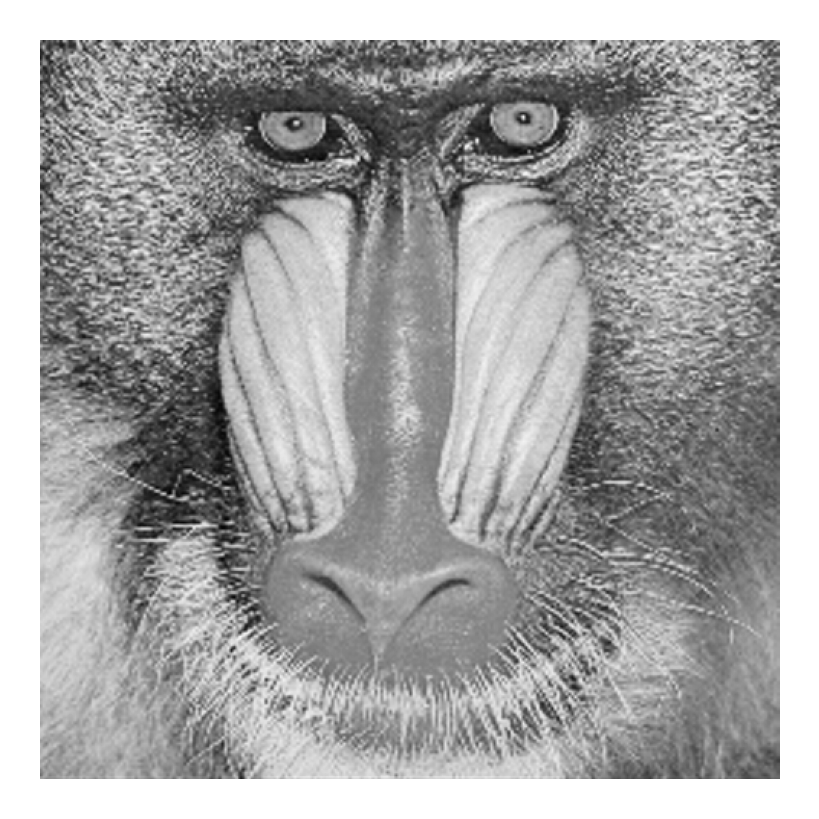}
    \includegraphics[width=0.25\textwidth]{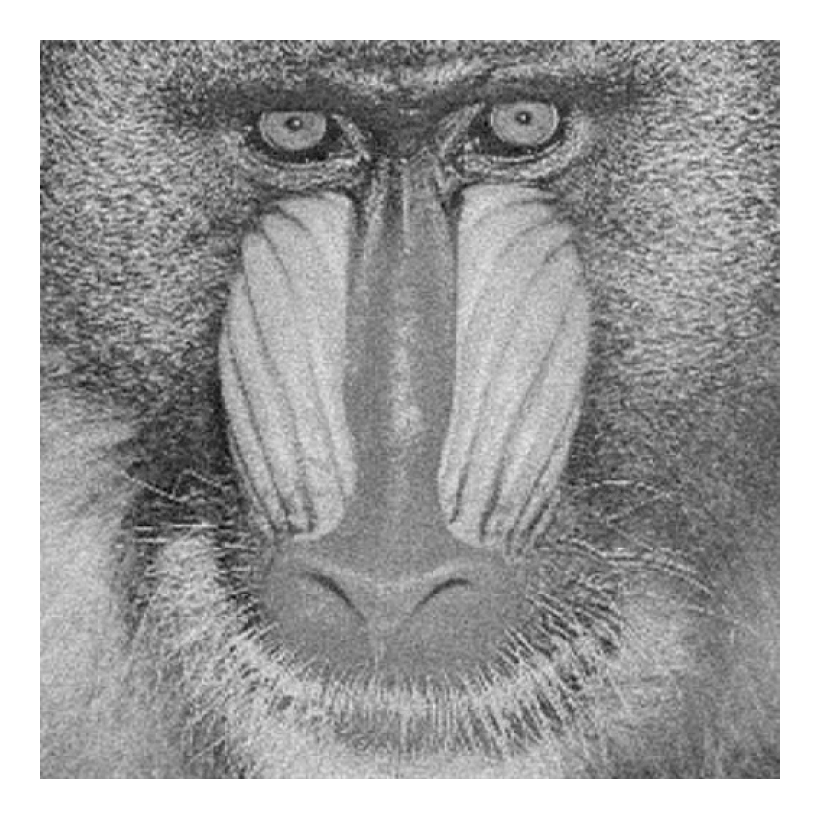}
    \includegraphics[width=0.25\textwidth]{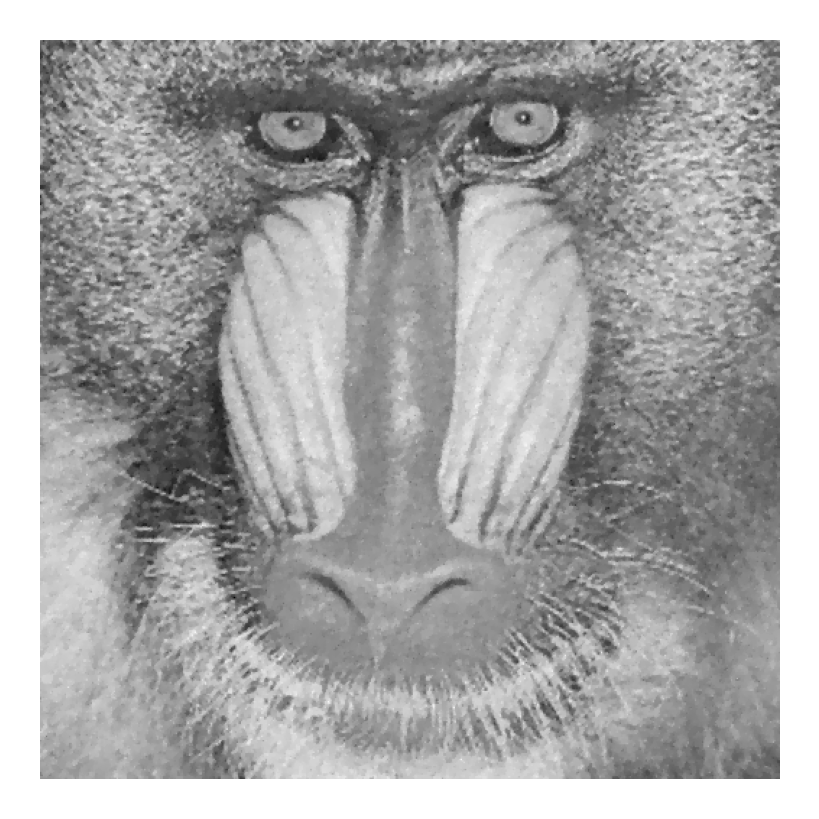}
    \caption{(left) The mandrill image, (mid) recovered images from multi-term KoPA model and (right) total variation regularization.}
    \label{fig:mandril_fitted}
\end{figure}
 

\section{Conclusion and Discussions}\label{sec:conclusion}

In this article, we propose to use the Kronecker product approximation as an
alternative of the low rank approximation of large
matrices. Comparing with the low rank approximation, KoPA is more
flexible because any configuration of the Kronecker product can
potentially be used, leading to different levels of approximation and
compression. To select the configuration, we propose to use the
extended information criterion, which includes MSE, AIC and BIC as
special cases. We establish the asymptotic
consistency of the configuration selection procedure, and uses an example with a random Kronecker product to illustrate how the technical assumptions are fulfilled. Extension to the multi-term Kronecker product model is also investigated. 
Both
simulations and ananysis of image examples demonstrate that KoPA can
be superior over the low rank approximations in the sense that it can give
a closer approximation of the original matrix/image with a higher
compression rate.

We conclude with a discussion of future directions. First of all, the Kronecker product model \eqref{eq:model-normalized} is not permutation-invariant. In other words, after a permutation of columns and rows, the signal from the matrix $\bm Y$ may or may not be a Kronecker product. When the columns and rows have an order in nature as in image data and in spatial-temporal data, it is not an issue. But in general, especially when the data is allowed to be shuffled, a pre-processing step for ordering rows and columns should be investigated before conducting KoPA analysis.
Another extension is to consider a multi-term model, where each term
can have its own configuration. This approach certainly allows a
greater flexibility, but also poses new challenges not only on the
configuration and order selections, but also on the estimation and
algorithms as well. 
It would be ideal if a natural
and interpretable procedure for the estimation, order determination,
and configuration selection and be developed, with theoretical guarantees.

\acks{Chen's research is supported
in part by National Science Foundation
grants DMS-1737857, IIS-1741390 and
CCF-1934924. Xiao's research is supported in part by National Science Foundation grant DMS-1454817, and a research grant from NEC Labs America. }

\bibliography{references.bib}
\newpage
\appendix 

\centerline{\bf \Large Appendix}

\bigskip

\section{Proof of Theorem \ref{thm:truth} and Corollary~\ref{corr:truth-random}}
Without loss of generality, we assume $\sigma = 1$. Noticing that
$$\hat\lambda = \|\mathcal R_{m_0, n_0}[\bm Y]\|_S = \|\lambda \vec(\bm A)\vec(\bm B)'+\sigma 2^{-(M+N)/2}\mathcal R_{m_0, n_0}[\bm E]\|_S,$$
by triangular inequality, we have
$$\left| \hat\lambda - \|\lambda \vec(\bm A)\vec(\bm B)'\|_S\right| \leqslant \sigma 2^{-(M+N)/2}\|\mathcal R_{m_0, n_0}[\bm E]\|_S,$$
where $\|\lambda \vec(\bm A)\vec(\bm B)'\|_S=\lambda$. The following bound for $\|\mathcal R_{m_0, n_0}[\bm E]\|_S$ can be obtained using the concentration inequality from \cite{Vershynin2010Introduction},
$$P(\|\mathcal R_{m_0, n_0}[\bm E]\|_S\geqslant 2^{(m_0+n_0)/2} + 2^{(M+N-m_0-n_0)/2} +t)\leqslant e^{-t^2/2}.$$
Therefore, $\|\mathcal R_{m_0, n_0}[\bm E]\|_S = s_0 + O_p(1)$ and 
$$|\hat\lambda - \lambda|\leqslant 2^{-(M+N)/2}(s_0 + O_p(1))=r_0+O_p(2^{-(M+N)/2}),$$
which yields $\hat\lambda - \lambda = O_p(r_0)$.\\
The bounds for $\hat{\bm A}$ and $\hat{\bm B}$ corresponds to the error bounds in estimating the left and right singular vectors of $\mathcal R_{m_0, n_0}[\bm Y]$, which is a direct consequence of the analysis in \cite{wedin1972perturbation} by observing that 
$$\|\hat{\bm A}-\bm A\|_F^2 = \|\vec(\hat{\bm A}) - \vec(\bm A)\|_2^2 = 2\sin^2\Theta(\vec(\hat{\bm A}), \vec(\bm A)).$$
A sharper bound is provided in \cite{cai2018rate}.

Since above analysis holds for any fixed value of $\lambda$, Corollary~\ref{corr:truth-random} follows immediately.

\section{Proof of Theorem \ref{thm:ic-gap-nonrandom}}
We first show and prove several technical lemmas.
\begin{lem}\label{lem:expect-log}
Suppose $a_n>0, a_n\rightarrow 0$ and $x_n=O_p(1)$ is a sequence of continuous random variables with density functions $p_n$ satisfying 
\begin{enumerate}[label=(\roman*)]
    \item $\mathbb E|x_n|\leqslant C$ for some constant $C$ for every $n$,
    \item $1+a_nx_n>0$ almost surely,
    \item $a_n^{-2}\sup_{x\leqslant -1/(2a_n)}p_n(x)\rightarrow 0$,
\end{enumerate}
then we have
$$\mathbb E\ln\left(1+a_nx_n\right) = O\left(a_n\right).$$
\end{lem}
\begin{proof}
Let $p_n(x_n)$ be the density function of $x_n$. For the positive part, we have
$$E_+ = \int_0^{+\infty}\ln(1+a_nt)p_n(t)dt\leqslant\int_0^{+\infty}a_ntp_n(t)dt\leqslant a_n\mathbb E|x_n|\leqslant Ca_n.$$
For the negative part, we have
\begin{align*}
    E_-&=\int_{-1/a_n}^{0}\ln(1+a_nt)p_n(t)dt\\
    &= \int_{-1/a_n}^{-1/(2a_n)}\ln(1+a_nt)p_n(t)dt+\int_{-1/(2a_n)}^{0}\ln(1+a_nt)p_n(t)dt\\
    &\geqslant \left[\sup_{t\leqslant -1/(2a_n)}p_n(t)\right]\int_{-1/a_n}^{-1/(2a_n)}\ln(1+a_nt)dt + \int_{-1/(2a_n)}^{0}2a_ntp_n(t)dt\\
    &\geqslant -\dfrac{1+\ln 2}{2a_n}\sup_{t<-1/(2a_n)}p_n(t) + 2a_n\int_{-\infty}^0 tp_n(t)dt\\
    &\geqslant o(a_n) - 2Ca_n.
\end{align*}
Hence, 
$$\mathbb E\ln(1+a_nx_n) = E_++E_-=O(a_n).$$
\end{proof}
The conditions in Lemma~\ref{lem:expect-log} are easy to verify in the subsequent proofs. Condition (ii) ensures the logarithm is well-defined on the whole support. Condition (i) is satisfied when $x_n$ converges in mean to a random variable $x$ with finite expectation. Condition (iii) is controlling the left tails of the densities, and is easily fulfilled if they are exponential. 

\begin{lem}\label{lem:spectral-norm-bound}
Let $\bm X$ be an arbitrary $P\times Q$ real matrix with $P\leqslant Q$ and $\bm E$ be a $P\times Q$ matrix with IID standard Gaussian entries. Then we have
$$\mathbb E\|\bm X+\bm E\|_S^2\leqslant \|\bm X\|_S^2 + (\sqrt{P}+\sqrt{Q})^2 + 4\|\bm X\|_S\sqrt{P} + \sqrt{2\pi}(\sqrt{P}+\sqrt{Q}) + 2=:U^2.$$
Furthermore, the departure from the expectation is sub-Gaussian such that for any positive $t$, we have
$$P[\|\bm X + \bm E\|_S\geqslant U+t]\leqslant e^{-t^2/2}.$$
\end{lem}
\begin{proof}
Without loss of generality, we assume $\bm X = [\bm X_1, \bm X_2]$, where $\bm X_1\in\mathbb R^{P\times P}$ is a diagonal matrix and $\bm X_2\in\mathbb R^{P\times(Q-P)}$ is zero. Such a form of $\bm X$ can always be achieved by multiplying $\bm X$ and $\bm E$ from left and right by orthogonal matrices, without changing the distribution of $\bm E$. Similarly, we partition $\bm E$ into $[\bm E_1, \bm E_2]$ with $\bm E_1\in\mathbb R^{P\times P}$ and $\bm E_2\in\mathbb R^{P\times (Q-P)}$. Then
\begin{align*}
    \|\bm X+\bm E\|_S^2 &=\sup_{u\in\mathbb R^{P}, \|u\|=1} u'(\bm X+\bm E)(\bm X+\bm E)'u\\
    &=\sup_{u\in\mathbb R^{P}, \|u\|=1} u'\bm X\bm X'u + u'\bm E\bm Eu + 2u'\bm X\bm E'u\\
    &\leqslant \|\bm X\|_S^2 + \|\bm E\|_S^2 + 2\|\bm X\|_S\|\bm E_1\|_S
\end{align*}
According to \cite{Vershynin2010Introduction}, we have $\mathbb E\|\bm E_1\|_S\leqslant 2\sqrt{P}$ and $$P[\|\bm E\|_S \geqslant \sqrt{P}+\sqrt{Q} + t]\leqslant e^{-t^2/2}.$$
Therefore, 
$$\mathbb E\|\bm E\|_S^2 = \int_{t=0}^\infty P[\|\bm E\|_S > t]2tdt \leqslant (\sqrt{P}+\sqrt{Q})^2 + \sqrt{2\pi}(\sqrt{P}+\sqrt{Q})+2.$$
Hence, we have
$$\mathbb E\|\bm X+\bm E\|_S^2\leqslant \|\bm X\|_S^2 + (\sqrt{P}+\sqrt{Q})^2 + 4\|\bm X\|_S\sqrt{P} + \sqrt{2\pi}(\sqrt{P}+\sqrt{Q}) + 2=:U^2.$$
Since for any fixed $\bm X$, $\|\bm X+\bm E\|_S$ is a function of $\bm E$ with Lipschitz norm 1, by concentration inequality, for any positive $t$, we have
$$P[\|\bm X +\bm E\|_S\geqslant U + t]\leqslant e^{-t^2/2}.$$
\end{proof}

We rewrite the information criterion as
$$\mathrm{IC}_\kappa(m, n) = D\left[\ln\|\bm Y - \hat{\bm Y}^{(m, n)}\|_F^2+\kappa r_{m, n}^2 - 2\kappa D^{-1/2}\right],$$
where $D=2^{M+N}$ and $r_{m, n} = 2^{-(m+n)/2}+2^{-(m^\dagger+n^\dagger)/2}$. The constant term $2\kappa D^{-1/2}$ is irrelevant to the configuration $(m, n)$ and is therefore ignored in subsequent proofs. Without loss of generality, we define the following expected information criterion
$$\mathrm{EIC}_\kappa(m, n) = D\left[\mathbb E\ln\|\bm Y - \hat{\bm Y}^{(m, n)}\|_F^2+\kappa r_{m, n}^2\right]$$
for simplicity. The difference in expected information criterion between wrong configurations and the true configuration is of central interest, so we define
$$\Delta \mathrm{EIC}_\kappa(m, n) = \mathrm{EIC}_\kappa(m, n) - \mathrm{EIC}_\kappa(m_0, n_0) $$

Under the true configuration $(m_0, n_0)$, we have
$$\mathbb E\|\bm Y - \hat{\bm Y}^{(m, n)}\|_F^2\leqslant \mathbb E\|\bm Y - \lambda \bm A\otimes\bm B\|_F^2=\sigma^2D^{-1}\mathbb E\|\bm E\|_F^2=\sigma^2.$$ Therefore, we have
\begin{equation}
    \mathrm{EIC}_\kappa (m_0, n_0) \leqslant D\left[\ln \mathbb E\|\bm Y - \hat{\bm Y}^{(m, n)}\|_F^2 + \kappa r_0^2\right]\leqslant D\left[\ln\sigma^2+\kappa r_0^2\right],\label{eq:proof-ic-true-config}
\end{equation}
where $r_0 = r_{m_0, n_0}$. 

Define
\begin{equation}
    \hat\lambda^{(m, n)}:=\|\mathcal R_{m, n}[\bm Y]\|_S = \|\lambda \mathcal R_{m, n}[\bm A\otimes \bm B] + \sigma D^{-1/2}\mathcal R_{m,n}[\bm E]\|_S. \label{eq:proof-lambda-hat}
\end{equation}
To calculate the information criterion for wrong configurations, we use the following equality
$$\|\bm Y -\hat{\bm Y}^{(m, n)}\|_F^2 = \|\bm Y\|_F^2 - \left[\hat\lambda^{(m, n)}\right]^2. $$
Notice that 
$$\|\bm Y\|_F^2 = \|\lambda\bm A\otimes \bm B\|_F^2 + \sigma^2D^{-1}\|\bm E\|_F^2 + 2 \lambda\sigma D^{-1/2}\tr[(\bm A\otimes \bm B)\bm E'],$$
where $\|\lambda\bm A\otimes \bm B\|_F^2 = \lambda^2$, $\sigma^2D^{-1}\|\bm E\|_F^2= \sigma^2(1+O_p(D^{-1/2}))$ and $\tr[(\bm A\otimes \bm B)\bm E']$ follows a standard normal distribution. We have
\begin{equation}
    \|\bm Y\|_F^2 = \lambda^2 + \sigma^2 + R_1,\label{eq:proof-y2}
\end{equation}
where 
$$R_1 = O_p\left((\sigma^2+\lambda\sigma)D^{-1/2}\right).$$

For wrong configurations $(m, n)\in\mathcal W$, without loss of generality, we assume $m+n\leqslant (M+N)/2$. 
According to Lemma~\ref{lem:spectral-norm-bound}, we have the upper bound for \eqref{eq:proof-lambda-hat}:
\begin{align}
[\hat\lambda^{(m,n)}]^2 &\leqslant \lambda^2\phi^2+\sigma^2r_{m,n}^2 + 4\lambda\phi\sigma 2^{(m+n)/2}D^{-1/2}+O_p((\lambda\sigma + \sigma^2)D^{-1/2})\nonumber\\
&\leqslant \lambda^2\phi^2+\sigma^2r_{m,n}^2 + 4\lambda\sigma D^{-1/4}+O_p((\lambda\sigma + \sigma^2)D^{-1/2}). \label{eq:proof-lambda-hat-bound}
\end{align}
Hence,
$$\|\bm Y - \hat{\bm Y}^{(m,n)}\|_S^2\geqslant \lambda^2(1-\phi^2) + \sigma^2(1-r_{m,n}^2) - 4\lambda\sigma D^{-1/4} + O_p((\lambda\sigma +\sigma^2)D^{-1/2}).$$
The last two terms are minor terms by Assumption~\ref{assump:gap}. Therefore,
\begin{equation}
    \mathrm{EIC}_\kappa(m,n)\geqslant D\left[\ln(\lambda^2\psi^2 + \sigma^2(1-r_{m,n}^2)) -  O\left(\dfrac{\lambda\sigma D^{-1/4}}{\sigma^2 + \lambda^2\psi^2}\right)+\kappa r_{m,n}^2\right].\label{eq:proof-eic-wrong}
\end{equation}
Here Lemma~\ref{lem:expect-log} is applied since the stochastic term in \eqref{eq:proof-lambda-hat-bound} has an exponential tail bound. \\
Notice that $\mathrm{EIC}_\kappa(m,n)$ in \eqref{eq:proof-eic-wrong} is either a monotone increasing function or a uni-modal function of $r_{m,n}^2$ on $[1/2, 4D^{1/2}]$. Therefore, the minimum of the right hand side of \eqref{eq:proof-eic-wrong} is obtained on the boundary. When $r_{m,n}^2 = 1/2$, \eqref{eq:proof-eic-wrong} becomes
\begin{equation}
    \mathrm{EIC}_\kappa(m,n)\geqslant D\left[\ln(\lambda^2\psi^2 + \sigma^2/2) -  O\left(\dfrac{\lambda\sigma D^{-1/4}}{\sigma^2 + \lambda^2\psi^2}\right)+\kappa/2\right].\label{eq:proof-eic-corner}
\end{equation}
When $r_{m,n}^2 = 4D^{-1/2}$, \eqref{eq:proof-eic-wrong} becomes
\begin{equation}
    \mathrm{EIC}_\kappa(m,n)\geqslant D\left[\ln(\lambda^2\psi^2 + \sigma^2) -  O\left(\dfrac{\lambda\sigma D^{-1/4}}{\sigma^2 + \lambda^2\psi^2}\right)\right].\label{eq:proof-eic-center}
\end{equation}
In conclusion, for any wrong configuration $(m, n)\in\mathcal W$, we have
\begin{equation}
    \Delta\mathrm{EIC}_\kappa(m,n)\geqslant D\left[\alpha - O\left(\dfrac{\lambda\sigma D^{-1/4}}{\sigma^2 + \lambda^2\psi^2}\right) - \kappa r_0^2,\right]\label{eq:proof-deic}
\end{equation}
where 
$$\alpha = \left[\ln\left(1+\dfrac{\lambda^2\psi^2}{\sigma^2}\right)\right]\wedge\left[\ln\left(\dfrac{1}{2}+\dfrac{ \lambda^2\psi^2}{\sigma^2}\right)+\dfrac{\kappa}{2}\right].$$
When $\kappa \geqslant 2\ln 2$, $\alpha$ takes the first value in the preceding equation. The assumptions imposed in Theorem~\ref{thm:ic-gap-nonrandom} ensure the leading term $\alpha$ in \eqref{eq:proof-deic} dominates other terms so that the minimum of $\Delta\mathrm{EIC}$ over the wrong configurations is strictly positive.

\medskip
We now address Remark~6. It turns out possible to use only the MSE to select the configuration, which corresponds to $\kappa=0$. It requires a stronger signal-to-noise ratio $\lambda^2\psi^2/\sigma^2>1/2$ so that the leading term $\alpha$ in \eqref{eq:proof-deic} is positive, and hence Theorem~\ref{thm:ic-gap-nonrandom} continues to hold. 

\medskip
\stepcounter{remark}
\noindent {\bf Remark \arabic{remark}.} 
Note that the upper bound used in \eqref{eq:proof-lambda-hat-bound} is quite conservative, because the maximums of $\phi$ and $2^{(m+n)/2}$ over $\mathcal W$ are taken separately. It leads to a simple form of Assumption~\ref{assump:gap}, which is actually not as optimal as possible. 
If we define $\phi^{(m,n)} = \|\mathcal R_{m,n}[\bm A\otimes \bm B]\|_S$, 
then the condition \eqref{eq:gap2} in Assumption~\ref{assump:gap} can be relaxed to 
$$\lim_{\asymptotic}\inf_{(m,n)\in\mathcal W} (2^{(m+n)/2}+2^{(m^\dagger+n^\dagger)/2})\cdot \dfrac{\lambda}{\sigma}\cdot \dfrac{1-[\phi^{(m,n)}]^2}{\phi^{(m,n)}}=\infty.$$
However, in the main text we choose to introduce the concept of representation gap and present a simple version of Assumption~\ref{assump:gap}.

\section{Proof of Theorem~\ref{thm:consistency_nonrandom}}
We begin with the tail bounds for $\|\bm E\|_F^2$. 
According to the tail bounds for $\chi^2$ random variable given in \cite{Laurent2000adaptive}, it holds that for any $t > 0$,
\begin{align}
    P\left[D^{-1}\|\bm E\|_F^2 > 1 + \sqrt{2}D^{-1/2}t + D^{-1}t^2\right]&\leqslant e^{-t^2/2},\label{eq:proof-E2-upper-tail}\\
    P\left[D^{-1}\|\bm E\|_F^2 < 1 - \sqrt{2}D^{-1/2}t\right]&\leqslant e^{-t^2/2},\label{eq:proof-E2-lower-tail}
\end{align}
where $D=2^{M+N}$. Therefore, at the true configuration $(m_0, n_0)$, we have
\begin{align}
    &P\left[\|\bm Y - \hat{\bm Y}^{(m_0, n_0)}\|_F^2 > \sigma^2 + \sqrt{2}\sigma^2D^{-1/2}t+ \sigma^2 D^{-1}t^2\right]\nonumber\\
    \leqslant & P\left[\|\sigma D^{-1/2}\bm E\|_F^2 > \sigma^2 + \sqrt{2}\sigma^2D^{-1/2}t+ \sigma^2D^{-1} t^2\right]\nonumber\\
    \leqslant & e^{-t^2/2}.\label{eq:proof-True-Conf-upper-tail}
\end{align}
Noticing that
$$\|\bm Y\|_F^2 = \lambda^2 + \sigma^2D^{-1}\|\bm E\|_F^2 + 2\lambda\sigma D^{-1/2}Z,$$
where $Z = \tr[(\bm A\otimes \bm B)\bm E']$ is a standard Gaussian random variable,
by \eqref{eq:proof-E2-lower-tail} we have 
\begin{equation}
    P\left[\|\bm Y\|_F^2 < \lambda^2 + \sigma^2 - (\sqrt{2}\sigma^2 + 2\lambda\sigma)D^{-1/2}t\right]\leqslant 2e^{-t^2/2}.\label{eq:proof-Y2-lower-tail}
\end{equation}

Now we consider the tail bound for $\hat\lambda^{(m, n)}$ of wrong configurations. According to Lemma~\ref{lem:spectral-norm-bound}, we have the tail bound for $\hat\lambda^{(m,n)}$ as
\begin{equation}
P[\hat\lambda^{(m,n)}\geqslant U + \sigma D^{-1/2}t]\leqslant e^{-t^2/2},\label{eq:proof-lambdahat-tail}
\end{equation}
where
$$U^2 = \lambda^2\phi^2+\sigma^2r_{m,n}^2 +4\lambda\phi\sigma 2^{(m+n)/2}D^{-1/2} + \sqrt{2\pi} \sigma^2r_{m,n}D^{-1/2} +2\sigma^2D^{-1}<(\lambda+\sigma)^2.$$
Let $\alpha=\ln(1+(\lambda/\sigma)^2\psi^2)$ be the positive gap constant. We have
\begin{align}
    &P\left[\mathrm{IC}_\kappa (m_0, n_0) > \mathrm{EIC}_\kappa(m_0, n_0) +  D\alpha / 3\right]\nonumber\\
    =& P\left[\|\bm Y - \hat{\bm Y}^{(m_0, n_0)}\|_F^2 > \sigma^2 e^{\alpha/3}\right]\nonumber\\
    \leqslant & \exp\left(-c_1^2D/2\right),\label{eq:proof-prob-ic-true}
\end{align}
where 
$$c_1^2 = e^{\alpha/3} - 1.$$
For any $(m, n)\in\mathcal W$, it holds that
\begin{align}
    & P\left[\mathrm{IC}_\kappa (m, n) < \mathrm{EIC}_\kappa(m_0, n_0) +  D\alpha / 3 \right]\nonumber\\
    \leqslant & P\left[\mathrm{IC}_\kappa (m, n) < \mathrm{EIC}_\kappa(m, n) -  D\alpha / 3 \right]\nonumber\\
    \leqslant & P\left[\|\bm Y\|_F^2 - \hat\lambda^2 < \lambda^2 + \sigma^2 - \lambda^2\phi^2 - 2h\right]\nonumber\\
    \leqslant & P\left[\|\bm Y\|_F^2 < \lambda^2 + \sigma^2 - h\right] + P\left[\hat\lambda^2 > U^2 + h\right]\nonumber\\
    \leqslant & 2\exp\left(-c_2^2D/2\right) + \exp\left(-c_3^2D/2\right)\label{eq:proof-prob-ic-W}
\end{align}
where we use \eqref{eq:proof-Y2-lower-tail} and \eqref{eq:proof-lambdahat-tail} to obtain \eqref{eq:proof-prob-ic-W},
$$h = \dfrac{1}{2}\left(1-e^{-\alpha/3}\right)(\lambda^2(1-\phi^2) +\sigma^2),\quad c_2 = \dfrac{h}{\sqrt{2}\sigma^2 + 2\lambda\sigma},$$
and $c_3$ is the solution of 
$$\sigma^2c_3^2 + 2(\lambda+\sigma)\sigma c_3 = h.$$

We conclude that
\begin{align}
    &P\left[\mathrm{IC}_\kappa(m_0, n_0) \geqslant \min_{(m, n)\in\mathcal W}\mathrm{IC}_\kappa(m, n)\right]\nonumber\\
    \leqslant &\sum_{(m, n)\in\mathcal W}P[\mathrm{IC}_\kappa(m_0, n_0) \geqslant \mathrm{IC}_\kappa(m, n)]\nonumber\\
    \leqslant & \sum_{(m, n)\in\mathcal W}\bigg(P[\mathrm{IC}_\kappa(m_0, n_0) \geqslant \mathrm{EIC}_\kappa(m_0, n_0) +  D\alpha / 3]\nonumber\\
    &\hspace{4em} +P[\mathrm{IC}_\kappa(m, n) \leqslant \mathrm{EIC}_\kappa(m_0, n_0) +  D\alpha / 3]\bigg)\nonumber\\
    \leqslant & 4(M+1)(N+1)\exp\left[-c^2D/2\right] \rightarrow 0,\label{eq: proof-prob-final}
\end{align}
where $c = \min\{c_1, c_2, c_3\}$. By calculating the orders of $c_1$, $c_2$, $c_3$, it holds that
$$c^2\geqslant O\left((e^{\alpha/3}-1)\wedge \left(\dfrac{e^\alpha - e^{2\alpha/3}}{1+\lambda/\sigma}\right)^2\right).$$
Specifically, if $\alpha\rightarrow 0$ (or equivalently, $(\lambda/\sigma)^2\psi^2\rightarrow 0$), we have
$$c^2\geqslant O\left(\dfrac{\lambda^2}{\sigma^2}\psi^2\wedge \dfrac{(\lambda^2/\sigma^2)^2}{(1+\lambda/\sigma)^2}\psi^4\right)$$
The right hand side is much greater than $\ln (MN)$, under Assumptions~\ref{assump:hd} and \ref{assump:gap}.

\section{Proof of Theorem~\ref{thm:ic-gap}}
The proof is very similar to the proofs of  Theorem~\ref{thm:ic-gap-nonrandom} and Theorem~\ref{thm:consistency_nonrandom}, so we only point out the major steps, but omit the details. Condition \eqref{eq:gap-random-tail-cond} implies that $\lambda^2 = \lambda_0^2(1+o_p(1))$ and $\psi^2 = \psi_0^2(1+o_p(1))$. The proof of Theorem~\ref{thm:ic-gap-nonrandom} follows immediately by replacing $\lambda^2$ and $\psi^2$ with the deterministic values $\lambda_0^2$ and $\psi_0^2$, except that an $o_p(\lambda_0^2+\psi_0^2)$ term is added to \eqref{eq:proof-lambda-hat}. Since the additional stochastic term is negligible and has finite expectation,  Theorem~\ref{thm:ic-gap-nonrandom} continues to hold.

The consistency follows same lines as those of Theorem~\ref{thm:consistency_nonrandom} except that the deviations $\lambda^2-\lambda_0^2$ and $\psi^2-\psi_0^2$ should be incorporated into \eqref{eq: proof-prob-final}. Specifically, Assumption~\ref{assump:gap_random} implies that for any small constant $\delta$, with probability larger than $1-o(1/(MN))$, we have $\lambda^2 \geqslant \lambda_0^2(1-\delta)$ and $\psi^2\geqslant \psi_0^2(1-\delta)$. Proof of Theorem~\ref{thm:consistency_nonrandom} follows immediately by replacing $\lambda^2$ and $\psi^2$ with $\lambda_0^2(1-\delta)$ and $\psi_0^2(1-\delta)$. The following probability of exceptions should be added to \eqref{eq: proof-prob-final}. 
$$(M+1)(N+1)\left[P[\lambda^2<\lambda_0^2(1-\delta)] + P[\psi^2<\psi_0^2(1-\delta)]\right]=o(1),$$
which does not affect consistency but may reduce the convergence rate. 


\section{Proof of Lemma \ref{lem:further-decomposition} and Corollary~\ref{corr:further-decomp-random}}
Consider the complete Kronecker product decomposition of $\bm A$ with respect to the configuration $(m\wedge m', n\wedge n',(m-m')_+,(n-n')_+)$:
\begin{equation}
   \bm A = \sum_{i=1}^I\mu_i \bm C_i\otimes \bm D_i, \label{eq: decomp_a}
\end{equation}
where $I = 2^{m\wedge m' + n\wedge n'}\wedge 2^{(m-m')_++(n-n')_+}$, $\mu_1\geqslant \mu_2\geqslant \dots \geqslant \mu_I $ are the coefficients in decreasing order. $\bm C_i$ and $\bm D_i$ satisfy
\begin{equation}
\langle \bm C_i, \bm C_j\rangle = \langle \bm D_i, \bm D_j\rangle = \delta_{i,j },\label{eq: orthonormal_cond_CD}
\end{equation}
where $\delta_{i, j}$ is the Kronecker delta function such that $\delta_{i, j}=1$ if and only if $i=j$ and $\delta_{i, j} = 0$ otherwise, and $\langle \bm A, \bm B\rangle := \tr[\bm A'\bm B]$ is the trace inner product. Notice that the decomposition in (\ref{eq: decomp_a}) corresponds to the singular value decomposition for $\mathcal R_{m\wedge m', n\wedge n'}[\bm A]$. Therefore, the singular values $\mu_1,\dots, \mu_I$ are uniquely identifiable and the components $\bm C_i$, $\bm D_i$ are identifiable if the singular values are distinct. In particular, 
$$\mu_1 = \|\mathcal R_{m\wedge m', n\wedge n'}[\bm A]\|_S.$$
Similarly, the KPD of $\bm B$ with the configuration $((m'-m)_+, (n'-n)_+,M-m\vee m',N-n\vee n')$ is given by
$$\bm B = \sum_{j=1}^J\nu_j \bm F_j\otimes \bm G_j,$$
where $J = 2^{(m'-m)_++(n'-n)_+}\wedge 2^{M+N-m\vee m'-n\vee n'}$ and 
$$\nu_1 = \|\mathcal R_{(m'-m)_+, (n'-n)_+}[\bm B]\|_S.$$
With the two KPD of $\bm A$ and $\bm B$, we can rewrite $\bm A\otimes \bm B$ as
$$\bm A\otimes \bm B = \left(\sum_{i=1}^I\mu_i\bm C_i\otimes \bm D_i\right)\otimes \left(\sum_{j=1}^J\nu_j\bm F_j\otimes \bm G_j\right) = \sum_{i=1}^I\sum_{j=1}^J \mu_i\nu_j\bm C_i\otimes \bm D_i\otimes \bm F_j\otimes \bm G_j.$$
Notice that the Kronecker product satisfies distributive law and associative law.
The matrix $\bm D_i$ is $2^{(m-m')_+}\times 2^{(n-n')_+}$ and the matrix $\bm F_j$ is $2^{(m'-m)_+}\times 2^{(n'-n)_+}$. For all possible values of $m, m', n, n'$, either one of $\bm D_i$ and $\bm F_j$ is a scalar, or they are both vectors; and for both cases $\bm D_i\otimes \bm F_j = \bm F_j\otimes \bm D_i$. Therefore, 
\begin{equation}
\bm A\otimes \bm B = \sum_{i=1}^I\sum_{j=1}^J \mu_i\nu_j\bm C_i\otimes \bm F_j\otimes \bm D_i\otimes \bm G_j = \sum_{i=1}^I\sum_{j=1}^J \mu_i\nu_j \bm P_{ij}\otimes \bm Q_{ij},\label{eq: decompose_ab}
\end{equation}
where 
$$\bm P_{ij} := \bm C_i\otimes \bm F_j,\quad \bm Q_{ij}:=\bm D_i\otimes \bm G_j.$$
Notice that $\bm P_{ij}$ is a $2^{m'}\times 2^{n'}$ matrix and $\bm Q_{ij}$ is a $2^{M-m'}\times 2^{N-n'}$ matrix.  Therefore, (\ref{eq: decompose_ab}) is a KPD of $\bm A \otimes \bm B$ indexed by $(i, j)$ with respect to the Kronecker configuration $(m',n',M-m',N-n')$ as long as $\bm P_{ij}$ and $\bm Q_{ij}$ satisfy the orthonormal condition in (\ref{eq: orthonormal_cond_CD}). In fact, 
\begin{align*}
    \langle \bm P_{ij}, \bm P_{kl}\rangle &= \tr[\bm P_{ij}'\bm P_{kl}]\\
    &=\tr[(\bm C_i\otimes \bm F_j)'(\bm D_k\otimes \bm G_l)]\\
    &=\tr[(\bm C_i'\bm D_k)\otimes (\bm F_j'\bm G_l)]\\
    &=\tr[\bm C_i'\bm D_k] \tr[\bm F_j'\bm G_l]\\
    &=\delta_{i,j}\delta_{k,l},
\end{align*}
and similar results hold for $\bm Q_{ij}$. It follows that
$$\|\mathcal R_{m', n'}[\bm A\otimes \bm B]\|_S = \max_{i,j} \mu_i\nu_j = \mu_1\nu_1 = \|\mathcal R_{m\wedge m', n\wedge n'}[\bm A]\|_S\cdot\|\mathcal R_{(m'-m)_+, (n'-n)_+}[\bm B]\|_S,$$
and the proof of Lemma~\ref{lem:further-decomposition} is complete.

\medskip
Now we consider Corollary~\ref{corr:further-decomp-random}.
When $\bm A$ and $\bm B$ are generated as in Example~\ref{example:normal}, we have
\begin{align*}
    \|\mathcal R_{m\wedge m', n\wedge n'}[\tilde{\bm A}]\|_S &\leqslant 2^{(m\wedge m'+n\wedge n')/2} + 2^{((m-m')_++(n-n')_+)/2} + O_p(1),\\
    \|\mathcal R_{(m'-m)_+, (n'-n)_+}[\tilde{\bm B}]\|_S&\leqslant 2^{((m'-m)_++ (n'-n)_+)/2} + 2^{(M+N-m\vee m'-n\vee n')/2}+O_p(1),\\
    \|\tilde{\bm A}\|_F\|\tilde{\bm B}\|_F &=2^{(M+N)/2}(1 + O_p(r_0)).
\end{align*}
Hence,
\begin{multline*}
    \|\mathcal R_{m', n'}[\bm A\otimes \bm B]\|_S=\dfrac{\|\mathcal R_{m', n'}[\tilde{\bm A}\otimes \tilde{\bm B}]\|_S}{\|\tilde{\bm A}\|_F\|\tilde{\bm B}\|_F}\leqslant 2^{-(m'+n')/2} + 2^{-(M+N-m'-n')/2}\\
    +2^{-(|m-m'|+|n-n'|)/2} + 2^{-(M+N-|m-m'|-|n-n'|)/2} + o_p(1).
\end{multline*}
The maximum of the right hand side is obtained when $|m-m'|+|n-n'|=1$, or $m'+n'\in\{1, M+N-1\}$, for which
$$\|\mathcal R_{m', n'}[\bm A\otimes \bm B]\|_S\leqslant 1/\sqrt{2} + o_p(1).$$
Furthermore, it is straightforward to verify that the upper bound is attained when $m'+n'\in\{1, M+N-1\}$, which leads to Corollary~\ref{corr:further-decomp-random}.

\section{Proof of Lemma \ref{lem:norm-of-sum}}
We first prove the following technical lemma.

\begin{lem}
Let $U$, $V$ be two vector subspaces of $\mathbb R^n$ with $\Theta(U, V) = \theta\in [0, \pi/2]$, where $\Theta(U,V)$ denotes the smallest principal angle between $U$ and $V$.
Suppose $w\in\mathbb R^n$ is a unit vector and 
$$\|P_Uw\|=\cos\alpha,$$
for some $\alpha\in[0,\pi/2]$, where $P_U$ denotes the orthogonal projection to the space $U$. Then it holds that
$$\|P_Vw\|\leqslant \begin{cases}
\cos(\theta-\alpha) & \text{if }\alpha \leqslant \theta, \\
1 & \text{if } \alpha > \theta.
\end{cases}$$
\end{lem}
\begin{proof}
Let 
$$u = \dfrac{P_Uw}{\|P_Uw\|},$$
then $\|u\|=1$ and $u\in U$. Let $\{u_1, u_2,\dots, u_n\}$ be an orthogonal basis of $\mathbb R^n$ such that $u_1=u$. For any vector $v \in V$, we have
\begin{align*}
    v'w &= v'\left(\sum_{i=1}^nu_iu_i'\right)w\\
    &=  v'u_1u_1'w +\sum_{i=2}^nv'u_iu_i'w\\
    &\leqslant v'u_1u_1'w +\sqrt{\sum_{i=2}^nv'u_i}\sqrt{\sum_{i=2}^nu_i'w}\\
    &= \cos \eta\cos\alpha + \sin \eta\sin\alpha\\
    &= \cos(\eta-\alpha),
\end{align*}
where $v'u_1=\cos\eta$.  The proof is complete by noting that $\cos\eta =v'u_1\leqslant \cos\theta$.
\end{proof}

We now prove Lemma \ref{lem:norm-of-sum}.

\noindent {\it Proof of Lemma \ref{lem:norm-of-sum}.}
Recall that $\bm M_1$ and $\bm M_2$ are of the same dimension. We consider the maximization of $\|(\bm M_1+\bm M_2)u\|^2$ over all unit vectors $u$. First write
\begin{align*}
    \|(\bm M_1+\bm M_2)u\|^2 &= \|\bm M_1u + \bm M_2u\|^2\\
    &= \|\bm M_1P_{\bm M_1'}u + \bm M_2P_{\bm M_2'}u\|^2\\
    &= \|\bm M_1P_{\bm M_1'}u\|^2 + \|\bm M_2P_{\bm M_2'}u\|^2 + 2(\bm M_1P_{\bm M_1'}u)'\bm M_2P_{\bm M_2'}u,
\end{align*}
where $P_{\bm M}$ denotes the projection matrix to the column space of $\bm M$.
Since $\|\bm M_1\|_S=\mu$ and $\|\bm M_2\|_S=\nu$, we have
$$\|\bm M_1P_{\bm M_1'}u\|^2\leqslant \mu^2\|P_{\bm M_1'}u\|^2\quad\text{and}\quad \|\bm M_2P_{\bm M_2'}u\|^2\leqslant \nu^2 \|P_{\bm M_2'}u\|^2.$$
Since $\bm M_1P_{\bm M_1'}u\in \mathrm{span}(\bm M_1)$ and $\bm M_2P_{\bm M_2'}u\in \mathrm{span}(\bm M_2)$, it holds that
$$(\bm M_1P_{\bm M_1'}u)'\bm M_2P_{\bm M_2'}u\leqslant \cos\theta \mu\nu\|P_{\bm M_1'}u\|\|P_{\bm M_2'}u\|.$$
It follows that
$$\|(\bm M_1+\bm M_2)u\|^2\leqslant \mu^2\|P_{\bm M_1'}u\|^2 + \nu^2 \|P_{\bm M_2'}u\|^2 + 2\mu\nu\|P_{\bm M_1'}u\|\|P_{\bm M_2'}u\|\cos\theta.$$
Suppose $\|P_{\bm M_1'}u\| = \cos\alpha$ for some $\alpha\in[0,\pi/2]$. If $\alpha > \eta$, then $\|P_{\bm M_2'}u\|\leqslant 1$. The right hand side of the preceding inequality attains its maximum when $\|P_{\bm M_1'}u\|=\cos\eta$ and $\|P_{\bm M_2'}u\|=1$. Hence, we only consider the case $\alpha \leqslant \eta$, which implies that $\|P_{\bm M_2'}u\|\leqslant \cos(\eta-\alpha)$, and
$$\|(\bm M_1+\bm M_2)u\|^2\leqslant \mu^2\cos^2\alpha + \nu^2\cos^2(\eta-\alpha) + 2\mu\nu\cos\theta\cos\alpha\cos(\eta-\alpha).$$
Therefore,
\begin{align*}
    &\mu^2\cos^2\alpha + \nu^2\cos^2(\eta-\alpha) + 2\mu\nu\cos\theta\cos\alpha\cos(\eta-\alpha)\\
    =&\dfrac{1}{2}\mu^2(1+\cos 2\alpha) + \dfrac{1}{2}\nu^2(1+\cos (2\eta-2\alpha)) +\mu\nu\cos\theta[\cos\eta +\cos(\eta-2\alpha)]\\
    =&\dfrac{1}{2}(\mu^2+\nu^2+2\mu\nu\cos\theta\cos\eta)\\
    &\hfill + \left(\dfrac{1}{2}\mu^2 +\dfrac{1}{2}\nu^2\cos(2\eta)+\mu\nu\cos\theta\cos\eta\right)\cos(2\alpha) + \left(\dfrac{1}{2}\nu^2\sin(2\eta)+\mu\nu\cos\theta\sin\eta\right)\sin(2\alpha)\\
    \leqslant&\dfrac{1}{2}(\mu^2+\nu^2+2\mu\nu\cos\theta\cos\eta)\\
    &\hfil + \sqrt{\left(\dfrac{1}{2}\mu^2 +\dfrac{1}{2}\nu^2\cos(2\eta)+\mu\nu\cos\theta\cos\eta\right)^2+\left(\dfrac{1}{2}\nu^2\sin(2\eta)+\mu\nu\cos\theta\sin\eta\right)^2}\\
    =&\dfrac{1}{2}\left(\mu^2+\nu^2 + 2\mu\nu\cos\theta\cos\eta+\sqrt{\left(\mu^2+\nu^2 + 2\mu\nu\cos\theta\cos\eta\right)^2-4\mu^2\nu^2\sin^2\theta\sin^2\eta}\right).
\end{align*}
The proof is complete.

\section{Proofs of Theorem \ref{thm:gap-two-term} and Corollary \ref{corr:two-term-random-scheme}}
\label{sec:appendix7}
The proof of Theorem \ref{thm:gap-two-term} is similar to the proofs of Theorem~\ref{thm:ic-gap-nonrandom} and Theorem~\ref{thm:consistency_nonrandom}, so we only point out the main steps here and omit the details.

Following the same argument as in the proof of Theorem~\ref{thm:ic-gap-nonrandom}, the expected information criteria of the true configuration is
$$\mathrm{EIC}_\kappa(m_0, n_0) = D\left[\ln\left(\lambda_2^2+\sigma^2\right) + \kappa r_0^2\right].$$
For a wrong configuration $(m,n)\in\mathcal W$, $\hat\lambda^{(m,n)}$ is obtained by
$$\hat\lambda^{(m,n)} = \|\lambda_1\mathcal R[\bm A_1\otimes \bm B_1] +\lambda_2\mathcal R[\bm A_2\otimes \bm B_2] + \sigma D^{-1/2}\mathcal R[\bm E]\|_S.$$
According to Lemma~\ref{lem:norm-of-sum} and Assumption~\ref{assump:Theta}, we have
\begin{equation}
\|\lambda_1\mathcal R[\bm A_1\otimes \bm B_1] +\lambda_2\mathcal R[\bm A_2\otimes \bm B_2]\|_S^2\leqslant \lambda_1^2\phi_1^2+\lambda_2^2\phi_2^2+2\lambda_1\lambda_2\phi_1\phi_2\xi<(\lambda_1+\lambda_2)^2. \label{eq:proof-two-term-signal}
\end{equation}
By Lemma~\ref{lem:spectral-norm-bound}, we have
\begin{multline}
    [\hat\lambda^{(m,n)}]^2\leqslant \lambda_1^2\phi_1^2 + \lambda_2^2\phi_2^2 + 2\lambda_1\lambda_2\phi_1\phi_2\xi +\sigma^2r_{m,n}^2\\ + O((\lambda_1+\lambda_2)\sigma D^{-1/4}) + O_p\left((\lambda_1+\lambda_2+\sigma)\sigma D^{-1/2}\right).\label{eq:proof-two-term-lambda-hat}
\end{multline}
With \eqref{eq:proof-two-term-lambda-hat} replacing \eqref{eq:proof-lambda-hat-bound}, the rest of the proof follows the same line of the proof of Theorem~\ref{thm:ic-gap-nonrandom}.\\
The proof of consistency is same as in the proof of Theorem~\ref{thm:consistency_nonrandom} except that the formula of $\hat\lambda^{(m,n)}$ in \eqref{eq:proof-two-term-lambda-hat} is used in \eqref{eq:proof-prob-ic-W}.

\medskip
We now prove Corollary~\ref{corr:two-term-random-scheme}.
When model \eqref{eq:model-two-term} is generated under the random scheme in Example~\ref{example:two-term-normal}, we only consider the wrong configuration close to the true configuration. It can be verified that the separation $\Delta \mathrm{EIC}(m,n)$ is larger at other configurations. Consider $(m,n)$ such that $|m_0-m|+|n_0-n|=1$. Then from Corollary~\ref{corr:further-decomp-random}, we have
$$\phi_1 = \dfrac{1}{\sqrt{2}} + O_p(r_0),\quad \phi_2 = \dfrac{1}{\sqrt{2}} + O_p(r_0).$$
Now consider the principle angles between $\mathcal R[\bm A_1\otimes \bm B_1]$ and $\mathcal R[\bm A_2\otimes\bm B_2$ as in Lemma~\ref{lem:norm-of-sum}, We have
$$\cos\theta = O_p(2^{-(m+n)}),\quad \cos\eta = O_p(2^{-(m^\dagger + n^\dagger)}).$$
By Lemma~\ref{lem:norm-of-sum}, \eqref{eq:proof-two-term-signal} can be revised to
$$\|\lambda_1\mathcal R[\bm A_1\otimes \bm B_1] +\lambda_2\mathcal R[\bm A_2\otimes \bm B_2]\|_S^2\leqslant \dfrac{\lambda_1^2}{2}+O_p(\lambda_1^2r_0).$$
Corollary~\ref{corr:two-term-random-scheme} follows immediately.
\end{document}